\documentclass[reqno]{amsart}
\usepackage[foot]{amsaddr}

\usepackage[top=30mm,bottom=30mm,inner=30mm,outer=35mm]{geometry}

\pdfoutput=1

\usepackage{booktabs}
\usepackage{rotating}
\usepackage{amssymb}
\usepackage[british]{babel}
\usepackage{color}
\usepackage{enumitem}
\usepackage[T1]{fontenc}
\usepackage[utf8]{inputenc}
\usepackage[l2tabu,orthodox]{nag}
\usepackage{tensor}
\usepackage[vcentermath]{youngtab}
\usepackage[boxsize=1em,centertableaux]{ytableau}
\usepackage{tipa}
\usepackage{chemfig}

\usepackage{nicefrac}

\usepackage{todonotes}

\usepackage[
		pdftex,
		bookmarks,
		bookmarksopen,
		bookmarksdepth=3,
		bookmarksopenlevel=1,
		bookmarksnumbered,
		breaklinks,
		colorlinks,
		linkcolor=linkcolor,
		citecolor=citecolor,
		urlcolor=urlcolor,
	]{hyperref}

\hypersetup{
	pdftitle    = {Algebraic Conditions for Conformal Superintegrability in 
	Arbitrary Dimension},
	pdfauthor   = {Jonathan Kress, Konrad Schöbel, Andreas Vollmer},
	pdfsubject  = {article},
	pdfcreator  = {\LaTeX{} with `amsart' class},
	pdfproducer = {pdf\LaTeX},
	pdfkeywords = {superintegrable systems}
}

\definecolor{linkcolor}{rgb}{0.5,0.0,0.0}
\definecolor{citecolor}{rgb}{0.0,0.5,0.0}
\definecolor{urlcolor} {rgb}{0.0,0.0,0.5}

\title[Conformal Superintegrability in Arbitrary Dimension]
	{Algebraic Conditions for\\ Conformal Superintegrability\\ in Arbitrary 
	Dimension}
\subjclass[2010]{
	Primary
	14H70;  
	Secondary
	53C18,  
	70H06,  
	70H33,  
	35N10.  
}

\author{Jonathan Kress$^\sharp$}
\author{Konrad Schöbel$^*$}
\author{Andreas Vollmer$^{\sharp\flat\mathsection}$}

\email{j.kress@unsw.edu.au}
\email{konrad.schoebel@htwk-leipzig.de}
\email{andreas.vollmer@polito.it} 
\email{andreas.d.vollmer@gmail.com}

\address[$\sharp$]{%
	School of Mathematics and Statistics \\
	University of New South Wales \\
	Sydney 2052 \\
	Australia
}

\address[$^*$]{%
	Faculty for Digital Transformation \\
	HTWK Leipzig University of Applied Sciences \\
	04251 Leipzig \\
	Germany
}

\address[$^\flat$]{%
	Institute of Geometry and Topology \\
	University of Stuttgart \\
	70049 Stuttgart \\
	Germany
}

\address[$^\mathsection$]{%
	Dipartimento di Scienze Matematiche
	"Giuseppe Luigi Lagrange" \\
	Politecnico di Torino \\
	Corso Duca degli Abruzzi, 24 \\
	10129 Torino \\
	Italy
}

\numberwithin{equation}{section}

\usepackage{amsthm,thmtools}
\newtheorem{theorem}{Theorem}[section]
\newtheorem{proposition}[theorem]{Proposition}
\newtheorem{lemma}[theorem]{Lemma}
\newtheorem{corollary}[theorem]{Corollary}

\theoremstyle{definition}
\newtheorem{definition}[theorem]{Definition}
\newtheorem{assume}[theorem]{Assumption}
\newtheorem{remark}[theorem]{Remark}
\newtheorem{example}[theorem]{Example}

\newcommand{\R}{\mathbb{R}}

\newcommand{\cs}{\ensuremath{\sigma}}
\newcommand{\D}[2]{\ensuremath{\frac{\partial #1}{\partial #2}}}
\newcommand{\schouten}{\ensuremath{\mathsf{P}}}
\newcommand{\lb}{\left(}
\newcommand{\rb}{\right)}
\newcommand{\lcb}{\left\{}
\newcommand{\rcb}{\right\}}
\newcommand{\lsb}{\left[}
\newcommand{\rsb}{\right]}

\setlist[enumerate,1]{label=(\roman*)}

\begin{document}

\begin{abstract}

	We show that the definition of a second order superintegrable system on a
	\mbox{(pseudo-)} Riemannian manifold gives rise to a conformally 
	invariant notion of superintegrability. Conformal equivalence is the 
	natural extension of the well-known St\"ackel transform, which in turn 
	originates from the classical Maupertuis-Jacobi principle.
	We extend our recently developed algebraic geometric approach for the 
	classification of second order superintegrable systems in arbitrarily 
	high dimension to conformally superintegrable systems,
	which are presented via conformal scale choices of second order
	superintegrable systems defined within a conformal geometry. 

	For superintegrable systems on constant curvature spaces, we find that the
	conformal scales of St\"ackel equivalent systems arise from eigenfunctions
	of the Laplacian and that their equivalence is characterised by a 
	conformal density of weight two.

	Our approach yields an algebraic equation that governs the classification 
	under conformal equivalence for a prolific class of second order 
	conformally superintegrable systems. This class contains all 
	non-degenerate examples known to date, and is given by a simple algebraic 
	constraint of degree two on a general harmonic cubic form.
	In this way the yet unsolved classification problem is put into the reach 
	of algebraic geometry and geometric invariant theory.
	In particular, no obstruction exists in dimension three, and thus the 
	known classification of conformally superintegrable systems is reobtained 
	in the guise of an unrestricted univariate sextic.
	In higher dimensions, the obstruction is new and has never been revealed 
	by traditional approaches.
	
\end{abstract}

\maketitle
\setcounter{tocdepth}{1}
\tableofcontents

\section{Introduction}
Transformation groups play an important role in the natural sciences: The 
Poincar\'e group, for instance, and its subgroups, are pivotal in special 
relativity, for Maxwell's field equations, in particle physics, and many 
other fields.
Felix Klein's Erlangen program has put the concept of transformations at the 
core of geometry, later generalised by Cartan \cite{Klein1900,Sharpe2000}. 
In particular, relaxing the invariance group of the geometry considered, some 
properties are not preserved any more.
Inspired by this idea, the current paper reconsiders second order 
superintegrable Hamiltonian systems. These have been extensively studied as 
structures in (pseudo-)Riemannian geometry, but not yet as structures in 
conformal geometry.

Superintegrable systems, traditionally, are Hamiltonian systems on a 
(pseudo-)Riemannian geometry that admit a maximal amount of (hidden) 
symmetry. 
They are often seminal models in science.
Historically, the theory of superintegrability arose from classical (and 
quantum) mechanics: While it is impossible, even for relatively simple 
models, to solve Hamiltonian's or Schr\"odinger's equation in exact, closed 
terms, for superintegrable systems the solution can be found by 
\emph{quadrature}, i.e.\ using algebraic operations and the integration of 
known 
functions.
Prominent examples of second order superintegrable systems are the 
Kepler-Coulomb and the Harmonic oscillator models. They have 
fundamental significance for the understanding of celestial mechanics, atomic 
orbitals, material science and many other disciplines.

\subsection{What Geometry underpins 
superintegrability?}\label{sec:geometry.sis}

Traditionally, second order superintegrable systems are defined on a 
(pseudo-)Riemannian manifold $(M,g)$. The suitable symmetry group for these 
systems is the semi-direct product $\mathfrak 
S=\mathrm{Diff}(M)\rtimes\mathrm{Aff}(\R)$ of diffeomorphisms (coordinate 
transformations on $M$) and the affine group 
$\mathrm{Aff}(\R)=\R^*\ltimes\R$, where
$\R^*=\R\setminus\{0\}$, see Section~\ref{sec:preliminaries}.
However, these are not the only possible transformations of superintegrable 
systems. Indeed, conformal geometry manifests itself in the theory of 
superintegrable systems through \emph{St\"ackel equivalence} or 
\emph{coupling constant metamorphosis}; these will be discussed in detail in 
Section~\ref{sec:staeckel}. Historically, they can be traced back to the 
1700s in the form of Maupertuis' principle 
\cite{Tsiganov,Lagrange_1788,Jacobi_1866,Maupertuis_1750}.
St\"ackel transformations are linked to very special conformal 
transformations, namely those that originate from superintegrable potentials.

Arbitrary conformal transformations, however, do not preserve 
superintegrability. Instead they lead to \emph{conformally superintegrable 
systems}. Although these systems are well studied, their underlying conformal 
geometry is understood only superficially to date.
\emph{One purpose of the present paper is to remedy this, and to derive a 
suitable concept of conformal equivalence on conformally superintegrable 
systems from St\"ackel equivalence.}

Given the significance of second order (conformally) superintegrable systems, 
it is natural to seek a classification.
In \cite{Kress&Schoebel&Vollmer}, the authors present an algebraic geometric 
framework for a classification of second order superintegrable systems in 
arbitrarily high dimension, and on arbitrary geometries. This framework put 
earlier attempts by various authors (see below) onto a firm base, yet it is 
not closed under conformal transformations of superintegrable systems. 
Instead, conformal transformations lead to the more general concept of 
conformally superintegrable systems.
\emph{The current paper develops an algebraic geometric framework for 
conformally superintegrable systems that is closed under conformal 
transformations.}
In particular we obtain: Non-degenerate second order 
superintegrable systems on conformal geometries are characterised by a 
\emph{structure tensor}, i.e.\ a trace-free and totally symmetric tensor 
field $S_{ijk}$. This field is invariant under the geometry's transformation 
group, i.e.~under the conformal group, and from $S_{ijk}$ we can reconstruct 
all superintegrable potentials and integrals of the system.

This will lead to a natural definition of \emph{superintegrability on 
conformal manifolds}, whose symmetry group we identify as $\mathfrak 
S=\mathrm{Conf}(M)\rtimes\R^*$.
In this way we naturally incorporate conformal geometry into the theory of 
superintegrable systems. Somewhat surprisingly, this appears to never have 
been attempted before, although it sheds considerable light on the geometry 
underpinning superintegrability and opens the subject for subsequent studies 
using Cartan geometry, tractor calculus, algebraic geometry, representation 
theory and geometric invariant theory.
In particular, superintegrable systems on (pseudo-)Riemannian geometries can 
be viewed as specific \emph{conformal scale choices} of a conformally 
invariant superintegrable system.

\subsection{State of the art}\label{sec:SoA}
A vast literature exists both on second order conformally superintegrable 
systems and on St\"ackel transformations. Most importantly, second order 
conformally superintegrable systems are classified completely in dimension~2 
\cite{Kalnins&Kress&Miller-I,Kalnins&Kress&Miller-II}. For conformally flat 
spaces in dimension~3, at least so called non-degenerate systems are 
classified \cite{Capel_phdthesis,Kalnins&Kress&Miller-IV}. The conformal 
classes of non-degenerate systems are classified in dimensions~2 and~3 
\cite{Kress07,Capel_phdthesis}.

St\"ackel transform is well understood as an equivalence relation on 
second order (conformally) superintegrable systems.
It has first been introduced, under the name \emph{coupling constant 
	metamorphosis}, by Hietarinta et al.\  in \cite{HGDR84}, for integrable 
Hamiltonian systems with potential.
The concept of St\"ackel transformations has been introduced in 
\cite{BKM1986} by Boyer, Kalnins \& Miller, for integrable systems that admit 
separation of variables.
Coupling constant metamorphosis and St\"ackel transform are not identical in 
general, but do coincide for second order (conformal) integrals of the 
motion. They are therefore equivalent in the context which is of interest 
here. Details on the interrelation between both concepts can be found in 
\cite{Post10}.
Higher order integrals are also discussed in \cite{Kalnins_2009}, where it is 
proven that in general coupling constant metamorphosis neither preserves the 
order of the integrals of the motion, nor even their polynomiality in momenta.
A multi-parameter generalisation of St\"ackel transform exists as well
\cite{Sergyeyev_2008,Blaszak_2012,Blaszak_2017}. Instead of following this 
direction, we are going to use a geometric approach that does not rely on a 
specific parametrisation of the space of compatible potentials of a given 
superintegrable system, but instead is based on the symmetry groups mentioned 
in Section~\ref{sec:geometry.sis} above.

St\"ackel transformations can historically be traced back to 
\emph{Maupertuis-Jacobi transformations}, which take a Hamiltonian with 
potential to a potential-free one, see for example 
\cite{Bolsinov1995,Tsiganov}.
Equivalence classes of superintegrable systems, under St\"ackel transforms, 
have been established for dimension~2 and~3 in \cite{Kress07} and in 
\cite{Kalnins&Kress&Miller-IV,Capel_phdthesis}, respectively. For these 
geometries, every conformally superintegrable system can, via a St\"ackel 
transformation using the potential, be transformed into a 
properly\footnote{For clarity, in this paper 
	we use the adjective ``proper'' to refer to superintegrable systems, 
	emphasising the distinction to \emph{conformally} superintegrable 
	systems.} 
superintegrable system \cite{Capel_phdthesis}.

As mentioned, the existing work on second order conformally superintegrable 
systems and their conformal equivalence is restricted to dimensions 2 and 3. 
In dimension 2, a classification exists, but ignores the geometric structure 
of the classification space~\cite{Kress07}. In dimension 3, a 
classification in terms of representations of the rotation group exists
\cite{Capel_phdthesis,Capel&Kress,Capel&Kress&Post}. These latter references 
are one major inspiration for our work as they highlight the power of the 
geometric approach, revealing for example a natural algebraic hierarchy of 
systems related to an inclusion tree of certain algebraic ideals. 
Unfortunately, there is little hope to apply the methods used in those 
references to higher dimensions, neither conceptually nor practically, as the 
equations become ever more extensive with increasing dimension.
\emph{The current paper develops a new approach, extending and generalising 
the framework from \cite{Kress&Schoebel,Kress&Schoebel&Vollmer}, which 
formulates the governing equations for second order properly superintegrable 
systems in dimension $n\geqslant3$ in a concise form, making the problem 
manageable for higher dimensions.}

For completeness we mention that conformal transformations are not the only 
possible transformations of superintegrable systems. For instance, B\^ocher 
transformations of certain conformally superintegrable systems are studied in 
\cite{KMS2016,Capel&Kress&Post} and there is some indication that they can be 
understood as boundaries of orbit closures on the algebraic geometric variety 
classifying the superintegrable systems \cite{Kress&Schoebel}.
Yet another transformation of superintegrable systems is possible if the 
underlying metrics share the same geodesics up to reparametrisation. Such 
metrics are called projectively (or geodesically) equivalent. For some 
examples of superintegrable systems defined on projectively equivalent 
geometries, see \cite{Valent,KKMW03,Vollmer_2020}.

To summarise, second order non-degenerate conformally superintegrable systems 
are to date classified in dimensions two and three, for manifolds that are 
conformally flat. Higher dimensions are out of the scope of traditional 
methods, which rely on the correspondence with properly superintegrable 
systems and on the extensive use of computer algebra. A particular challenge 
with traditional methods is the fast growth of the number of partial 
differential equations with increasing dimension. \emph{In the current paper, 
	we shall overcome this problem and outline how to approach the 
	classification of second order conformally superintegrable systems in 
	arbitrarily high dimension.}
For the most prolific class of systems we find, somewhat surprisingly, that 
the underlying structure equations reduce to only a single, algebraic 
equation of degree~2.

\subsection{Special functions and superintegrable systems}
Special functions are ubiquitous and powerful tools in science and 
technology. For instance spherical harmonics appear in the solution of 
the Schr\"odinger equation for the hydrogen atom, eventually giving rise to 
the periodic table of the elements.
Superintegrable systems have long been known as a rich source of 
special functions, such as hypergeometric orthogonal
polynomials~\cite{KMP07,KMP13}, Painlevé transcendants
\cite{Marquette_2020,Gravel}, Jacobi-Dunkl polynomials \cite{GIVZ13}, and 
exceptional polynomials
\cite{Post&Tsujimoto&Vinet,Hoque&Marquette&Post&Zhang}.
In particular, a striking resemblance has been found between two directed 
graphs: The first describes degenerations and confluences of hypergeometric 
polynomials in the Askey scheme \cite{Askey,Askey&Wilson}. The vertices of 
the other graph represent (conformal classes of) superintegrable systems in 
dimension~2, and 
its edges 
represent orbit degenerations and B{\^o}cher contractions.
In dimension 3, the generic superintegrable system on the 3-sphere has indeed 
been related to bivariate Wilson polynomials \cite{KMP11}, and interbasis 
expansions for the isotropic 3D harmonic oscillator are linked to bivariate 
Krawtchouk polynomials \cite{GVZ14}.

A systematic documentation is indispensable for the use of special functions, 
and also needs to include their properties, interrelations and 
hierarchies.
Unfortunately, such documentation quickly becomes extensive and laborious. 
The Bateman Manuscript Project, for instance, fills five volumes 
\cite{Bateman53,Bateman54} and the steadily 
growing Mathematical Functions Site \cite{Wolfram} comprises to date more 
than 300,000 formulae.
In the face of the sheer amount of information, a natural question is whether 
there is a unified theory of special functions.
While a general unified theory might be out of reach, a systematisation of 
certain classes of special functions is a realistic goal. Ideally, such a 
framework should explain and organise at least some of the major properties 
of the functions it comprises.
One attempt at such a partial systematisation is the Askey-Wilson scheme 
\cite{Askey,Askey&Wilson}, which establishes a hierarchy of hypergeometric 
orthogonal polynomials.

Conformal superintegrability has also been related to different Laplace 
equations~\cite{Kalnins&Kress&Miller&Post11,Kalnins&Miller2016}, particularly 
in the context of the corresponding quantum systems.
With respect to special functions we find that second order properly 
superintegrable systems on constant curvature spaces correspond to 
eigenfunctions of the Laplace operator, where the eigenvalues are determined 
by the scalar curvature.

The classification of special functions arising from superintegrable systems 
is clearly out of the scope of the current paper.
Our approach has the potential to explain this correspondence, including 
higher-dimensional hypergeometric polynomials and higher dimensional 
generalisations of the Askey-Wilson scheme.

\subsection{Classifying second order superintegrability in arbitrarily high 
dimension}

In reference~\cite{Kress&Schoebel&Vollmer} the authors have developed an 
algebraic geometric framework for the classification of superintegrable 
systems. This framework generalises previous work in dimension two 
\cite{Kress&Schoebel} to arbitrarily high dimensions.
Older works in the field are \cite{Kalnins&Kress&Miller} and 
\cite{Kalnins&Kress&Miller3D} for dimensions two and three, respectively.
While \cite{Kress&Schoebel&Vollmer} for the first time provides a framework 
to classify, in an algebraic geometric way, superintegrable systems in 
arbitrarily high dimensions, this framework in its original form is not yet 
closed under conformal transformations.
\emph{The present paper extends the existing algebraic geometric framework to 
conformally superintegrable systems on (pseudo-)Riemannian metrics. This new 
framework is closed under conformal transformations}.
Second order conformally superintegrable systems will be thoroughly 
introduced in Section~\ref{sec:conf.SIS} and~\ref{sec:conformal.equivalence}.
Conceptually, these systems are traditionally defined using a Hamiltonian
\begin{equation}\label{eqn:natural.hamiltonian}
	H = g^{ij}(q)p_ip_j+V(q)
\end{equation}
where $g_{ij}$ denotes the underlying metric and where $q$ and $p$ are the 
canonical position and momenta variables on the manifold.
A second order conformally superintegrable system is a Hamiltonian system 
with a sufficiently high number of functions $F:T^*M\to\R$,
\[
	F=K^{ij}(q)p_ip_j+W(q)\,,
\]
satisfying 
\begin{equation}\label{eqn:conformal.integral}
	\{ H,F \} = \omega_q(p)\,H
\end{equation}
for some 1-form $\omega=\omega_idx^i$.
In this context, the scalar function $V$ is called the \emph{potential}.
Functions~$F$ satisfying \eqref{eqn:conformal.integral} are called 
\emph{(conformal) integrals}.
The collection of integrals~$F$ forms a linear space~$\mathcal F$.
On the other hand, if we start from the space $\mathcal F$, we are going to 
see that the collection~$\mathcal V$ of potentials~$V$ compatible with 
$\mathcal F$ forms a linear space of functions on the base manifold.

Before formulating the main results of the present paper, we would like to 
draw the reader's attention to a subtlety regarding the spaces $\mathcal F$ 
and $\mathcal V$, as different conventions can be found in the literature.
In the present paper, we will be working with the maximal spaces $\mathcal V$ 
and $\mathcal F$ (modulo some normalisation detailed below). This allows us 
to formulate the results in a clean and concise manner. This convention also 
appears to be the most common one in the relevant literature on St\"ackel 
transforms, as it ensures well-definedness of the conformal equivalence of 
two superintegrable Hamiltonians.
A competing convention regularly found in the literature is to restrict to a 
linear subspace of~$\mathcal F$ with a basis of \emph{functionally 
independent} integrals $F$ and with a Hamiltonian for which $V\in\mathcal V$ 
is a specific function. Our results can straightforwardly, but 
somehow tediously, be adapted to these settings. Such specifications are 
omitted here as they do not contribute to further understanding the geometry 
underpinning superintegrability, which is the main objective of the paper.

\subsection{First Main Result: Conformal superintegrability in higher 
			dimensions}\label{sec:main.result.higher.dimension}
The method carried out in reference \cite{Kress&Schoebel&Vollmer} facilitates 
the classification of second order properly superintegrable systems, 
in particular of so-called \emph{abundant} systems. Abundantness is going to 
be introduced thoroughly later, and so here we limit ourselves here to saying 
that these systems comprise all known non-degenerate second order conformally 
superintegrable systems.
In the present paper, we extend the framework to conformally superintegrable 
systems. We find that it is closed under conformal transformations and leads 
to a well-defined concept of superintegrability on conformal geometries 
arising from conformal equivalence classes of conformally superintegrable 
systems.
For the abundant case in dimensions~$n\geqslant3$, we show in 
Sections~\ref{sec:abundant} and~\ref{sec:ccsis} that such systems 
are in natural correspondence with (trace-free) cubic forms 
$\Psi_{ijk}p^ip^jp^k$ on $\R^n$ that satisfy the simple algebraic equation
\begin{equation}\label{eqn:master}
	\left( g^{ab}\Psi\indices{_{bi[j}}\Psi_{k]la} \right)_\circ = 0
\end{equation}
where $g^{ab}$ is an inner product on $\R^n$ with the same signature as the 
metrics on the underlying manifold. Here square brackets denote 
antisymmetrisation over enclosed indices, and the subscript $\circ$ stands 
for projection onto the trace-free part.

We show that initial data in the form of a cubic form 
$\Psi_{ijk}p^ip^jp^k$ satisfying \eqref{eqn:master} can be extended, locally 
to a \emph{conformal structure tensor} $S_{ijk}$ of an abundant second order 
superintegrable system.
We make this precise by introducing the concept of c-superintegrable systems, 
i.e.\ conformal equivalence classes whose underlying geometry is a conformal 
manifold.
In Section~\ref{sec:invariant.nonlinear.prolongation} we derive conformally 
invariant structural equations for abundant c-superintegrable systems. The 
equations governing abundant properly superintegrable 
systems~\cite{Kress&Schoebel&Vollmer} naturally follow from the equations we 
present here.

Condition~\eqref{eqn:master} is conformally invariant, and therefore a 
suitable foundation for an algebraic geometric classification of second order 
systems on the level of conformal geometries.
Condition~\eqref{eqn:master} is also surprisingly simple, and in dimensions 
$n\geqslant4$ it encodes new obstructions to conformal superintegrability. 
These obstructions do not exist in lower dimensions and have not been 
revealed by classical approaches. It is worthwhile to compare 
\eqref{eqn:master} to the corresponding equation for the case of proper 
superintegrability in~\cite{Kress&Schoebel&Vollmer}, which is not projected 
onto the trace-free part.

We also show: \emph{Abundant conformally superintegrable systems can only 
exist on conformally flat geometries. Such systems naturally correspond to 
solutions of Equation~\eqref{eqn:master}. The task of classifying equivalence 
classes of $n$-dimensional conformally superintegrable systems is therefore
equivalent to classifying harmonic cubics in $n$ variables that satisfy 
\eqref{eqn:master}.}
Note that while a general classification of harmonic cubics under the 
rotation group is out of sight, a classification under the additional 
condition~\eqref{eqn:master} may well be simple enough to admit 
a managable solution.

In dimension~$n=3$, particularly, \eqref{eqn:master} is trivially satisfied.
Thus abundant superintegrable systems in dimension~$3$ are in 
1-to-1 correspondance with harmonic ternary cubic forms, or, equivalently, 
with univariate sextic polynomials
\[
	p(x) = a_6x^6 + a_5x^5 + a_4x^4 + a_3x^3 + a_2x^2 + a_1x + a_0\,.
\]
The details are discussed in Section~\ref{sec:3D}.
It is known that any conformally superintegrable system is St\"ackel 
equivalent to a properly superintegrable systems \cite{Capel_phdthesis}. In 
dimension~3, every conformally superintegrable system is even St\"ackel 
equivalent to a properly superintegrable system on a constant curvature space 
\cite[Theorem 4]{Kalnins&Kress&Miller-IV}.

\subsection{Second Main Result: Superintegrable systems on constant curvature 
geometries.}

In our framework, second order conformally superintegrable systems are 
conformal scale choices of c-superintegrable systems.
Thus the conformal structure tensor $S_{ijk}$ determines a conformally
superintegrable system up to the choice of a \emph{conformal scale} expressed 
via a function $\cs$ that transforms as a weight-1 tensor density.
For the prolific class of abundant second order conformally superintegrable 
systems we find that the conformal scale satisfies a Helmholtz like equation,
\begin{equation}\label{eqn:Helmholtz.conformal}
	\left( R-4\frac{n-1}{n-2}\Delta \right)\cs^{1-\tfrac{n}{2}} = 
	S\,\cs^{1-\tfrac{n}{2}}\,,
\end{equation}
where $S=S_{abc}S^{abc}$ is a conformal density obtained from the structure 
tensor $S_{ijk}$ mentioned earlier, and where $R$ is the scalar curvature.
The operator on the left hand side of \eqref{eqn:Helmholtz.conformal} is the 
conformal Laplace operator.

If we restrict to properly superintegrable systems, the conformal structure 
tensor $S_{ijk}$ determines a superintegrable system up to provision of a 
suitable conformal scale. 
In the present paper, we prove the following:
\emph{For abundant properly superintegrable systems on manifolds of constant 
sectional curvature, the conformal scale $\cs$ is (the power of) an 
eigenfunction of the \emph{conformal} Laplacian} for an eigenvalue determined 
by the curvature $R$,
\begin{equation}\label{eqn:Laplace.scale}
	\left( \Delta+2\frac{n+1}{n-1} R\right)\,\cs^{n+2} = 0\,,
\end{equation}
which holds in addition to~\eqref{eqn:Helmholtz.conformal}, see 
Theorem~\ref{thm:eigenvalue.equation}.
Note that the operator in~\eqref{eqn:Laplace.scale} is not conformally 
invariant as Equation~\eqref{eqn:Laplace.scale} does not describe a 
property of the conformal class, but of an individual superintegrable system.
In particular we find that on the $n$-sphere the conformal scale function 
satisfies a Laplace equation with quantum number $n+1$.
Theorem~\ref{thm:generic.system} shows: \emph{The generic system on the 
$n$-sphere is never conformally equivalent to a superintegrable system on 
flat space.}

Moreover, in Theorem~\ref{thm:B.transformation} we extend the concept of 
structure functions 
to c-superintegrable systems:
If two abundant second order properly superintegrable systems on 
constant curvature spaces are conformally equivalent, then their structure 
functions behave like conformal densities of weight $-2$.
In fact (up to a certain conformally equivariant gauge transformation that 
leaves the conformally equivariant tensor $S_{ijk}$ unchanged)
\[  \mathbf{b}=B\det(g)^{\frac1n}\ \in\mathcal{E}[-2], \]
computed from the structure function $B$ and metric $g$ of any one of the two 
systems, coincides with the one computed from the structure function and 
metric of the other system, see Corollary~\ref{cor:b.density}.


\bigskip
\noindent
\textbf{Acknowledgements.}
We would like to thank Rod Gover for discussions and insightful input from 
the viewpoint of parabolic and conformal geometry, and for travel support via 
the Royal Society of New Zealand from Marsden Grant 16-UOA-051. We are 
grateful towards Benjamin McMillan, Thomas Leistner and Paul-Andy Nagy for 
helpful comments and discussions from the angle of conformal geometry and 
representation theory.
We thank Vladimir Matveev for comments on integrability, analyticity and 
projective geometry.
We are indepted to Joshua Capel for his insights into superintegrable systems 
in dimension 3 and the classification of their St\"ackel classes. We also 
thank Jeremy Nugent for discussions.

We thank the authors and contributors of the computer algebra system 
\texttt{cadabra2} \cite{Peeters06,Peeters07}, which we have used to find, 
prove and simplify some of the most important results in this work, for 
providing, maintaining and extending their software and distributing it
under a free license.

We acknowledge support from ARC Discovery Project grant DP190102360.
Funded by Deutsche Forschungsgemeinschaft (DFG, German Research Foundation) 
Project number 353063958.
AV also acknowledges support through the project PRIN 2017 ``Real and Complex 
Manifolds: Topology, Geometry and holomorphic dynamics'', and the MIUR grant 
``Dipartimenti di Eccellenza 2018-2022 (E11G18000350001)''. AV is a member of 
GNSAGA of INdAM (Istituto Nazionale di Alta Matematica).

\section{Preliminaries}\label{sec:preliminaries}

Before generalising to \emph{conformally} superintegrable systems, it is 
instructive to briefly review \emph{properly} superintegrable systems. We 
recall that for clarity, the adjective ``proper'' is used to refer to 
superintegrable systems, whenever a distinction from \emph{conformally} 
superintegrable systems is required.
While self-contained, this review only highlights the aspects needed for a 
later comparison to conformally superintegrable systems.
For a more in-depth review of proper superintegrability, we refer the 
interested reader to the literature cited in the introduction.
Reference~\cite{KKM_2018}, in addition, provides a synopsis of a large part 
of the existing literature.
The PhD thesis \cite{Capel_phdthesis} and the articles 
\cite{Kress&Schoebel,Kress&Schoebel&Vollmer} are foundational for the
algebraic geometric approach to properly superintegrable systems, see also 
\cite{KKM07b,Kalnins&Kress&Miller3D}.

A Hamiltonian system is a dynamical system characterised by a
function $H(\mathbf p,\mathbf q)$, referred to as \emph{Hamiltonian}.
Here, the position and momentum coordinates on the phase space are denoted by 
$\mathbf q=(q_1,\ldots,q_n)$ and $\mathbf p=(p_1,\ldots,p_n)$, respectively.
The evolution of the system is determined by Hamilton's equations
\begin{align}\label{eqn:hamilton.equations}
	\dot{\mathbf p}&=-\D{H}{\mathbf q}&
	\dot{\mathbf q}&=+\D{H}{\mathbf p}\,.
\end{align}
An \emph{integral}, aka \emph{first integral} or \emph{constant of the 
motion}, for the Hamiltonian $H$ is a function $F(\mathbf p,\mathbf q)$ on 
phase space that commutes with $H$ with respect to the canonical Poisson 
bracket.
It is therefore constant along solutions of \eqref{eqn:hamilton.equations},
\begin{equation}\label{eqn:integral.of.motion}
	\dot F = \{F,H\}
	= \D{F}{\mathbf q}\,\D{H}{\mathbf p}
	- \D{F}{\mathbf p}\,\D{H}{\mathbf q}
	=0.
\end{equation}
Note that, if the underlying manifold is endowed with a (pseudo-)Riemannian 
metric $g$, the usual derivatives in this equation may be replaced by 
covariant derivatives, using the Levi-Civita metric $\nabla^g$, without 
changing the Poisson bracket. For simplicity we assume the metric to be 
analytic from now on.

An integral restricts the trajectory of the system to a hypersurface in phase 
space.
A (properly) \emph{superintegrable system} is a Hamiltonian system that 
possesses the maximal number of $2n-1$ functionally independent 
constants of motion $F^{(0)},\ldots,F^{(2n-2)}$. Its trajectories in phase 
space are the (unparametrised) curves given as the intersections of the 
hypersurfaces $F^{(\alpha)}(\mathbf p,\mathbf q)=c^{(\alpha)}$, where the 
constants $c^{(\alpha)}$ are determined from the initial conditions.
For convenience it is customary to choose $F^{(0)}=H$ without loss of 
generality.
In particular, we assume the base manifold is endowed with a 
(pseudo-)Riemannian metric $g$ and a natural 
Hamiltonian~\eqref{eqn:natural.hamiltonian},
\[
	H = G(\mathbf q,\mathbf p) + V(\mathbf q)\,,
\]
where $G(\mathbf q,\mathbf p)=g_{\mathbf q}(\mathbf p,\mathbf p)$ denotes the 
kinetic part and 
$V(\mathbf q)$ is a smooth scalar function called \emph{potential}.
\begin{definition}
	A \emph{second order} superintegrable systems is a Hamiltonian together 
	with a linear space $\mathcal F$ of integrals of the form
	\begin{equation}\label{eqn:quadratic.integral}
		F = K(\mathbf q,\mathbf p) + W(\mathbf q) 
		:= K^{ij}(\mathbf q)p_ip_j + W(\mathbf q)\,,
	\end{equation}
	satisfying \eqref{eqn:integral.of.motion}. Moreover, $\mathcal F$ must 
	contain $2n-1$ integrals that are functionally independent.
\end{definition}
\noindent Note that $\dim(\mathcal F)\geqslant2n-1$. In case of the equality, 
it 
is common practice to only specify the $2n-1$ linearly independent generators 
$F^{(\alpha)}$.
We also recall that, for~\eqref{eqn:quadratic.integral}, 
Equation~\eqref{eqn:integral.of.motion} is a polynomial condition in the 
momenta with homogeneous components of cubic and linear degree, respectively:
\begin{subequations}\label{eqn:proper.cubic.linear}
\begin{align}
	\label{eqn:proper.cubic}
	\{ K , G \} &= 0 \\
	\label{eqn:proper.linear}
	\{ K , V \} + \{ W , G \} &= 0
\end{align}
\end{subequations}
Condition~\eqref{eqn:proper.cubic} is equivalent to the requirement that the 
components $K_{ij}$ in $K(\mathbf q,\mathbf p)=K_{ij}p^ip^j$ are components 
of a Killing tensor field
\begin{equation}
	K_{(ij,k)} = 0\,.
\end{equation}
Here, the comma denotes a covariant derivative and round brackets denote 
symmetrisation in the enclosed indices.
Condition~\eqref{eqn:proper.linear} can be rewritten in the form
\[
	W_{,j} = K\indices{_j^k}V_{,k}
	\quad\text{or even}\quad
	dW = KdV\,,
\]
where by a slight abuse of notation $K$ denotes the endomorphism obtained 
from $K_{ij}$ using the metric~$g$.
The integrability condition for $W$ is known as the \emph{Bertrand-Darboux 
condition} \cite{bertrand_1857,darboux_1901},
\begin{equation}\label{eqn:Bertrand.Darboux}
	dKdV = 0\,.
\end{equation}
The Bertrand-Darboux Equation~\eqref{eqn:Bertrand.Darboux} is the 
compatibility condition for the potential $V$ and the space of Killing 
tensors~$K_{ij}$.
Let us denote the linear space of kinetic parts of the integrals  
$F\in\mathcal{F}$, viewed as endomorphisms, by
\[
	\mathcal K = \{ K : K(\mathbf p,\mathbf p)+W\in\mathcal F
						\text{ for some $W$} \}.
\]

\begin{definition}\label{def:proper.max.spaces}
	For a second order superintegrable system with potential $V$, we 
	introduce the spaces
	\begin{align*}
		\mathcal V^\text{max}
		&= \{ V : dKdV=0 \text{ holds for every }\ K\in\mathcal K \}
		\\
		\mathcal K^\text{max}
		&= \{ K : dKdV=0 \text{ holds for every }\ 
					V\in\mathcal V^\text{max} \}		
	\end{align*}
\end{definition}

\begin{remark}
	A second order superintegrable system is said to be \emph{irreducible} if 
	the endomorphisms $K\indices{_i^j}=g^{ja}K_{ia}$ obtained from its 
	associated Killing tensors $K\in\mathcal K$ form an irreducible set 
	\cite{Kress&Schoebel&Vollmer}.
	In reference \cite{Kress&Schoebel&Vollmer}, it is shown that for such 
	irreducible systems we can solve \eqref{eqn:Bertrand.Darboux} for all 
	second derivatives of the potential except its Laplacian. Thus the 
	\emph{Wilczynski equation} is obtained,
	\begin{equation}\label{eqn:Wilczynski.proper}
		V_{,ij} = T\indices{_{ij}^k}V_{,k} + \frac1n g_{ij} \Delta V\,,
	\end{equation}
	where $T\indices{_{ij}^k}$ is a tensor symmetric and trace-free in the 
	first two indices, depending on the values of the Killing 
	tensors $K^{(\alpha)}$ and their derivatives.
	
	The properties of the partial differential equation 
	\eqref{eqn:Wilczynski.proper} are discussed thoroughly in 
	\cite{Kress&Schoebel&Vollmer}, and similar equations appear in 
	\cite{Kalnins&Kress&Miller-III}. The most important fact is that 
	in~\eqref{eqn:Wilczynski.proper} the tensor $T\indices{_{ij}^k}$ is 
	determined by $\mathcal K$ independently from the potential. More 
	precisely, at a point $x_0\in M$, $T\indices{_{ij}^k}$ is determined by 
	the values of the Killing tensors $K_{ij}$ and their derivatives in $x_0$.
\end{remark}

\noindent In the classification theory of second order superintegrable 
systems, \emph{non-degenerate systems} have received particular 
attention, e.g.~\cite{KKM_2018,Kress&Schoebel&Vollmer,Capel_phdthesis}. 
These are the systems satisfying~\eqref{eqn:Wilczynski.proper} for which the 
dimension of $\mathcal V^\text{max}$ is maximal,\footnote{Note that for an 
analytic metric, the Killing tensors are analytic, and thus the structure 
tensor and the potentials are also analytic.} i.e.
\[
	\dim(\mathcal V^\text{max}) = n+2\,.
\]
The integrability conditions of~\eqref{eqn:Wilczynski.proper} are then 
generically satisfied 
\cite{KKM07c,Kalnins&Kress&Miller3D,Kress&Schoebel&Vollmer}.
Resubstituting \eqref{eqn:Wilczynski.proper} into 
\eqref{eqn:Bertrand.Darboux} and considering the coefficients of $\nabla V$ 
and $\Delta V$, one furthermore finds
\begin{equation}\label{eqn:shortcut.eqn}
	K_{ij,k}=\frac13\young(ji,k)T\indices{^a_{ji}}K_{ak}
\end{equation}
for non-degenerate systems.
Now consider non-degenerate systems for which the dimension of $\mathcal 
K^\text{max}$ is maximal as well, i.e.
\[
	\dim(\mathcal K^\text{max}) = \frac{n(n+1)}{2}\,.
\]
Such systems are called \emph{abundant} 
\cite{Kress&Schoebel&Vollmer}.
Abundantness implies that the integrability conditions 
of~\eqref{eqn:shortcut.eqn} are generically satisfied.
Abundant systems should be considered the most important class of second 
order superintegrable systems, for at least two reasons: First, in dimensions 
two and three, all non-degenerate systems are abundant 
\cite{Kalnins&Kress&Miller-III}. Second, all examples 
of non-degenerate systems known to date are abundant systems.

\subsection{St\"ackel equivalence}\label{sec:staeckel}
In the literature, St\"ackel transformations are typically introduced by 
considering a Hamiltonian with a coupling constant. We shall start following 
this convention, but we are then going to present an alternative and 
equivalent formulation that is better suited for our purposes.
Consider a family of second order superintegrable systems on a 
(pseudo-)Riemannian manifold $(M,g)$ given by the family of Hamiltonians 
$H_\beta=H_0+V+\beta U$. Here $H_0=g(\mathbf p,\mathbf 
p)$ is the \emph{free Hamiltonian}, $H_0+V$ is the background 
Hamiltonian, and $\beta$ is called the \emph{coupling parameter}.
The concept of St\"ackel equivalence is based on the following fact, see for 
example~\cite{Kalnins_2009,Post10,Kalnins&Kress&Miller-III,HGDR84,BKM1986}.
\begin{lemma}\label{la:Stackel.equivalent.H}
 Let $H=H_0+V+\beta\,U$ be a family of second order superintegrable 
 Hamiltonians with integrals $F(\beta)$, for $V,U\in\mathcal V^\text{max}$.
 Then the Hamiltonian $\tilde H=\frac{H+\eta}{U}$ admits the integral of 
 motion $\tilde F(\eta)=F(\tilde 
 H)$, parametrised by $\eta$.
\end{lemma}
\noindent The Hamiltonian $\tilde H$ is called the \emph{St\"ackel transform} 
of $H$ with conformal factor $U$.
While the lemma is true for either sign, it will often be preferrable to work 
with $U>0$ exclusively, in order to preserve the signature of the underlying 
metric. Note that locally this is always possible by a redefinition of 
$\beta$, if needed.

St\"ackel transformations are also known under the name \emph{coupling 
constant metamorphism} since they can be thought of as a transformation 
exchanging the roles of the coupling parameter and the energy. However, we 
remark that the two concepts coincide for second order superintegrability, 
but they are not identical in general \cite{Kalnins_2009,Post10}.

Note that in Lemma~\ref{la:Stackel.equivalent.H}, it has been 
exploited that we may add any constant $\eta$ to $H$ without changing the 
integrals. We could analogously have written down the St\"ackel transform of 
the Hamiltonian $H=H_0+V+\beta U+\eta$ admitting the integrals $F^{(\alpha)}$ 
\cite{BKM1986,Post10}, as
\begin{subequations}
\label{eq:staeckel}
	\begin{align}
		\label{eq:staeckel.H}
		\tilde H &= U^{-1}H\,,
		\\ 
		\label{eq:staeckel.F}
		\tilde F^{(\alpha)} &= F^{(\alpha)} + \frac{1-W^{(\alpha)}}{U}\,H\,.
	\end{align}
\end{subequations}
We observe that this transformation preserves the kinetic part up to a trace 
contribution, i.e.\ up to a term proportional to $g(\mathbf p,\mathbf p)$ 
with a coefficient that depends on the position only.

One special case deserves explicit mentioning: When $H=H_0+\beta U$, we 
arrive at $\tilde H=\frac{H_0}{U}+\beta$. This special 
case is known as a \emph{Maupertuis-Jacobi transformation} 
\cite{Tsiganov,Lagrange_1788,Jacobi_1866,Maupertuis_1750}.
For conciseness, we have restricted ourselves to one coupling parameter in 
Lemma~\ref{la:Stackel.equivalent.H}. A multi-parameter generalisation of 
St\"ackel transforms exists as well 
\cite{Sergyeyev_2008,Blaszak_2012,Blaszak_2017}.
We are not going to follow this direction further, however, because for the 
purposes of the present paper a parameter-free viewpoint on St\"ackel 
equivalence is much more appropriate.
Equivalently and better suited for the purposes of the present paper, we 
define St\"ackel equivalence as follows.
\begin{definition}\label{def:staeckel.equivalence}
	Two second order properly superintegrable systems are said to be 
	St\"ackel equivalent if their Hamiltonians and integrals 
	satisfy~\eqref{eq:staeckel}.
\end{definition}

\noindent We observe that for St\"ackel equivalent Hamiltonians $H=H_0+V$ and 
$\tilde H=\tilde H_0+\tilde V$, their underlying metrics are conformally 
equivalent, i.e.~$\tilde g_{ij}=\Omega^2\,g_{ij}$, if $U=\Omega^2>0$. In case 
of negative sign, $U<0$, the metric's signature is merely inverted.
For the corresponding integrals $F=K^{ij}p_ip_j+W$ and $\tilde F=\tilde 
K^{ij}p_ip_j+\tilde W$, \eqref{eq:staeckel} implies
\begin{align*}
	\tilde K
	&= K+\frac{1-W}{U}g
	\\
	\tilde W
	&= W +\frac{1-W}{U}\,V\,.
\end{align*}
The trace-free part of the Killing tensors, with raised indices, 
therefore remain unchanged under St\"ackel transformations.

\subsection{Symmetry group}\label{sec:projective.V.space.proper}

Consider a non-degenerate second order properly superintegrable system with 
potential $V$.
The symmetry group of such a system is the semi-direct product $\mathfrak 
S=\mathrm{Diff}(M) \rtimes \mathrm{Aff}(\R)$ of the diffeomorphisms of the 
manifold $M$ and the affine group $\mathrm{Aff}(\R)\simeq\R^*\ltimes\R$.
An element $\Phi=(\phi,(a,b))\in\mathfrak S$ transforms a Hamiltonian 
according to
\begin{equation}\label{eqn:symmetry.group.proper}
	\Phi(g^{ij}p_ip_j+V) = \phi^*(g)^{ij}p_ip_j+a\phi^*(V)+b\,,
\end{equation}
where $\phi^*$ is the pullback with $\phi$.
Indeed, the underlying geometry and the space of compatible Killing tensors 
does not change under $\mathfrak S$. Moreover, the structure tensor $T_{ijk}$ 
in \eqref{eqn:Wilczynski.proper} remains unchanged under $\mathfrak S$.
The structure tensor remains unchanged even\footnote{%
	Note that the action might not be proper.}
under $\mathfrak S'=\mathfrak S\times\R^*$, where 
$\Phi'=(\phi,(a,b),c)\in\mathfrak S$ 
transforms a Hamiltonian according to
\begin{equation}\label{eqn:symmetry.structure.tensor.proper}
	\Phi'(g^{ij}p_ip_j+V) = c\phi^*(g)^{ij}p_ip_j+a\phi^*(V)+b\,.
\end{equation}
The quotient of the $(n+2)$-dimensional space $\mathcal V^\text{max}$ under 
the affine group $\mathrm{Aff}(\R)$ is an $n$-dimensional projective space.
Finally, note that the St\"ackel transformations of two equivalent 
Hamiltonians are also equivalent.

\subsection{Young projectors}

In order to keep the notation concise, tensor symmetries are described by 
Young projectors in the following.
In doing so, we adhere to the convention used for properly superintegrable systems in \cite{Kress&Schoebel&Vollmer}, which we briefly review in the current section.
A comprehensive introduction to representations of symmetric and linear 
groups is out of scope of the present paper, but can be found in 
\cite{Fulton,Fulton&Harris} for instance.

Let $n>0$ be an integer. A partition of $n$ into a sum of ordered, positive 
integers can be represented by a \emph{Young frame}, i.e.\ by non-increasing 
rows of square boxes, which by convention are left-aligned.
For instance, to denote the partition $5=3+1+1$ we may draw the Young frame
\[
	{\Yboxdim{7pt}\yng(3,1,1)}\,.
\]
Irreducible representations of the permutation group~$S_n$ and the induced 
Weyl representations of $\mathrm{GL}(n)$ can also be labelled by Young 
frames. A Young frame filled with tensor index names is called a \emph{Young 
tableau}; it explicitly defines a projector onto an irreducible 
representation.
Two simple examples are complete symmetrisation,
\[
	\begin{ytableau}
		{i_1} & {i_2} & \text{\tiny$\cdots$} & {i_k}
	\end{ytableau}\,,
\]
and antisymmetrization,
\[
	{\begin{ytableau}
		{i_1} \\ {i_2} \\ \text{\tiny$\vdots$} \\ {i_k}
	\end{ytableau}}\,.
\]
For example, a 2-tensor $\tau_{ij}$ can be decomposed into its symmetric and 
antisymmetric parts,
\[
	\tau_{ij}
	= \frac12\,\young(ij)\tau_{ij} + \frac12\,\young(i,j)\tau_{ij}
	= \frac12\,(\tau_{ij}+\tau_{ji})+ \frac12\,(\tau_{ij}-\tau_{ji}).
\]
The symmetric part can be decomposed further, according to irreducible 
representations of $\mathrm{SL}(n)$, into a trace-free and a trace component.
The projection onto the completely trace-free part of a tensor is denoted by 
the sub- or superscript ``$\circ$''. For example,
\[
	\tau_{ij}
	= \frac12\,{\young(ij)}_\circ\,\tau_{ij}
	+ \frac1n\,g_{ij}\,\tau\indices{^a_a}
	+ \frac12\,\young(i,j)\tau_{ij}\,.
\]
A general Young tableau denotes the composition of its row symmetrisers 
and column antisymmetrisers, and by convention the antisymmetrisers are 
applied first. For instance,
\[
	\young(ji,k)T_{ijk} = {\young(ji)\young(j,k)} T_{ijk}
	= (T_{ijk}-T_{ikj})+(T_{jik}-T_{jki})\,.
\]
The adjoint of a Young tableau is given by applying the symmetrisers first,
\[
  {\young(ji,k)}^*T_{ijk} = {\young(j,k)\young(ji)} T_{ijk}
  = (T_{ijk}+T_{jik})-(T_{ikj}+T_{kij})\,.
\]
For a 3-tensor $T_{ijk}$ we have the decomposition
\[
	T_{ijk}
	= \frac16\young(ijk)T_{ijk}
	+ \frac14\young(ij,k)T_{ijk}
	+ \frac14\young(ik,j)T_{ijk}
	+ \frac16\young(i,j,k)T_{ijk}.
\]
One particular $4$-index Young tableaux that we make intensive use of is the 
projector
\[
	{\young(ij,kl)}^*T_{ijkl}
	=
	\young(i,k)
	\young(j,l)
	\young(ij)
	\young(kl)
	T_{ijkl}\,,
\]
which projects onto algebraic curvature tensors.
The well known Ricci decomposition can then be written as
\[
	R_{ijkl}
	=W_{ijkl}
	+\frac1{4 (n-1)}{\young(ik,jl)}^*\mathring R_{ik}g_{jl}
	+\frac1{8n(n-1)}{\young(ik,jl)}^*g_{ik}g_{jl}.
\]
where
\[
	W_{ijkl} = \frac{1}{12}\,{\young(ik,jl)}^*_\circ\,R_{ijkl}
\]
is the Weyl curvature, $\mathring R_{ij}$ the trace-free part of the Ricci 
tensor and $R$ the scalar curvature.
The \emph{Schouten tensor} is given by
\begin{equation}\label{eqn:Schouten}
	(n-2)\schouten_{ij}
	= R_{ij}-\tfrac1{2(n-1)}Rg_{ij}
	= \mathring R_{ij} + \tfrac{n-2}{2n(n-1)}Rg_{ij}\,.
\end{equation}

\section{Conformal structure tensors}

In the present chapter superintegrable systems on 
(pseudo-)Riemannian manifolds are generalised to \emph{conformally 
superintegrable systems}. We subsequently reconsider them as systems on 
conformal geometries.
Before we begin, we need to introduce the conformal counterpart of Killing tensors.

\begin{definition}
	A (second order) \emph{conformal Killing tensor} is a symmetric tensor 
	field $C_{ij}$ on a	(pseudo-)Riemannian manifold satisfying the 
	conformal Killing equation
	\begin{equation}
		\label{eq:Killing}
		\young(ijk) C_{ij,k}
		=\young(ijk)\,g_{ij}\omega_k\,,
	\end{equation}
	where $\omega$ is a 1-form.
\end{definition}

\noindent The 1-form $\omega$ can be expressed in terms of the conformal 
Killing tensor. Indeed, contracting~\eqref{eq:Killing} in $(i,j)$, we find
\begin{equation}\label{eqn:omega.from.divC}
	\omega_k = \frac{1}{n+2}\,\lb
					2C\indices{^a_{k,a}}
					+ C\indices{^a_{a,k}}
				\rb\,.
\end{equation}

\begin{remark}\label{rmk:fracefreeness.conformal.killing.tensors}~

\noindent(i) Any Killing tensor is also a conformal Killing tensor, with 
$\omega=0$. 
In particular, the metric $g$ is trivially a conformal Killing tensor.

\noindent(ii) If $K_{ij}$ is a conformal Killing tensor, any trace 
modification $C_{ij}=K_{ij}+\lambda g_{ij}$ is also a conformal Killing 
tensor. If $K_{ij}$ is a proper Killing tensor, $C_{ij}$ is a conformal 
Killing 
tensor with $\omega=d\lambda$.
\end{remark}

We mention that while for a proper Killing tensor $K_{ij}$, the function 
$K(\mathbf p,\mathbf p)$ is preserved along geodesics, for a conformal 
Killing tensor $C_{ij}$ the function $C(\mathbf p,\mathbf p)$ is  preserved 
along null geodesics.

\subsection{Conformally superintegrable systems}\label{sec:conf.SIS}
The concept of a superintegrable system on a (pseudo-)Riemannian manifolds is 
generalised as follows.

\begin{definition}\label{def:main.notions}
~
	\begin{enumerate}
		\item
			By a \emph{conformally (maximally) superintegrable system}, we mean a Hamiltonian system
			admitting $2n-1$ functionally independent conformal integrals of 
			the motion
			$F^{(\alpha)}$,
			\begin{align}
				\label{eq:integral}
				\{F^{(\alpha)},H\}&= \omega^{(\alpha)}\,H &
				\alpha&=0,1,\ldots,2n-2,
			\end{align}
			with a function $\omega(\mathbf{p},\mathbf{q})$ polynomial in 
			momenta.
			The Hamiltonian can be assumed to be among the conformal 
			integrals. Thus by convention we set
			\[
				F^{(0)} = H\,,\quad \omega^{(0)}=0.
			\]
		\item
			A conformal integral of the motion is \emph{second order} if it 
			is of the form
			\begin{equation}
				\label{eq:quadratic}
				F^{(\alpha)}=C^{(\alpha)}+V^{(\alpha)},
			\end{equation}
			where
			\[
				C^{(\alpha)}(\mathbf p,\mathbf 
				q)=\sum_{i=1}^nC^{(\alpha)}_{ij}(\mathbf q)p^ip^j
			\]
			is quadratic in momenta and $V^{(\alpha)}=V^{(\alpha)}(\mathbf q)$
			a function depending only on positions.
			In this case, the function $\omega$ has to be linear in the 
			momenta $\mathbf{p}$.
		\item
			A conformally superintegrable system is \emph{second order} if 
			its conformal integrals $F^{(\alpha)}$ are second order and
			\begin{equation}
				\label{eq:Hamiltonian}
				H=G+V,
			\end{equation}
			where
			\[
				G(\mathbf p,\mathbf q)=\sum_{i=1}^ng_{ij}(\mathbf q)p^ip^j
			\]
			is given by the (pseudo-)Riemannian metric $g_{ij}(\mathbf q)$ on 
			the underlying
			manifold.
		\item
			We call $V$ a \emph{conformal superintegrable potential} if the Hamiltonian
			\eqref{eq:Hamiltonian} gives rise to a conformally superintegrable system.
	\end{enumerate}
\end{definition}

A function $F(\mathbf q(t),\mathbf p(t))$ that satisfies \eqref{eq:integral} 
is constant on the null locus of the Hamiltonian, since
\[
		\dot F = \{F,H\}
		= \omega\ H\,.
\]
By adding a constant $c$ to the Hamiltonian, we achieve that $F$ is constant 
on shells where the new Hamiltonian is constant and equal to $c$.
Since we are concerned exclusively with second order maximally 
superintegrable systems, we omit the terms ``second order'' and 
``maximally'' without further mentioning.

\begin{assume}\label{assumptions:tracefree}
	From now on, unless otherwise stated, we assume that the quadratic 
	parts correspond to \emph{trace-free} conformal Killing tensors, 
	except for the Hamiltonian $H=F^{(0)}$. This is no restriction, as the 
	trace-free part of any such conformal Killing tensor is itself a 
	conformal Killing tensor.
\end{assume}

\noindent For \eqref{eq:quadratic}, condition \eqref{eq:integral} splits into 
two homogeneous parts with respect to momenta. These parts are cubic 
respectively 
linear in $\mathbf p$:
\begin{subequations}
	\label{eq:1st+3rd}
	\begin{align}
		\label{eq:3rd}
		\{C^{(\alpha)},g\}&=2\omega\,G
		\\
		\label{eq:1st}
		\{C^{(\alpha)},V\}+\{V^{(\alpha)},G\}&=2\omega\,V
	\end{align}
\end{subequations}
The condition \eqref{eq:3rd} for $C(\mathbf p,\mathbf q)=C_{ij}(\mathbf 
q)p^ip^j$ is equivalent to $C_{ij}$ being a conformal Killing tensor.
The components of $\omega(\mathbf p,\mathbf q)=\omega^i(\mathbf q)p_i$ are 
given by a 1-form, also denoted by~$\omega$ in the following.
Compare \eqref{eq:1st+3rd} to the analogous 
equations~\eqref{eqn:proper.cubic.linear} for proper superintegrability.
\bigskip

The metric $g$ allows us to identify symmetric forms and endomorphisms by 
abuse of notation.
Interpreting a conformal Killing tensor in this way as an endomorphism on 
$1$-forms, Equation~\eqref{eq:1st} can be written in the form
\begin{equation}\label{eqn:dV.KdV}
	dV^{(\alpha)} = C^{(\alpha)}dV+\omega\,V\,.
\end{equation}
Its integrability condition is the \emph{Bertrand-Darboux 
condition}
\begin{subequations}\label{eqn:conformal.dKdV}
\begin{equation}\label{eqn:conformal.dKdV.abstract}
	d(C^{(\alpha)}dV)+V\,d\omega+dV\wedge\omega=0
\end{equation}
which in components reads
\begin{equation}\label{eqn:conformal.dKdV.ij}
	\young(i,j)
	\biggl( C\indices{^m_i}V_{,jm} + C\indices{^m_{i,j}}V_{,m}
	+\omega_iV_{,j} +\omega_{i,j}V  \biggr) = 0\,.
\end{equation}
\end{subequations}
This is the counterpart to condition~\eqref{eqn:Bertrand.Darboux} from proper 
superintegrability.

By virtue of the Bertrand-Darboux equation~\eqref{eqn:conformal.dKdV}, the 
potentials $V^{(\alpha)}$ for $\alpha\not=0$ are eliminated from our 
equations. As we are going to see in the following, for non-degenerate 
systems the remaining potential $V=V^{(0)}$ can be eliminated as well, 
leaving equations on the conformal Killing tensors~$C^{(\alpha)}$ alone.
In analogy to proper superintegrability we denote by $\mathcal F$ the linear 
space spanned by second order conformal integrals 
$F^{(\alpha)}=C^{(\alpha)}(\mathbf p,\mathbf p)+V^{(\alpha)}$ 
satisfying~\eqref{eq:integral}.

\begin{definition}\label{def:max.spaces}
	Let $H=g(\mathbf p,\mathbf p)+V$ with $\mathcal F=\langle 
	F^{(\alpha)}\rangle$. Analogously to 		
	Definition~\ref{def:proper.max.spaces} we introduce 
	first
	\begin{align*}
		\mathcal C
		&= \{ C : C(\mathbf p,\mathbf p)+W\in\mathcal F\ 
					\text{ for some $W$} \}
		\\
		\intertext{and then}
		\mathcal V^\text{max}
		&= \{ V : \eqref{eqn:conformal.dKdV}
					\text{ holds for every } C\in\mathcal C \}
		\\
		\mathcal C^\text{max}
		&= \{ C : \eqref{eqn:conformal.dKdV}
					\text{ holds for every } V\in\mathcal V^\text{max} \}\,.
	\end{align*}
\end{definition}

\begin{assume}\label{assumption:killing.space}~
	Unless otherwise stated, a conformally superintegrable Hamiltonian 
	will be considered together with all its conformal integrals 
	$F=C_{ij}p^ip^j+W$ where $C\in\mathcal C^\text{max}$ and where~$W$ 
	satisfies Equation~\eqref{eqn:dV.KdV}, i.e.
	\[
		dW = CdV+\omega\,V\,.
	\]
	This assumption is no restriction, and ensures a tidy exposition as 
	we do not need to specify the subspace $\mathcal C\subset\mathcal 
	C^\text{max}$ each time.
\end{assume}

\subsection{Conformal equivalence}
\label{sec:conformal.equivalence}

In analogy to Section~\ref{sec:projective.V.space.proper} we obtain a 
symmetry group of conformally superintegrable systems on (pseudo-)Riemannian 
manifolds. For the time being, we insist that the metric is preserved (up to 
coordinate transformations). In contrast to properly superintegrable systems, 
the affine group does not map potentials to potentials for conformally 
superintegrable systems, since \eqref{eqn:conformal.integral} contains the 
potential $V$ without derivatives.
The symmetry group of a conformally superintegrable system on a fixed 
(pseudo-)Riemannian manifold therefore is $\mathfrak 
S=\mathrm{Diff}(M)\rtimes\R^*$ where $\Phi=(\phi,a)\in\mathfrak S$ acts as
\[
	\Phi(g^{ij}p_ip_j+V) = \phi^*(g)^{ij}p_ip_j+a\phi^*(V)\,.
\]
This is the counterpart of the symmetry 
group~\eqref{eqn:symmetry.group.proper} for properly superintegrable systems.
Analogously to~\eqref{eqn:symmetry.structure.tensor.proper} in 
Section~\ref{sec:projective.V.space.proper}, the symmetry 
group of the structure tensor $S_{ijk}$ of a conformally superintegrable 
system is $\mathfrak S'=\mathfrak S\times\R^*$.

In the current section, St\"ackel equivalence is going to be generalised to 
conformal equivalence of conformally superintegrable systems on 
(pseudo-)Riemannian geometries.
In a next step, we then define superintegrability on conformal geometries, 
and we shall see that with this definition, the symmetry group of the 
structure tensor coincides with the symmetry group of the superintegrable 
system.

Lemma~\ref{la:Stackel.equivalent.H} is the basis for the definition of 
St\"ackel equivalence of properly superintegrable Hamiltonians, see 
Definition~\ref{def:staeckel.equivalence}.
For conformally superintegrable systems, we have the following 
generalisation, compare for example 
\cite{Kress07,Kalnins&Kress&Miller-II,Kalnins&Kress&Miller-IV,BKM1986}.

\begin{definition}\label{def:conformal.equivalence}
 Consider two second order conformally superintegrable systems with 
 Hamiltonians $H=g(\mathbf p,\mathbf p)+V$ respectively $\tilde H=\tilde 
 g(\mathbf p, \mathbf p)+\tilde V$ and conformal integrals $F^{(\alpha)}$ 
 respectively $\tilde F^{(\alpha)}$ as in Definition~\ref{def:main.notions}.
 We say that the two systems are \emph{conformally equivalent} if the 
 associated Hamiltonians $H$ and $\tilde H$ satisfy $\tilde H=\Omega^{-2}H$ 
 for some function $\Omega$, and if the trace-free parts of the associated 
 conformal Killing tensors, viewed as $(2,0)$-tensors with upper indices, 
 span the same linear space.
\end{definition}

\noindent Compare this to Definition~\ref{def:staeckel.equivalence}. 
St\"ackel equivalence is a special case of conformal equivalence, insofar as 
in~\eqref{eq:staeckel.F} the conformal factor $\Omega$ is given by a 
potential $U\in\mathcal V^\text{max}$. Regarding the St\"ackel transformation 
of the integrals, Equation~\eqref{eq:staeckel.F}, we remark the following:
In \cite[Theorem 4.1.8]{Capel_phdthesis} it is proven that if $H=g(\mathbf 
p,\mathbf p)+V$ is a Hamiltonian with conformal integral $F=C(\mathbf 
p,\mathbf p)+W$, and $U\in\mathcal V^\text{max}$, then there is a 
trace correction $\lambda$ such that $\tilde H=\frac{H}{U}$ admits the proper 
integral $\tilde F=C(\mathbf p,\mathbf p)+\lambda\,g(\mathbf p,\mathbf 
p)-\frac{VW}{U}$.
This fact provides us with a way to transform a second order conformally 
superintegrable system into a properly superintegrable system, possibly on a 
different (pseudo-)Riemannian manifold.
We conclude that St\"ackel equivalence can indeed be understood as special 
case of conformal equivalence.\smallskip

In dimensions 2 and 3 it is even proven that any second order non-degenerate 
superintegrable system on a conformally flat manifold is St\"ackel equivalent 
to a properly superintegrable system on a manifold of constant curvature; see 
\cite[Theorem 3]{Kalnins&Kress&Miller-II} and \cite[Theorem 
4]{Kalnins&Kress&Miller-IV} respectively.
St\"ackel classes of 2-dimensional second order systems are studied in 
\cite{Kress07} using properties of their associated quadratic algebras.
\medskip

\noindent As mentioned earlier, we impose trace-freeness on conformal Killing 
tensors, except for the metric, which thereby becomes a distinguished 
conformal Killing tensor.
Trace-freeness is motivated, among other reasons, by the conformal 
transformation rules for conformal Killing tensors.
Let us assume that $H,\tilde H$ are two conformally equivalent Hamiltonians. 
Let $F(\mathbf q)=C_{\mathbf q}(\mathbf p,\mathbf p)+W(\mathbf q)$ be a 
conformal integral for $H$, i.e.
\[
	\lcb H,C(\mathbf p,\mathbf p)+W\rcb=\omega(\mathbf p)H\,.
\]
Then we compute
\begin{align}
	\nonumber
	\lcb \tilde H,F\rcb
	&= \lcb \frac{H}{\Omega^2},C(\mathbf p,\mathbf p)+W\rcb
	\\ \nonumber
	&= \frac{1}{\Omega^2}\,\lcb H , C(\mathbf p,\mathbf p)+W\rcb
		-\frac{2}{\Omega^3}\,\lcb \Omega,C(\mathbf p,\mathbf p)\rcb\,H
	\\ \label{eqn:computation.F.invariant}
	&=	\big(
			\omega(\mathbf p) - 2C(\mathbf p,d\Upsilon)
		\big)\,\tilde H
\end{align}
with $\Upsilon=\ln\Omega$.
The following proposition is thus obtained, recalling 
Assumption~\ref{assumption:killing.space}, which ensures that the spaces 
$\mathcal F,\tilde{\mathcal F}$ in the claim are maximal.

\begin{proposition}\label{prop:trafo.conformal.Killing.and.integrals}
	Let $g$ and $\tilde g=\Omega^2g$ be a pair of conformally equivalent 
	metrics, $\Omega>0$.
	
	\noindent(i)
	If $C_{ij}$ is a trace-free conformal Killing tensor for~$g$ then $\tilde 
	C_{ij}=\Omega^4C_{ij}$ is a trace-free conformal Killing tensor for 
	$\tilde g$.
	
	\noindent(ii)
	If $(g,V,\mathcal F)$ and $(\tilde g,\tilde V,\tilde{\mathcal F})$ are 
	conformally equivalent second order conformally superintegrable systems, 
	where for $F=C(\mathbf p,\mathbf p)+W\in\mathcal F$ with $F\not\propto 
	H$, the coefficients $C_{ij}$ are the components of a \emph{trace-free} 
	conformal Killing tensor. Then $\tilde{\mathcal F}=\mathcal F$.
\end{proposition}
\noindent Note that $\tilde C_{ij}$ is trace-free with respect to $\tilde g$, 
but also with respect to $g$, since $g$ and $\tilde g$ are conformally 
equivalent.
\begin{proof}
	For part (i), we have the equation 
	${\text{\tiny\young(ijk)}(C_{ij,k}-g_{ij}\omega_k)=0}$ for the conformal 
	Killing tensor~$C$.
	A straightforward computation then shows that
	\[
		\young(ijk)\,\left(
						\tilde C_{ij,k}
						-\tilde g_{ij}\tilde\omega_k
					\right)=0
	\]
	where
	\begin{equation*}
		\tilde C_{ij} := \Omega^4C_{ij}\,, \qquad
		\tilde\omega_i := \Omega^2\,(\omega_i-2C_{ia}\Upsilon^{,a})\,,
	\end{equation*}
	and $\Upsilon=\ln\Omega$.
	Part~(ii) then follows immediately 
	from~\eqref{eqn:computation.F.invariant} in the light of 
	Assumption~\ref{assumption:killing.space}.
\end{proof}

Proposition~\ref{prop:trafo.conformal.Killing.and.integrals} yields the 
transformation rules under conformal equivalence: The Hamiltonian is 
conformally modified, but the space of integrals is preserved.
We can encode this in an efficient manner using weighted tensor densities.
A conformal density $\delta$ of weight $w$ is a section in 
$\mathcal{E}[w]:=S^2T^*M\otimes(\Lambda^nTM)^{-w/n}$ such that 
$\phi^*(\delta)=\Omega^w\delta$ if $\phi\in\mathrm{Conf}(M)$ is a conformal 
transformation, i.e.\ $\phi^*(g)=\Omega^2g$.
Given the data $(g,V,\mathcal F)$, we observe that 
$g\in\mathcal{E}_{(0,2)}[2]$ and 
$V\in\mathcal{E}[-2]$ are weighted densities. We are thus able to define 
conformally invariant densities of weight $0$,
$\mathbf g\in\mathcal{E}_{(0,2)}[0]$ and
$\mathbf v\in\mathcal{E}[0]=\mathcal{C}^\infty(M)$,
given in local components by
\[
	\mathbf{g}_{ij} = \frac{g_{ij}}{|\det(g)|^{\tfrac1n}}\,,
\]
compare~\cite{curry_2014}, and
\[
	\mathbf{v} = |\det(g)|^{\tfrac1n}V\,,
\]
respectively.
One straightforwardly verifies that $\phi^*(\mathbf{g})=\mathbf{g}$ and 
$\phi^*(\mathbf{v})=\mathbf{v}$.
This leads us to the following definition:
\begin{definition}\label{def:c-superintegrable}
	Let $M$ be a conformal manifold.
	A \emph{second order c-superintegrable system} on $M$ is given by a 
	conformally invariant function $\mathbf{v}$ and a maximal linear space 
	$\mathcal F$ of invariant scalar functions $F:T^*M\to\R$ on $M$ such that
	\begin{enumerate}
		\item If $F\in\mathcal F$, then $F=C(\mathbf q)^{ij}p_ip_j+W(\mathbf 
		q)$ where $C^{ij}$ are the components of a trace-free conformal 
		Killing tensor.
		\item There is a density $\Omega\in\mathcal{E}[1]$ of weight~1 such 
		that $V=\Omega^2\mathbf{v}$ satisfies the Bertrand-Darboux 
		condition~\eqref{eqn:conformal.dKdV} for all $F\in\mathcal F$.
	\end{enumerate}
\end{definition}

\noindent Note that the definition is conformally invariant, and for any 
$\Omega\in\mathcal{E}[1]$ we have that 
$(\Omega^{-2}\mathbf{g},\Omega^{2}\mathbf{v},\mathcal F)$ is a conformally 
superintegrable system.

\begin{remark}\label{rmk:symmetry.group.cSIS}
	Let us consider the symmetry group of c-superintegrable systems.
	It is given by $\mathfrak S
	=\mathrm{Conf}(M)\rtimes\R^*$, where $\R^*$ acts as
	\[
		(\mathbf{v},\mathcal F)\mapsto\lb a\mathbf{v},\mathcal F\rb
	\]
	It contains, in particular, all diffeomorphisms, and the transformations 
	having a constant conformal factor. We remark that two different 
	conformal 
	factors do not necessarily result in different geometries. For instance, 
	if $g$ is a Euclidean metric in dimension two, then any conformally 
	equivalent metric~$\Omega^2g$ with
	\[
			\lb\D{\Omega}{x}\rb^2
			+ \lb\D{\Omega}{y}\rb^2
			=
			\Omega\D{^2\Omega}{x^2}
			+ \Omega\D{^2\Omega}{y^2}
	\]
	will also be Euclidean.
\end{remark}

\subsection{Structure tensors of a conformally superintegrable system}

We consider the Bertrand-Darboux equation~\eqref{eqn:conformal.dKdV}. It is a 
compatibility condition for 
the potential and the space of conformal Killing tensors associated to a 
second order superintegrable systems.
Our aim is to solve the overdetermined system of linear 
equations~\eqref{eqn:conformal.dKdV} for the highest derivatives of 
the potential $V$.
Following the analogous discussion in~\cite{Kress&Schoebel&Vollmer}, we use a 
generalisation of Cramer's rule.

\begin{definition}
	On an inner product space, the \emph{Gram Coefficients} $G_k(A)$ of a 
	linear map $A$ are the coefficients of the polynomial
	\[
	\det(1+tAA^*)=\sum_{k=0}^\infty G_k(A)t^k\,,
	\]
	where $A^*$ is the adjoint of $A$.
\end{definition}

Up to sign and order, the $G_k(A)$ are the coefficients of the 
characteristic polynomial of $AA^*$.
The following proposition provides us with an explicit expression for the 
structure tensor, which we will then apply to irreducible systems.

\begin{proposition}\cite{DTGVL}
	\label{prop:Moore-Penrose}
	A linear map~$A$ on an inner product space has rank~$r$ if and only if
	\begin{equation}
		\label{eq:rank}
		G_r(A)\not=0=G_{r+1}(A).
	\end{equation}
	In this case, the system of linear equations
	\[
	Ax=b
	\]
	has a solution $x$ if and only if the augmented matrix $(A|b)$ satisfies
	\[
		G_{r+1}(A|b)=0.
	\]
	Moreover, the minimal norm solution is given by
	\[
	x=A^\dagger b,
	\]
	where
	\begin{equation}
		\label{eq:Moore-Penrose}
		A^\dagger=\frac1{G_r(A)}\sum_{k=1}^rG_{r-k}(A)(-A^*A)^{k-1}A^*.
	\end{equation}
	is the \emph{Moore-Penrose inverse} of $A$.
\end{proposition}

\noindent We consider the Bertrand-Darboux 
equation~\eqref{eqn:conformal.dKdV}, when 
written in local coordinates, as a linear system
\begin{equation}\label{eq:Ax=b1y+b0V}
	AX=b_1(dV)+b_0V,
\end{equation}
where $X$ is a vector that contains the unknown components of the trace-free 
Hessian
\[
	\mathring\nabla^2_{ij}V = \nabla^2_{,ij}V-\frac1n\,g_{ij}\Delta V.
\]
The coefficient matrix $A$ contains the components of the Killing tensors 
$C^{(\alpha)}$,\ \ $\alpha=0,1,\ldots,2(n-1)$.
On the right hand side, $b_1$ is a vector of terms involving the components 
of the differential $dV$ of~$V$, i.e.\ $b_1(dV)$ comprises the components of 
the second and third term in the Bertrand-Darboux equation for each 
$C^{(\alpha)}$.
Likewise, $b_0$ is a column vector such that $b_0V$ contains the components 
of the fourth term in~\eqref{eqn:conformal.dKdV.ij}.
Note that the rank~$r$ for~\eqref{eq:Ax=b1y+b0V} does not need to be constant.
If the conformal Killing tensors are analytic, however, then so are the 
components of the matrix~$A$. Consequently the Gram coefficients $G_k(A)$ are 
also analytic.
Proposition~\ref{prop:Moore-Penrose} ensures that in such case the rank of 
$A$ is constant on an open and dense subset.

Provided $r$ is constant, we can express the trace-free Hessian of $V$ using 
the Moore-Penrose inverse, by writing
\begin{equation}\label{eq:Ax=b1y+b0V.solution}
	X = A^\dagger b_1(dV)+A^\dagger b_0V.
\end{equation}

\begin{definition}
	A conformally superintegrable system on a (pseudo-)Riemannian manifold 
	$M$ has \emph{rank} $r$, if the coefficient matrix $A$ in
	\eqref{eq:Ax=b1y+b0V} has rank $r$ on an open and dense subset of $M$.
\end{definition}

\begin{remark}
	The rank of a conformally superintegrable system is at most
	\begin{equation}\label{eq:rank:max}
		r_\text{max}=\frac{n(n+1)}2-1=\frac{(n-1)(n+2)}2.
	\end{equation}
	As in the case of properly superintegrable systems 
	\cite{Kress&Schoebel&Vollmer}, we can characterise systems of maximal 
	rank in terms of their trace-free conformal Killing tensors.
	As before, the metric allows us to identify	(trace-free) bilinear forms 
	and (trace-free) endomorphisms on the tangent space. This 
	identification will tacitly be made in the following.
	
	Due to Proposition~\ref{prop:trafo.conformal.Killing.and.integrals}, any 
	conformally superintegrable system that is conformally equivalent 
	to a maximal rank conformally superintegrable system is itself of 
	maximal rank.
\end{remark}

\begin{definition}\label{def:irreducible}~
	\begin{enumerate}
		\item
		A set of endomorphisms is \emph{irreducible} if they do not have a
		non-trivial invariant subspace in common.
		\item
		A set of endomorphism fields on a (pseudo-)Riemannian manifold $M$ is
		called \emph{irreducible}, if they are pointwise irreducible
		on an open and dense subset of~$M$.
		\item
		We call a conformally superintegrable system \emph{irreducible}, if 
		its conformal Killing tensors form an irreducible set.
		\item
		We call a c-superintegrable system \emph{irreducible}, if 
		its conformal Killing tensors form an irreducible set.
	\end{enumerate}
\end{definition}

\noindent The next lemma follows analogously to the corresponding statement 
in the case of properly superintegrable systems \cite{Kress&Schoebel&Vollmer}.
\begin{lemma} 
	\label{lemma:irreducible}
	A conformally superintegrable system has maximal rank if and only if it 
	is irreducible.
\end{lemma}

\noindent Irreducibility thus ensures that we can solve~\eqref{eq:Ax=b1y+b0V} 
for $X$. In particular we find the minimal-norm 
solution~\eqref{eq:Ax=b1y+b0V.solution}, to which we may add any element from 
the kernel of $A$.
For second order conformally superintegrable systems, the Bertrand-Darboux 
equation~\eqref{eqn:conformal.dKdV} can therefore be solved assuming 
irreducibility. 
We have the requirement that the trace of the Hessian of $V$ is the Laplacian 
of $V$ and thus the potential $V$ of an irreducible conformally 
superintegrable system satisfies
\begin{equation}
	\label{eqn:Wilczynski.conformal}
	V_{,ij}=T\indices{_{ij}^m}V_{,m}+\tfrac1ng_{ij}\Delta V+\tau_{ij}V
\end{equation}
with a (not necessarily unique) $(2,1)$-tensor $T_{ijk}$ and a $(2,0)$-tensor 
$\tau_{ij}$.
We refer to~\eqref{eqn:Wilczynski.conformal} as \emph{Wilczynski equation} 
because our methods are inspired by Wilczynski's series of papers on the 
projective differential geometry of 
surfaces~\cite{Wilczynski_I,Wilczynski_IV}.
Equations similar to \eqref{eqn:Wilczynski.conformal} appear in
\cite{Kalnins&Kress&Miller-III} in local coordinates and for dimension three.
The tensors $T$ and $\tau$ necessarily satisfy the following symmetries
\begin{subequations}\label{eqn:relations.T.tau}
	\begin{align}
	\young(i,j)\,\lb T\indices{_{ij}^m}V_{,m} +\tau_{ij}V \rb &= 0
	\\
	g^{ij}\,\lb T\indices{_{ij}^m}V_{,m} +\tau_{ij}V \rb &= 0\,.
	\end{align}
\end{subequations}
We call $T_{ijk}$ the \emph{primary structure tensor} and $\tau_{ij}$ the 
\emph{secondary structure tensor} of the conformally superintegrable system. 
Note that these tensors are not invariant under conformal transformations.

The analog of~\eqref{eqn:Wilczynski.conformal} in proper superintegrability 
is Equation~\eqref{eqn:Wilczynski.proper}, which formally 
coincides with~\eqref{eqn:Wilczynski.conformal} for $\tau_{ij}=0$. However, 
note that~\eqref{eqn:Wilczynski.proper} was obtained from 
\eqref{eqn:Bertrand.Darboux}, where $K$ is a proper Killing tensor. 
Instead, the Wilczynski equation~\eqref{eqn:Wilczynski.conformal} is obtained 
from the Bertrand-Darboux equation~\eqref{eqn:conformal.dKdV}, where 
trace-free conformal Killing tensors appear.
In spite of this difference, the following lemma shows that 
vanishing $\tau$ indeed follows from proper superintegrability. We are going 
to see below in Corollary~\ref{cor:tau0.then.proper} that, for non-degenerate 
second order conformally superintegrable systems, the converse holds as well.
\begin{lemma}\label{la:proper.then.tau0}
 Consider a second order superintegrable system with potential $V$ and 
 associated proper Killing tensors $K^{(\alpha)}$. Let 
 $C^{(\alpha)}=\mathring K^{(\alpha)}$ and assume they 
 satisfy the Wilczynski equation~\eqref{eqn:Wilczynski.conformal} with $V$. 
 Then $\tau_{ij}=0$.
\end{lemma}
\begin{proof}
 We choose a specific $\alpha$ and suppress the corresponding superscript for 
 conciseness.
 By hypothesis, there is a function $\lambda$ such that 
 \begin{equation}\label{eqn:K=C+lambda.g}
 	K_{ij}=C_{ij}+\frac1n\lambda g_{ij}
 \end{equation}
 satisfies~\eqref{eqn:Bertrand.Darboux}.
 
 The proper Killing tensor $K_{ij}$ satisfies the Killing equation 
 ${\young(ijk)}K_{ij,k}=0$. 
 Substituting~\eqref{eqn:K=C+lambda.g} into the Killing equation, and then 
 using the conformal Killing equation 
 ${\young(ijk)}(C_{ij,k}-g_{ij}\omega_k)=0$, we find
 $\lambda_{,k}=-n\omega_k$.
 We conclude that $\omega$ is exact, $d\omega=0$.
 It then follows that $\tau_{ij}=0$, as~\eqref{eqn:conformal.dKdV} does not 
 have a term involving $V$ without derivative.
\end{proof}

\subsection{Non-degenerate and abundant systems}

In the previous section, it was shown that for irreducible second order 
superintegrable systems, the Bertrand-Darboux 
equation~\eqref{eqn:conformal.dKdV} can be solved for the second derivatives 
of $V$ up to the Laplacian $\Delta 
V$.
The Wilczynski equation~\eqref{eqn:Wilczynski.conformal} then allows one to 
express all higher covariant derivatives of $V$ linearly in $V$, $\nabla V$ 
and $\Delta V$.  
Hence all higher derivatives of $V$ at some point are determined by the value 
of $(V,\nabla V,\Delta V)$ at this point, i.e. by $n+2$ constants.  This 
motivates the following definition.

\begin{definition}
	\label{def:non-degenerate}
	A conformally superintegrable system is called \emph{non-degenerate} if 
	it satisfies the Wilczynski condition~\eqref{eqn:Wilczynski.conformal}, 
	and if \eqref{eqn:Wilczynski.conformal} admits an $(n+2)$-dimensional 
	space of solutions~$V$.\footnote{Note that for an analytic 
	metric, the (trace-free) conformal Killing tensors are analytic, and thus 
	the structure tensors and the potentials are also analytic.}
\end{definition}

\noindent
Due to~\eqref{eqn:relations.T.tau}, the structure tensors satisfy the 
following symmetry conditions for a non-degenerate potential:
\begin{align*}
	T\indices{_{ji}^m}&=T\indices{_{ij}^m}
	&
	g^{ij}T\indices{_{ij}^m}&=0
	\\
	\tau_{ij}&=\tau_{ji}
	&
	g^{ij}\tau_{ij}&=0\,.
\end{align*}

\begin{lemma}
	For a non-degenerate conformally superintegrable system the structure 
	tensors $T_{ijk}$ and $\tau_{ij}$ are unique.
\end{lemma}
\begin{proof}
	Assume that the Wilczynski condition~\eqref{eqn:Wilczynski.conformal} 
	were satisfied for two different pairs of structure tensors, say 
	$T_{ijk},\tau_{ij}$ and $\tilde T_{ijk},\tilde\tau_{ij}$ respectively.
	Then, subtracting the corresponding copies 
	of~\eqref{eqn:Wilczynski.conformal},
	\[
	0 = ( T\indices{_{ij}^k}-\tilde T\indices{_{ij}^k} )\,V_{,k}
	+ ( \tau_{ij}-\tilde\tau_{ij} )\,V\,.
	\]
	By the hypothesis of non-degeneracy, the coefficients of $V_{,k}$ and $V$ 
	have to vanish independently, i.e.\ $T\indices{_{ij}^k} = \tilde 
	T\indices{_{ij}^k}$ and $\tau_{ij} = \tilde\tau_{ij}$.
\end{proof}

\begin{example}
	\label{ex:IHO}
	The isotropic harmonic oscillator is an irreducible system in the sense 
	of Definition~\ref{def:irreducible} and has the potentials
	\[
		V(\mathbf x)=\frac{\omega^2}2(\mathbf x-\mathbf x_0)^2+V_0
	\]
	with the $n+2$ free parameters $\omega^2$, $\mathbf x_0$ and $V_0$. 
	$V(\mathbf x)$ solves the Wilczynski 
	equation~\eqref{eqn:Wilczynski.conformal}.
	The \emph{isotropic harmonic oscillator} on flat $n$-dimensional space 
	has vanishing structure tensor~$T$. It is properly superintegrable and 
	therefore also the structure tensor $\tau_{ij}$ vanishes.
	Any conformally superintegrable system conformally equivalent to the 
	isotropic harmonic oscillator is characterised by $\mathring T_{ijk}=0$, 
	and we obtain~\eqref{eqn:Wilczynski.conformal} in the form
	\begin{equation}
		V_{,ij}=\frac{
					n\,t_i V_{,j}
					+ n\,t_j V_{,i}
					- 2\,g_{ij} t^m V_{,m}
				}{(n-1)(n+2)}
				+ \frac1n\,g_{ij}\Delta V+\tau_{ij}V\,.
	\end{equation}
	We are going to see in \eqref{eqn:tau} that
	\[
		\tau_{ij}
		= \frac23\,\left(\frac{n}{(n-1)(n+2)}\right)^2\,(t_i t_j)_\circ
		- 2\,\mathring{\schouten}_{ij},
	\]
	where $\mathring\schouten_{ij}$ is the trace-free part of the Schouten 
	tensor.
\end{example}

Note that non-degeneracy is conformally invariant, and thus can be defined 
for a c-superintegrable system.
\begin{definition}
	We call a c-superintegrable system \emph{non-degenerate} if one (and 
	hence all) members of the corresponding equivalence class are 
	non-degenerate in the sense of Definition~\ref{def:non-degenerate}.
\end{definition}

\noindent From now on we restrict to non-degenerate systems that 
satisfy the Wilczynski equation~\eqref{eqn:Wilczynski.conformal}.
In the present section, our aim is to formulate and study the integrability 
conditions imposed by \eqref{eqn:Wilczynski.conformal} and the non-degeneracy 
condition onto the structure tensors of a conformally superintegrable system.

By definition, a superintegrable system on an $n$-dimensional 
(pseudo-)Riemannian manifold has the maximal number of $2n-1$ 
\emph{functionally} independent integrals.
However, as for properly superintegrable systems, all known non-degenerate 
second order systems also admit maximally many \emph{linearly} independent 
conformal Killing tensors.
In analogy to properly superintegrable systems, we therefore 
define:
\begin{definition}\label{def:abundant.system}
We call a non-degenerate second order conformal superintegrable system in
dimension~$n$ \emph{abundant}, if the subspace
\[
	\mathring{\mathcal C}=\{\mathring{C} : C\in\mathcal C^\text{max} : 
	\mathrm{tr}(C)=0 \}
\]
has dimension
\[
	\dim(\mathring{\mathcal C})
	= \tfrac{n(n+1)}{2}-1
	= \frac{(n-1)(n+2)}{2} = r_\text{max}.
\]
\end{definition}

\noindent By virtue of 
Proposition~\ref{prop:trafo.conformal.Killing.and.integrals}, 
abundantness is a conformally invariant property and hence a property of a 
c-superintegrable system.
Moreover, it is manifest that if a conformally superintegrable system is 
abundant, then a properly superintegrable system that is conformally 
equivalent to it is also abundant in the sense 
of~\cite{Kress&Schoebel&Vollmer}.

Definition~\ref{def:abundant.system} is going to become natural in 
Section~\ref{sec:abundant}, where we find that it is tantamount with the 
generic integrability of the prolongation equations for the trace-free 
conformal Killing tensors arising from a second order conformally 
superintegrable system. At the end of the current section we shall also give 
an equivalent characterisation of abundantness relying on a relaxation of 
linear independence.

We remark that Definition~\ref{def:abundant.system} generalises the concept 
of abundantness introduced in Section~\ref{sec:preliminaries}.
Abundantness is trivial in dimension $n=2$.
In dimension $n=3$, the situation is a bit more involved. It is still true 
that every non-degenerate second order (properly) superintegrable system on a 
conformally flat manifold extends from $2n-1=5$ functionally independent 
integrals to $\tfrac{n(n+1)}{2}=6$ linearly independent integrals 
\cite{Kalnins&Kress&Miller-III}. This is known as the 
``$(5\Rightarrow6)$-Lemma''. Furthermore, every 3-dimensional conformal 
system is St\"ackel equivalent to a proper one with constant curvature 
\cite{Capel_phdthesis}.

\begin{remark}
	For the sake of expositional simplicity, we require an abundant system to 
	have $\tfrac{n(n+1)}{2}-1$ \emph{linearly} independent trace-free 
	conformal Killing tensors. Functional independence of $2n-1$ of the 
	arising conformal integrals is not yet required in the definition, but in 
	Lemma~\ref{la:functional.dependence} we prove that systems 
	which do admit $2n-1$ \emph{functionally} independent conformal integrals 
	lie dense among abundant systems.
\end{remark}

We devote the remainder of this paragraph to an alternative characterisation 
of abundantness.

\begin{definition}
 We say that a collection of linearly independent Killing tensors 
 $K^{(\alpha)}_{ij}$ is \emph{conformally linearly independent} if
 \begin{equation}\label{eqn:independence.criterion}
   \sum_\alpha c_\alpha K^{(\alpha)} = f\,g\,,
 \end{equation}
 with constants $c_\alpha$ and a function $f$, implies $f=0$ and $c_\alpha=0$ 
 for all $\alpha$.
\end{definition}

The following lemma ensures that a conformally superintegrable system is 
abundant if the conformal Killing tensors $C^{(\alpha)}$ for 
$\alpha\in\{1,2,\dots,\tfrac{n(n+1)}{2}\}$ in 
Definition~\ref{def:main.notions} 
are conformally linearly independent.

\begin{lemma}
 Let $(K^{(1)},\dots,K^{(m)})$ be a collection of Killing tensors 
 and denote the corresponding trace-free conformal Killing 
 tensors by $C^{(\alpha)}=K^{(\alpha)}-\frac1n\,\mathrm{tr}(K^{(\alpha)})\,g$.
 Then the tupel $(g,C^{(1)},\dots,C^{(m)})$ is linearly independent if any 
 only if 
 $(K^{(1)},\dots,K^{(m)})$ is conformally linearly independent.
\end{lemma}
\begin{proof}
 Assume first that $(g,C^{(1)},\dots,C^{(m)})$ is linearly dependent. This 
 means there is a combination
 \[
  \sum c_\alpha C^{(\alpha)} = c_0\,g
 \]
 with constants $(c_0,c_1,\dots,c_m)\ne0$.
 This means
 \[
  \sum c_\alpha K^{(\alpha)} = f\,g
 \]
 for a function $f$ obtained from $c_0$ and the trace terms. This proves one 
 implication.
 For the other implication, we assume that $(K^{(1)},\dots,K^{(m)})$ are 
 conformally linearly dependent. Therefore we have
 \[
 	\sum_{\alpha=1}^m c_{(\alpha)}C^{(\alpha)}
 	= \left(
 			f+\frac1n\,\sum c_{(\alpha)}\mathrm{tr}(K^{(\alpha)})
 	 \right)\,g\,.
 \]
 In this equation the left hand side and the right hand side have to vanish 
 independently, as they are trace-free respectively pure trace.
 We conclude
 \[
 	\sum c_{\alpha}C^{(\alpha)} = 0\,,\qquad
 	f=-\frac1n\,\sum c_{\alpha}\mathrm{tr}(K^{(\alpha)})
 \]
 Because of the hypothesis of conformal linear dependence, we have 
 $(c_1,\dots,c_m)\ne0$. This implies that $(g,C^{(1)},\dots,C^{(m)})$ are 
 linearly dependent.
\end{proof}

Moreover we observe: If the integrals 
$F^{(\alpha)}=C^{ij}_{(\alpha)}p_ip_j+W^{(\alpha)}\in\mathcal F$ are 
functionally independent, then their kinetic parts are associated to 
conformally linearly independent conformal Killing tensors.
This follows from Theorem~1 of~\cite{Kalnins&Kress&Miller-III}.

%

\subsection{Structure tensors and c-superintegrability}
We are now going to determine how the structure tensors behave under 
conformal changes of the superintegrable system.
\begin{lemma}\label{la:Stackel.transformation.rules}
Let $H=g(\mathbf p,\mathbf p)+V$ and $\tilde{H}=\Omega^{-2}\,H$ be a pair of 
conformally equivalent Hamiltonians, $\Omega>0$.
Assume $H$ gives rise to an irreducible conformally superintegrable system 
such that the Wilczynski equation~\eqref{eqn:Wilczynski.conformal} is 
satisfied.
Then $\tilde{H}$ satisfies~\eqref{eqn:Wilczynski.conformal} as well. (We 
decorate the corresponding objects with a tilde).
In particular, the structure tensors $T_{ijk}$ and $\tau_{ij}$ are 
transformed to, respectively,
\begin{subequations}\label{eqn:transformation.rules}
\begin{align}
 \tilde{T}\indices{_{ij}^k}
 &= T\indices{_{ij}^k}-3\,{\young(ij)}_\circ \Upsilon_{,i} g\indices{_j^k}
 \label{eqn:transformation.T}
 \\
 \tilde{\tau}_{ij}
 &= \tau_{ij} + 2\,T\indices{_{ij}^k}\Upsilon_{,k}
    -{\young(ij)}_\circ \left( \Upsilon_{,ij} +2\Upsilon_{,i}\Upsilon_{,j} \right)
 \label{eqn:transformation.tau}\,,
\end{align}
\end{subequations}
where $\Upsilon=\ln\Omega$.
\end{lemma}
\begin{proof}
By a straightforward computation, using the Wilczynski 
equation~\eqref{eqn:Wilczynski.conformal} and the product rule, we arrive at
\begin{align*}
 \hat{\nabla}^2_{ij}\hat{V}
 &=\frac1n\,\hat{g}_{ij}\hat{\Delta}\hat{V}
    + \left(
 		T\indices{_{ij}^k}
 		-3\,\left(\Upsilon_i\,g_j^k +\Upsilon_j\,g_i^k \right)
 		-\tfrac6n\,g_{ij}\,\Upsilon^{,k}
      \right)\,\hat{V}_k \\
 &\quad
 +\left( \tau_{ij} +2\,T\indices{_{ij}^k}\,\Upsilon_{,k}
		-2\,\tfrac{ \nabla^2_{ij}\Omega }{\Omega}
		+\tfrac2n\,g_{ij}\,\tfrac{\Delta\Omega}{\Omega}
		-2\,\Upsilon_{,i}\Upsilon_{,j}
		+\tfrac2n\,g_{ij}\,\Upsilon^{,a}\Upsilon_{,a}
 		\right)\,\hat{V}\,.
\end{align*}
The result then simplifies further using
$\Omega^{-1}\nabla^2_{ij}\Omega = \Upsilon_{,i}\Upsilon_{,j} +\frac12\,\young(ij)\Upsilon_{,ij}$.
\end{proof}

\noindent According to the Wilczynski 
equation~\eqref{eqn:Wilczynski.conformal}, the 
structure tensor $T_{ijk}$ of a conformally superintegrable system determines 
the conformally superintegrable potential up to the action of 
$\mathfrak{S}=\mathrm{Diff}(M)\rtimes\R^*$.
The following corollary is straightforwardly obtained, but will be 
fundamental.

\begin{corollary}
Let the hypothesis and notation be as in 
Lemma~\ref{la:Stackel.transformation.rules}.

\noindent (i) Under conformal transformations, the trace $t_i = 
T\indices{_{ia}^a}=T\indices{_{ai}^a}$ undergoes a translation by 
$\Upsilon_{,i}$,
\begin{equation}\label{eqn:trafo.t.i}
	\tilde t_i
	= t_i - \frac3n\,(n-1)(n+2)\,\Upsilon_{,i}\,.
\end{equation}

\noindent (ii) Under conformal transformations, the trace 
$T\indices{^a_{ai}}$ remains unchanged,
\begin{equation*} 
	\tilde T\indices{^a_{ai}} = T\indices{^a_{ai}}\,.
\end{equation*}

\noindent (iii) Under conformal transformations, the trace-free part of the 
primary structure tensor remains unchanged,
\begin{equation}\label{eqn:trafo.S.ijk}
	\mathring{\tilde{T}}\indices{_{ij}^k} = \mathring T\indices{_{ij}^k}\,.
\end{equation}
\end{corollary}

\begin{remark}
	Although $\mathring T\indices{_{ij}^k}$ is conformally invariant, it is 
	often advantageous to work with the tensor $\mathring T_{ijk}$ instead, 
	which in the context below turns out to be a totally symmetric tensor 
	field.
	While $\mathring T_{ijk}$ is not actually invariant, under conformal 
	transformations it is just multiplied with a power of $\Omega$,
	\[
		\mathring{\tilde{T}}_{ijk}
		= \Omega^2 g_{ka}\mathring{\tilde{T}}\indices{_{ij}^a}
		= \Omega^2\mathring T_{ijk}\,,
	\]
	i.e.\ it has weight $2$.
\end{remark}
\bigskip

\noindent According to \cite{Kress&Schoebel&Vollmer}, $\mathring 
T\indices{_{ij}^k}$ is closely related to the Weyl curvature.
Under mild assumptions we shall find below that $\mathring 
T\indices{_{ij}^k}$ carries enough information to reconstruct the conformal 
equivalence class of a (conformally) superintegrable system.

In view of properly superintegrable systems, at first sight it might seem 
that the conformal case would require additional information, in the form of 
an additional structure tensor $\tau_{ij}$.
We will find, however, that for abundant systems $\tau_{ij}$ is determined by 
$T_{ijk}$ and the Ricci curvature.
We summarise the discussion in this section with the following proposition.

\begin{proposition}
	\label{prop:Wilczynski}
	On a (pseudo-)Riemannian manifold $M$, every non-degenerate irreducible 
	conformally superintegrable system admits tensor fields $T$ and $\tau$ 
	with the following properties:
	\begin{enumerate}
		\item
		$T$ is well-defined and smooth on the open and dense subset 
		$N=\{G_{r_\text{max}}(A)\ne0\}\subseteq M$.
		It is of valence~$3$, symmetric and trace-free in its first
		two indices:
		\begin{align}
			\label{eq:T:symmetries}
			T_{jik}&=T_{ijk}&
			g^{ij}T_{ijk}&=0
		\end{align}
		Moreover, $\tau$ is well-defined and smooth on $N$. It is of 
		valence~$2$, symmetric and trace-free:
		\begin{align}
			\label{eq:tau:symmetries}
			\tau_{ji}&=\tau_{ij}&
			g^{ij}\tau_{ij}&=0
		\end{align}
		\item
		The conformally superintegrable potential 
		satisfies the Wilczynski equation~\eqref{eqn:Wilczynski.conformal}.
		\item
		$T$ and $\tau$ are uniquely determined by the metric and by the 
		trace-free conformal Killing tensors $C^{(\alpha)}$ in the 
		conformally superintegrable system, and depend only on the subspace 
		$\mathcal C$ spanned by these $C^{(\alpha)}$, i.e. it is invariant 
		under basis changes on $\mathcal C$. Note that such a change of 
		basis is conformally invariant.
		\item\label{item:transformation}
		The trace-free part $\mathring T\indices{_{ij}^k}$ of the 
		$(2,1)$-tensor field $T\indices{_{ij}^k}$ is 
		invariant under conformal changes of the conformally superintegrable 
		system.
	\end{enumerate}
	The components $T_{ijk}$ of $T$ are given explicitly in terms of the
	Killing tensors by the rank-$r$ Moore-Penrose inverse, where
	$r=r_\text{max}$ is the maximal rank \eqref{eq:rank:max}.
\end{proposition}
\begin{proof}
	The first three assertions follow analogous to the case of properly 
	superintegrable systems, see \cite{Kress&Schoebel&Vollmer}, such that the 
	tensors $T$ and $\tau$ are given by $A^\dagger b_1$ and $A^\dagger b_0$, 
	respectively, using Equation~\eqref{eq:Ax=b1y+b0V}.
	To see~\ref{item:transformation}, take the trace-free part 
	of~\eqref{eqn:transformation.T}.
\end{proof}

Let us reconsider the aforesaid in the light of 
Definition~\ref{def:c-superintegrable}.
The corollary below is easily obtained using the transformation 
rules~\eqref{eqn:transformation.rules} and recalling
the following transformation rules for the curvature.
\begin{remark}\label{rmk:trafo.curvature}
	The following transformation rules can be found in 
	\cite{curry_2014,Kulkarni69,Kulkarni70}, for instance.
	The Weyl tensor is conformally invariant.
	The trace-free Ricci tensor and the scalar curvature transform, 
	respectively, according to
	\begin{subequations}\label{eqn:trafo.Ricci.curvature}
		\begin{align}
			\mathring{R}_{ij}
			&\to \mathring{R}_{ij} - (n-2)
			\left(
			\Upsilon_{,ij}-\Upsilon_{,i}\Upsilon_{,j}
			\right)_\circ
			\\
			R
			&\to \Omega^{-2} \bigg( R+2(n-1)\left(
			\Delta\Upsilon-\tfrac{n-2}{2}\,\Upsilon^{,a}\Upsilon_{,a}
			\right)\bigg)\,.
		\end{align}
	\end{subequations}
	Therefore the Schouten tensor transforms according to
	\begin{equation}\label{eqn:transformation.schouten}
		\schouten_{ij}
		\to \schouten_{ij} - \Upsilon_{,ij} + \Upsilon_{,i}\Upsilon_{,j}
		- \frac12\,g_{ij}\,\Upsilon_{,c}\Upsilon^{,c}\,.
	\end{equation}
\end{remark}

\begin{corollary}~
	
	\begin{enumerate}
	\item
		Every non-degenerate irreducible c-superintegrable system on a 
		conformal manifold $(M,c=[g])$ admits a well-defined totally 
		symmetric and trace-free tensor field $S=\mathring T$ that is 
		invariant under conformal transformations.
	\item\label{item:equal.S.equal.CCSIS}
		If this tensor coincides for two non-degenerate irreducible 
		c-superintegrable systems on the same conformal manifold, then these 
		systems are conformally equivalent if also
		\[
			\tau_{ij}
			-\frac23\,S_{ijk}\bar t^k
			-2\,\mathring\schouten_{ij}
			+\frac13\,{\young(ij)}_\circ\bar t_i\bar t_j
		\]
		coincides for both systems.
	\end{enumerate}
\end{corollary}

\section{Conformally superintegrable potentials}\label{sec:non-degenerate}

Written in local coordinates, the Wilczynski 
equation~\eqref{eqn:Wilczynski.conformal} is a 
second order partial differential equation (PDE) for the potential $V$. In 
Proposition~\ref{prop:Wilczynski} we have seen that for irreducible second 
order superintegrable systems, the coefficients in this PDE are determined by 
the space of trace-free conformal Killing tensors and the metric.

\subsection{Prolongation of a superintegrable potential}
\label{sec:prolongation.potential}
The Wilczynski equation~\eqref{eqn:Wilczynski.conformal} expresses the 
trace-free part of the Hessian of the potential $V$ linearly in the 
differential $\nabla V$, the Laplacian $\Delta V$, and the potential $V$ 
itself.
The coefficients in this expression are determined by the structure 
tensors~$T$ and~$\tau$.
In the following Proposition, the Wilczynski 
equation~\eqref{eqn:Wilczynski.conformal} is extended by 
a second equation which expresses the derivatives of $\Delta V$ linearly in 
$\Delta V$, $\nabla V$ and $V$. Again, the coefficients are determined by the 
structure tensors.
The system~\eqref{eq:prolongation:V} below is an extended system of PDEs 
for~\eqref{eqn:Wilczynski.conformal}. Such an extended system is called a 
\emph{prolongation} of the initial PDE, and it allows one to make use of the 
powerful theory of parallel linear connections.
In particular, by virtue of~\eqref{eq:prolongation:V} all higher derivatives 
of $V$, $\nabla V$ and $\Delta V$ are expressed in terms of these.
\begin{proposition}\label{prop:Wilczynski.conformal.prolongation}
	Equation~\eqref{eqn:Wilczynski.conformal} has the prolongation
	\begin{subequations}\label{eq:prolongation:V}
		\begin{alignat}{9}
			\label{eq:prolongation:V:1}
			V_{,ij}
			&=T\indices{_{ij}^m}&&V_{,m}
			+\tfrac1n&g_{ij}&\Delta V
			&+\tau_{ij}&V\\
			\label{eq:prolongation:V:2}
			\tfrac{n-1}n(\Delta V)_{,k}
			&=q\indices{_k^m}&&V_{,m}
			+\tfrac1n&t_k&\Delta V
			&+\gamma_k&V,
		\end{alignat}
	\end{subequations}
	with
	\begin{subequations}
		\begin{align}
			\label{eq:q}
			q\indices{_j^m}&:=Q\indices{_{aj}^{am}}\\
			\label{eq:t}
			t_j&:=T\indices{_{aj}^a}\\
			\label{eq:gamma}
			\gamma_k&:=\Gamma\indices{_{ak}^a}
		\end{align}
	\end{subequations}
	where
	\begin{align}
		\label{eq:Q}
		Q\indices{_{ijk}^m}&:=
		T\indices{_{ij}^m_{,k}}
		+T\indices{_{ij}^l}T\indices{_{lk}^m}
		-R\indices{_{ijk}^m} + \tau_{ij}g_k^m. \\
	   \label{eq:Gamma}
	   \Gamma_{ijk}&:=
	   \tau_{ij,k}+T\indices{_{ij}^a}\tau_{ak}
	\end{align}
\end{proposition}

\begin{proof}
	Equation~\eqref{eq:prolongation:V:1} is nothing but
	the Wilczynski equation~\eqref{eqn:Wilczynski.conformal}.  Substituting 
	it into its covariant
	derivative, we obtain
	\begin{align*}
		V_{,ijk}&
		=\bigl(
			T\indices{_{ij}^m_{,k}}
			+T\indices{_{ij}^l}T\indices{_{lk}^m}
			+\tau_{ij}g_k^m
		\bigr)V_{,m} \\
		&\quad+\tfrac1n
		\bigl(
			T_{ijk}\Delta V
			+g_{ij}(\Delta V)_{,k}
		\bigr)
		+\bigl(
		    \tau_{ij,k}+T\indices{_{ij}^m}\tau_{mk}
		\bigr)\,V.
	\end{align*}
	Antisymmetrising in $j$ and $k$ and applying the Ricci identity gives
	\[
		R\indices{^m_{ijk}}V_{,m}
		=\young(j,k)
		\Bigl[
			\bigl(
				T\indices{_{ij}^m_{,k}}
				+T\indices{_{ij}^l}T\indices{_{lk}^m}
				+\tau_{ij}g_k^m
			\bigr)V_{,m}
			+\tfrac1n
			\bigl(
				T_{ijk}\Delta V
				+g_{ij}(\Delta V)_{,k}
			\bigr)
			+\Gamma_{ijk}V
		\Bigr].
	\]
	Solving for the term involving $(\Delta V)_{,k}$ on the right hand side, we get
	\[
		\frac1n\young(j,k)g_{ij}(\Delta V)_{,k}
		=-\young(j,k)
		\bigl(
			Q\indices{_{ijk}^m}V_{,m}
			+\tfrac1nT_{ijk}\Delta V
			+\Gamma_{ijk}V
		\bigr).
	\]
	The contraction of this equation in $i$ and $j$ now yields
	\eqref{eq:prolongation:V:2}, since $T_{ijk}$ and $Q\indices{_{ijk}^m}$ are
	trace-free in $i$ and $j$ by definition.
\end{proof}



\subsection{Integrability conditions for a non-degenerate potential}

From the perspective of the equations \eqref{eq:prolongation:V}, 
non-degeneracy implies that the corresponding integrability conditions be 
satisfied generically, 
independently of the potential. With this condition we finally eliminate the 
potential $V$ (and therefore all $V^{(\alpha)}$) from the system, leaving a 
system of equations depending only on the structure tensors $T_{ijk}$ and 
$\tau_{ij}$, as well as on the underlying metric $g$ and its curvature.

\begin{proposition}
	The Wilczynski equation~\eqref{eqn:Wilczynski.conformal} locally has a 
	non-degenerate solution 
	$V$ 
	if and only if the following integrability conditions hold:
	\begin{subequations}\label{eq:SIC:V}
		\begin{align}
			\label{eq:SIC:V:T}
			\young(j,k)\Bigl(T_{ijk}+\tfrac1{n-1}g_{ij}t_k\Bigr)&=0\\
			\label{eq:SIC:V:Q}
			\young(j,k)\Bigl(Q_{ijkl}+\tfrac1{n-1}g_{ij}q_{kl}\Bigr)&=0\\
			\label{eq:SIC:V:Gamma}
			\young(j,k)\Bigl(
			\Gamma_{ijk}+\tfrac{1}{n-1}\,g_{ij}\gamma_k
			\Bigr)&=0\\
			\label{eq:SIC:V:q}
			\young(k,l)
			\bigl(
				q\indices{_k^n_{,l}}
				+T\indices{_{ml}^n}q\indices{_k^m}
				+\tfrac1{n-1}t_kq\indices{_l^n}
				+g_k^n\gamma_l
			\bigr)&=0\\
			\label{eq:SIC:V:gamma}
			\young(i,j)\Bigl(
			 \gamma_{i,j}+q\indices{_i^m}\tau_{mj}+\tfrac{1}{n-1}\,t_i\gamma_j
			\Bigr)&=0.
		\end{align}
	\end{subequations}
\end{proposition}

\begin{proof}
	The system \eqref{eq:prolongation:V} allows one to write all higher
	derivatives of $V$, $\nabla V$ and $\Delta V$ as linear combinations of 
	$V$, $\nabla V$ and $\Delta V$.
	Necessary and sufficient integrability conditions are
	then given by the Ricci	identities
	\begin{align*}
		\young(j,k)V_{,ijk}&=R\indices{^m_{ijk}}V_{,m}&
		\young(k,l)(\Delta V)_{,kl}&=0\,,
	\end{align*}
	where we substitute~\eqref{eqn:Wilczynski.conformal} in the left hand 
	side of the equations.
	We obtain
	\begin{align*}
		\young(j,k)
		\bigl(
			Q\indices{_{ijk}^m}
			+\tfrac1{n-1}g_{ij}q\indices{_k^m}
		\bigr)V_{,m}
		+\frac1n\young(j,k)
		\bigl(
			T_{ijk}
			+\tfrac1{n-1}g_{ij}t_k
		\bigr)\Delta V \\
		+\young(j,k)\Bigl(
		\Gamma_{ijk}+\tfrac{1}{n+1}\,g_{ij}\gamma_k
		\Bigr)V
		&=0
	\end{align*}
	and, respectively,
	\begin{align*}
		\young(k,l)
		\bigl(
			q\indices{_k^n_{,l}}
			+T\indices{_{ml}^n}q\indices{_k^m}
			+\tfrac1{n-1}t_kq\indices{_l^n}
		\bigr)V_{,n}
		+\frac1n\young(k,l)
		\bigl(
			t_{k,l}
			+q_{kl}
		\bigr)\Delta V \\
		+n\,\young(k,l)\Bigl(
		\gamma_{k,l}+q\indices{_k^m}\tau_{ml}+\tfrac{1}{n-1}\,t_k\gamma_l
		\Bigr)V
		&=0.
	\end{align*}
	For a non-degenerate superintegrable potential the coefficients of $\Delta V$, $\nabla V$ and $V$ must vanish separately.
	In addition to the stated integrability conditions, this yields the 
	condition
	\begin{equation}
		\label{eq:t+q}
		\young(k,l)
		\bigl(
			t_{k,l}
			+q_{kl}
		\bigr)=0.
	\end{equation}
	The latter is redundant, since it can be obtained from \eqref{eq:SIC:V:Q} by
	a contraction over $i$ and $l$.
\end{proof}

Note that the equations~\eqref{eq:SIC:V} are not invariant under conformal 
trnasformations, as they emerge from coefficients of $V$, $\nabla V$ and 
$\Delta V$, respectively.
Still, after a conformal transformation as in 
Proposition~\ref{prop:trafo.conformal.Killing.and.integrals}, the form 
of~\eqref{eq:SIC:V} is the same, but the metric, the curvature and the 
structure tensors are replaced.
We can solve Equation~\eqref{eq:SIC:V:T} right away, because it is linear and
does not involve derivatives.

\begin{proposition}\label{prop:solve.SIC:V:1}
	The first integrability condition for a superintegrable potential,
	Equation~\eqref{eq:SIC:V:T}, has the solution
	\begin{equation}
		\label{eq:T(S,t)}
		T_{ijk}
		= S_{ijk}+\young(ij)\left( \bar t_ig_{jk}-\frac 1ng_{ij}\bar t_k 
		\right),
	\end{equation}
	where $S$ is an arbitrary totally symmetric and trace-free tensor. The 
	1-form $\bar t_i$ is given by
	\begin{equation}\label{eqn:t.bar}
	  \bar t_i
	  = \frac{n}{(n-1)(n+2)}\,t_i
	  = \frac{n}{(n-1)(n+2)}\,T\indices{_{ia}^a}\,.
	\end{equation}
	Note that $S_{ijk}$ and $\bar t_i$ are uniquely determined by $T$.
\end{proposition}

\begin{proof}
  Let us decompose $T_{ijk}$, which by definition is trace-free and symmetric 
  in $(i,j)$,
  \[
  	\Yboxdim{7pt}
  	\yng(2)\otimes\yng(1)
  	\simeq\yng(3)\oplus\yng(2,1)
  	\simeq \underbrace{ {\yng(3)}_\circ }_{=S_{ijk}}
  		\oplus\yng(1)
  		\oplus {\yng(2,1)}_\circ
  		\oplus\yng(1)\,.
  \]
  Due to~\eqref{eq:SIC:V:T}, the penultimate component of this decomposition 
  vanishes, and therefore we obtain
  \[
   T_{ijk} = S_{ijk} + \frac16\,\young(ijk)\,g_{ij}s_k
	     + \frac14\,\young(ij)\,\young(j,k)\,g_{ij}\xi_k\,,
  \]
  where $s_k$ and $\xi_k$ are components of some 1-forms.
  Substituting~\eqref{eq:SIC:V:T}, and taking the trace in $(j,k)$,
  \[
   t_i = T\indices{_{ia}^a} = \frac{n+2}{3}\,s_i-\frac{n-1}{4}\,\xi_i\,.
  \]
  Taking the trace in $(i,j)$ instead,
  \[
   0 = T\indices{^a_{ai}} = \frac{n+2}{3}\,s_i+\frac{n-1}{2}\,\xi_i\,.
  \]
  Solving for $s_i$ and $\xi_i$, we find
  \[
  	s_i = \frac{2t_i}{n+2}\,,\quad
  	\xi_i = -\frac{4t_i}{3(n-1)}\,.
  \]
  Resubstituting into the initial formula for $T_{ijk}$, we arrive 
  at~\eqref{eq:T(S,t)}.
\end{proof}

\begin{corollary}~
	\label{cor:symmetry}
	\begin{enumerate}
		\item The tensor $q_{ij}$ is symmetric,
			i.e. $q_{ji}=q_{ij}$.
		\item The 1-form $\bar t_i$ is the derivative of a function $\bar t$,
			i.e. $\bar t_i=\bar t_{,i}$.
	\end{enumerate}
\end{corollary}
\noindent Note that consequently also $t_i=t_{,i}$. Without loss of
generality we impose $\bar t=\tfrac{n}{(n-1)(n+2)}\,t$ on the arbitrary 
integration constant.
\begin{proof}
	The first statement follows from substituting \eqref{eq:T(S,t)} into the
	definition \eqref{eq:q} of $q_{ij}$.  The second then follows from \eqref{eq:t+q}.
\end{proof}

\begin{proposition}
	The third integrability condition for a superintegrable potential,
	Equation~\eqref{eq:SIC:V:Gamma}, has the solution
	\begin{equation}
		\label{eq:Gamma.decomposed}
		\Gamma_{ijk}
		=\Sigma_{ijk}+\young(ij)(\bar \gamma_ig_{jk}-\frac 1ng_{ij}\bar \gamma_k),
	\end{equation}
	where $\Sigma$ is an arbitrary totally symmetric and trace free tensor and
	\[
		\bar \gamma_i=\frac n{(n+2)(n-1)}\,\gamma_{i}.
	\]
	Note that $\Sigma$ and $\gamma$ are uniquely determined by $\Gamma$.
\end{proposition}

\noindent The proof is the same as that of Proposition~\ref{prop:solve.SIC:V:1}.

\subsection{The conformal scale function}
As we have just seen, the trace $t_i$ of the primary structure tensor is the 
differential of a function $t$.
We thus obtain the following transformation rule under conformal changes of 
the Hamiltonian, which is an immediate consequence of~\eqref{eqn:trafo.t.i}.

\begin{lemma}\label{la:trafo.t}
 Under a conformal change of the superintegrable system, say 
 $H\mapsto\Omega^{-2}H$, $\Omega>0$, the function $t$ transforms as
 \begin{equation}\label{eqn:trafo.t}
  \bar t \mapsto \bar t -3\Upsilon\quad\text{up to an irrelevant constant}\,,
 \end{equation}
 where $\Upsilon=\ln\Omega$.
\end{lemma}
Note that the function $\bar t$ is determined by the structure tensor $T$ 
only up to an irrelevant constant.
The symmetry group of c-superintegrable systems is 
$\mathfrak{S}=\mathrm{Conf}(M)\rtimes\R^*$, c.f.\ 
Remark~\ref{rmk:symmetry.group.cSIS}. The second factor of $\mathfrak S$ does 
not affect $\bar t$. Indeed, we see that if 
$\bar t^\text{new}-\bar t^\text{old}=c\in\R$, then 
$\Omega=e^{-\frac{c}{3}}$ and thus
\[
	H^\text{new}=e^{\tfrac{2c}{3}}\,H^\text{old}\,.
\]
Lemma~\ref{la:trafo.t} above therefore confirms that $\bar t$ behaves like a 
scale function.
\begin{definition}\label{def:scaling}
The \emph{conformal scale function} is the density of weight~1 defined by
\begin{equation*}
  \cs = e^{-\frac13\,\bar{t}}\,.\qedhere
\end{equation*}
\end{definition}

Lemma~\ref{la:trafo.t} allows us to change $\bar t$ arbitrarily, resulting in 
a natural gauge freedom of a conformally superintegrable system.
There are three natural scale choices, in particular, that are relevant in 
this paper, each of which has 
specific features that we can exploit to gain information or to simplify 
certain computations. Table~\ref{tab:scales} summarises some properties of 
the scale choices and the notation we use.
\medskip

\subsubsection{Standard scale}
This scale choice realises $\bar t_{,i}=0$.
\begin{definition}
	A conformally superintegrable system with $\bar t_i = 0$ is said to be in 
	\emph{standard scale}.
\end{definition}

\noindent We shall use a specific notation for the metric and the secondary 
structure tensor when we work in standard scale.
Given an arbitrary scale choice, we can apply a conformal transformation with 
conformal rescaling $\Omega=e^{\frac13\bar t}$. Let $\Upsilon=\ln\Omega$. The 
transformed metric of the system in standard scale, then is
\begin{subequations}
	\begin{align}
		\label{eq:standard.scale.metric}
		\tilde g_{ij} &= e^{\frac23\,\bar t}\,g_{ij} =: \mathfrak g_{ij}\,,
		\\
		\intertext{and the new structure tensors become
			$\tilde T\indices{_{ij}^k} = S\indices{_{ij}^k}$
			and}
		\label{eqn:aleph}
		\tilde\tau_{ij} &= 
		\tau_{ij}
		+ \frac23\,S\indices{_{ij}^k}\bar t_{,k}
		- {\young(ij)}_\circ
		\left(
		\frac13\bar t_{,ij} -\frac49\,\bar t_{,i}\bar t_{,j}
		\right)
		=: \aleph_{ij}\,.
		\\
		\intertext{For later reference, we mention the Schouten curvature 
		$\mathfrak{P}_{ij}$ of $\mathfrak g_{ij}$, which is given by}
		\label{eqn:frak.schouten}
		\mathfrak{P}_{ij} &=
		\schouten_{ij}
		- \Upsilon_{,ij} + \Upsilon_{,i}\Upsilon_{,j}
		- \frac12\,g_{ij}\,\Upsilon_{,c}\Upsilon^{,c}\,,
	\end{align}
	while the Weyl curvature remains unchanged under conformal 
	transformations. Equations \eqref{eqn:aleph} and 
	\eqref{eqn:frak.schouten} are obtained, respectively, from 
	\eqref{eqn:transformation.tau} and \eqref{eqn:transformation.schouten}.
\end{subequations}

\noindent Note that the standard scale is not unique, as we may add any 
constant to $\bar t$. For simplicity we usually choose $\bar t=0$.
As discussed after Lemma~\ref{la:trafo.t}, however, this only means that the 
Hamiltonian is multiplied by a constant, and moreover the structure tensor is 
not changed in the process.
From the viewpoint of conformally superintegrable systems however, if we 
multiply the metric by a constant, this typically changes the underlying 
metric (unless the transformation is already in $\mathrm{Diff}(M)$).
Yet, the space $\mathring{\mathcal C}$ of conformal Killing tensors of a 
conformally superintegrable system remains unaffected by such a change.

The standard scale has two major advantages: On the one hand, it yields very 
compact equations, facilitating some otherwise tedious computations. On the 
other hand, the standard scale exposes the invariant data of the problem, 
which is going to be particularly helpful when we discuss conformal 
equivalence classes.

\begin{example}
	The systems VII [5], O and A in Table~\ref{tab:3D} are in standard scale.
\end{example}
\medskip

\subsubsection{Flat scale}
This scale choice only exists for conformally flat metrics.
\begin{definition}
	A conformally superintegrable system with flat curvature is said to be in 
	\emph{flat scale}.
\end{definition}

\noindent We find, using \eqref{eq:standard.scale.metric}, that there is a 
function $\rho:\mathcal{C}^\infty(M)\to\R$ such that
\begin{equation}\label{eq:flat.scale.metric}
	\mathfrak g_{ij} = e^{\frac23\rho}\,h_{ij}.
\end{equation}
where $h$ has vanishing curvature.
A major advantage of flat scales is that covariant derivatives coincide with 
partial ones, facilitating concrete computations in local coordinates.
Moreover, the existence of a flat scale permits us to express the Ricci 
curvature in terms of the scalar function $\rho$ 
using~\eqref{eqn:trafo.Ricci.curvature}.

Flat scales are not unique. For example, we can add any constant to $\rho$.
According to \cite{Kulkarni70}, any conformal change transforming a flat 
metric into a flat metric is given via $\rho\to\rho-\eta$ where $\eta$ is a 
function satisfying
\[
	\lsb Q(Y,Z)+g(Y,Z)r \rsb\,X
	- \lsb Q(X,Z)+g(X,Z)r \rsb\,Y
	+g(Y,Z)Q(X) -g(X,Z)Q(Y) = 0
\]
where
\[
	Q(X,Y)=(\nabla^2\eta)(X,Y)-X(\eta)Y(\eta)
\]
with $g(Q(X),Y)=Q(X,Y)$ and $r=g(d\eta,d\eta)$.

\begin{example}
	All systems in Table~\ref{tab:3D} are in flat scale.
	In particular, note that the systems III [23] and V [32] are conformally 
	equivalent.
\end{example}
\medskip

\subsubsection{Proper scale}
A third natural choice is the proper scale, in which the system is 
properly superintegrable (up to a trace correction of the trace-free 
conformal 
Killing tensors).
As mentioned earlier, any conformally superintegrable system is conformally 
equivalent to a properly superintegrable system \cite[Theorem 
4.1.8]{Capel_phdthesis}. According to Lemma~\ref{la:proper.then.tau0} it 
satisfies $\tau_{ij}=0$.
\begin{definition}
	A conformally superintegrable system with $\tau_{ij} = 0$ is said to be 
	in \emph{proper scale}.
\end{definition}

\noindent Again, the proper scale choice is not unique.
Its advantage is that all the known results about properly superintegrable 
systems can be invoked.
Yet it is less useful for gaining insight into the underlying conformal 
geometry.
Nevertheless, from the viewpoint of conformal geometry, proper scale choices 
have some interesting properties which we explore in 
Section~\ref{sec:proper.constant.curvature} for constant curvature spaces.

\begin{example}
	All systems in Table~\ref{tab:3D} are in proper scale.
	For an example that is in proper scale, but neither in flat nor standard 
	scale, consider the metric $g=(z\bar z+4)^{-1}dzd\bar z$ on the 
	2-sphere. It admits the superintegrable potential
	\[
		V =
		a_0\,\left(\frac{z\bar z+4}{z+\bar z}\right)^2
		+ a_1\,\left(\frac{z\bar z+4}{z-\bar z}\right)^2
		+ a_2\,\lb\frac{z\bar z+4}{z\bar z-4}\rb^2
		+ a_3
	\]
	with parameters $a_i\in\R$, and satisfies
	\[
		\bar t = \frac32\,\ln\frac{i\,(z\bar z+4)^3}
							{(z\bar z-4)(z+\bar z)(z-\bar z)}
	\]
\end{example}

\begin{example}[Generic system on the 3-sphere]
\label{ex:generic.3.sphere}
	Consider the 3-sphere with metric
	\[
		g = d\phi^2 + \sin^2(\phi)\,\left(
				d\theta^2 +\sin^2(\theta)\,d\psi^2
			\right)\,.
	\]
	The potential
	\[
		V = \frac{a_1}{\cos^2(\phi)}
		+ \frac{a_2}{\sin^2(\phi)\cos^2(\theta)}
		+ \frac{a_3}{\sin^2(\phi)\sin^2(\theta)\cos^2(\psi)}
		+ \frac{a_4}{\sin^2(\phi)\sin^2(\theta)\sin^2(\psi)}
		+ a_0
	\]
	is non-degenerate and defines the so-called generic system on the 
	3-sphere; in \cite{Kalnins&Kress&Miller-IV} it is labelled VIII.
	It is in proper scale, but neither in flat nor in standard scale.
\end{example}

Note that the Harmonic Oscillator, see Example~\eqref{ex:IHO}, is 
simultaneously in standard, flat and proper scale. In Section~\ref{sec:ccsis} 
we find that this is an immediate consequence of $T_{ijk}=0$.

\begin{table}[th]
 \begin{tabular}{l|ccc}
    \toprule
    Objects & Standard Scale & Flat scale & Proper scale \\
    \midrule
    Function $\bar t$
    & 0 & $\rho$ & $\bar t$ \\
    Schouten tensor $\schouten_{ij}$
    & $\mathfrak{P}_{ij}$ & 0 & $\schouten_{ij}$ \\
    Secondary structure tensor $\tau_{ij}$
    & $\aleph_{ij}$ & $\tau_{ij}$ & 0 \\
    Metric $g_{ij}$
    & $\mathfrak g_{ij}$ & $h_{ij}$ & $g_{ij}$ \\
    \bottomrule
 \end{tabular}
 \medskip
 
 \caption{Notation and conventions for the three natural scale choices.}\label{tab:scales}
\end{table}

\section{Conformal Killing tensors in conformally superintegrable systems}\label{sec:abundant}

\subsection{Prolongation equations for trace-free conformal Killing 
tensors}

In Section~\ref{sec:prolongation.potential} we have discussed a prolongation 
system for the potential $V$.
Similarly, we can write down a prolongation for an arbitrary trace-free 
conformal Killing tensor $C_{ij}$. In general this system can be rather 
complicated \cite{Weir1977}, given also the explicit but complicated 
expressions well known for proper second order Killing tensors 
in~\cite{Wolf98,Gover&Leistner}.
However, as is shown in~\cite{Kress&Schoebel&Vollmer}, the prolongation 
system for proper second order Killing tensors in non-degenerate 
superintegrable systems simplifies considerably. In fact, the prolongation 
system in this case closes after the first covariant derivative. We observe 
the same phenomenon with \emph{trace-free} conformal Killing tensors.
We remark that trace-freeness is paramount. Indeed, for conformal Killing 
tensors with non-vanishing trace, the prolongation system would not be finite.

\begin{theorem}\label{thm:prolongation.K}
 A trace-free conformal Killing tensor $C_{ij}$ in a non-degenerate 
 conformally superintegrable system satisfies
 \begin{equation}\label{eq:prolongation:K}
	 C_{ij,k}
	 =\frac13\,\left(
	    \young(ji,k)T\indices{^m_{ji}}g\indices{_k^n}
	    -\frac{2}{n}\,g_{ij}(t^mg\indices{^n_k}-T\indices{_k^{m n}})
      \right)\,C_{mn}\,,
 \end{equation}
 with the primary structure tensor $T\indices{_{ij}^k}$ given by the 
 Wilczynski equation~\eqref{eqn:Wilczynski.conformal}.
 The Bertrand-Darboux condition~\eqref{eqn:conformal.dKdV} in this situation 
 is equivalent to \eqref{eq:prolongation:K} and
 \begin{equation}\label{eq:prolongation:omega}
 	\young(j,k)\,\left( \omega_{j,k} + C\indices{^m_j}\tau_{km} \right) = 0\,.
 \end{equation}
\end{theorem}
\noindent Note that from~\eqref{eq:prolongation:K} we obtain $\omega_i$ using 
Formula~\eqref{eqn:omega.from.divC}. We also remark 
that~\eqref{eq:prolongation:K} does not contain the secondary structure 
tensor $\tau_{ij}$. Indeed, we shall see that, under the hypothesis of the 
theorem, the tensor~$\tau_{ij}$ is obtained from $T_{ijk}$ and the Ricci 
curvature.

\begin{remark}
	Equation~\eqref{eq:prolongation:K} should be compared to the prolongation 
	equation~\eqref{eqn:shortcut.eqn} for a Killing tensor in a properly 
	superintegrable system.
	However, here $K_{ij,k}$ is not trace-free. Therefore, we need to 
	subtract the trace, obtaining
	\begin{equation}\label{eqn:compare.properly.supint.K.trace}
		K\indices{^a_{a,k}}
		=\frac23\,\left(t^m g\indices{_k^n} - 
		T\indices{_k^{mn}}\right)\,K_{mn}\,.
	\end{equation}
	Next, verify that
	\[
	\left(
	\young(ji,k)T\indices{^m_{ji}}g\indices{_k^n}
	-\frac{2}{n}\,g_{ij}(t^mg\indices{^n_k}-T\indices{_k^{m n}})
	\right)\,g_{mn}
	= 0\,,
	\]
	which, combined with~\eqref{eqn:shortcut.eqn} 
	and~\eqref{eqn:compare.properly.supint.K.trace}, yields
	\[
	C_{ij,k}
	= \frac13\,\left(
	\young(ji,k)T\indices{^m_{ji}}g\indices{_k^n}
	-\frac{2}{n}\,g_{ij}(t^mg\indices{^n_k}-T\indices{_k^{m n}})
	\right)\,C_{mn}
	\]
	where $C_{ij}=K_{ij}-\frac{1}{n}\,g_{ij}\,K\indices{^a_a}$.
	Summarizing, we have thus confirmed that the trace-free part $C_{ij}$ of 
	a properly superintegrable Killing tensor $K_{ij}$ 	
	satisfies~\eqref{eq:prolongation:K}.
\end{remark}

\begin{proof}[Proof of Theorem~\ref{thm:prolongation.K}]
 We decompose $C_{ij,k}$ as
 \begin{equation}\label{eqn:Kijk.decomposition}
 	C_{ij,k}
 	= \frac13\,\young(ji,k) C_{ij,k}
 	+ \frac16\,\young(ijk) C_{ij,k}
 \end{equation}
 The totally symmetric component is given by the conformal Killing equation,
 \[
  \young(ijk) C_{ij,k} = \young(ijk) \omega_kg_{ij}
 \]
 The hook symmetric component is obtained as follows:
 Substituting the Wilczynski equation~\eqref{eqn:Wilczynski.conformal} into 
 the Bertrand-Darboux equation~\eqref{eqn:conformal.dKdV} gives
 \begin{equation}\label{eqn:dKdV.Kij,k}
  \young(j,k)
  \Bigl[
  \bigl(
	   C\indices{^m_{j,k}}
	   - T\indices{_{jl}^m}C\indices{^l_k}
	   +\omega_jg_k^m
  \bigr)V_{,m}
  +\bigl(
	    C\indices{^m_j}\tau_{km}+\omega_{j,k}
   \bigr)V
  \Bigr] = 0.
 \end{equation}
 From non-degeneracy it follows that the coefficients of $V_{,m}$ and $V$ 
 vanish independently. The coefficient of $V$ 
 yields~\eqref{eq:prolongation:omega}. From the coefficients of~$V_{,m}$ we 
 obtain
 \[
	 \young(j,k)C_{ij,k}=
	 \young(j,k)\bigl(
		 T\indices{^l_{ji}}C_{lk}
		 +g_{ij}\omega_k
	 \bigr).
 \]
 Altogether, using~\eqref{eqn:Kijk.decomposition},
 \begin{equation*} 
	C_{ij,k}
	= \frac13\young(ji,k)\bigl(
		 T\indices{^l_{ji}}C_{lk}
		 +\omega_kg_{ij}
	 \bigr)+\frac16\young(ijk)\omega_kg_{ij}\,
 \end{equation*}
 The trace-freeness of $C_{ij}$ now implies
 \begin{equation}\label{eqn:omega}
  \omega_{k}
  = -\frac{1}{3n}\,g^{ij} \young(ji,k)\,T\indices{^m_{ji}}C_{mk}
  = \frac{2}{3n}\,(T\indices{_k^{ab}}-t^ag\indices{^b_k})\,C_{ab}\,,
 \end{equation}
 which completes the proof.
\end{proof}

\noindent Equation~\eqref{eq:prolongation:omega} allows us to prove the 
converse of Lemma~\ref{la:proper.then.tau0} for non-degenerate systems.
Note that there is a natural mapping from the space $\mathcal K$ of Killing 
tensors into the space~$\mathcal C^\circ$ of trace-free conformal Killing 
tensors,
\[
	\mathcal K\to\mathcal C^\circ\,,\qquad
	K_{ij} \mapsto C_{ij}=K_{ij}-\frac1n\,K\indices{^a_a}g_{ij}\,.
\]
This map is not surjective as not every conformal Killing tensor arises from 
a proper Killing tensor. Its range consists of trace-free conformal Killing 
tensors whose $\omega$ from~\eqref{eqn:omega.from.divC} is exact, 
$\omega=d\lambda$, and thus
\[
	\lcb
		C\in\mathcal C^\circ : 
		2C\indices{^{a}_{k,a}}+C\indices{^a_{a,k}}=\lambda_{,k}
								\text{ for some scalar $\lambda$}
	\rcb \to \mathcal K\,,\quad
	C\mapsto C-\frac1n\,\lambda g\,,
\]
is surjective. It is not injective as we may always add a constant multiple 
of the metric to a Killing tensor.
From~\eqref{eq:prolongation:omega} we infer that $\omega$ is exact for 
trace-free conformal Killing tensors that commute with $\tau$.

\begin{corollary}\label{cor:tau0.then.proper}
 If $\tau_{ij}=0$ for a non-degnerate second order conformally 
 superintegrable system, then the system is properly superintegrable.
\end{corollary}
\noindent Note that in the proof we do not take functional independence into 
account yet, but we will account for it in 
Lemma~\ref{la:functional.dependence}. This lemma ensures the existence of 
sufficiently many functionally independent conformal integals for almost any 
potential of a non-degenerate system, which suffices here as we consider the 
space $\mathcal V^\text{max}$.
\begin{proof}
 Let $C_{ij}$ be a trace-free conformal Killing tensor of the conformally 
 superintegrable system. We need to find a function $\lambda$ such that 
 $K_{ij}=C_{ij}+\frac1n\,g_{ij}\lambda$ is a proper Killing tensor, i.e.\ it 
 satisfies the Bertrand-Darboux condition~\eqref{eqn:Bertrand.Darboux}.
 We proceed in two steps. First we show that $d\omega=0$. Then we prove that this leads to a properly superintegrable system.
 
 For the first step, take the coefficient of $V$ in~\eqref{eqn:dKdV.Kij,k}. 
 For a non-degenerate system~\eqref{eq:prolongation:omega} yields
 \[
 	2\,d\omega_{ij} = \young(i,j)\,\omega_{i,j}
 	= \young(i,j)\,C\indices{^a_j}\tau_{ia}
 	= \young(i,j)\,K\indices{^a_j}\tau_{ia}
 \]
 and therefore $\omega$ is exact if $\tau_{ij}=0$, i.e.~$\omega=d\lambda$.
 Let $K_{ij}=C_{ij}+\frac1n\,g_{ij}\lambda$ with this specific 
 function~$\lambda$. We conclude, using the trace-freeness of $C_{ij}$,
 \begin{equation}\label{eq:BD.proper}
  \left( d(KdV) \right)_{ij}
  = \frac12\,\young(i,j)\,\left(
  			K_{ia,j}V^{,a}
  			+ K\indices{_i^a} V_{,aj}
  		\right)
  = 0\,,
 \end{equation}
 due to the conformal Bertrand-Darboux condition~\eqref{eqn:conformal.dKdV}.
 So $K_{ij}$ satisfies the proper Bertrand-Darboux 
 equation~\eqref{eqn:Bertrand.Darboux}.
 This proves the claim.
\end{proof}

\noindent To allow for a concise notation, we introduce the shorthand
\begin{equation}\label{eqn:P}
	P\indices{_{ijk}^{mn}}
	:=\frac16\young(mn)\left(
	\young(ji,k)T\indices{^m_{ji}}g\indices{_k^n}
	-\frac{2}{n}\,g_{ij}(t^mg\indices{^n_k}-T\indices{^m_k^n})
	\right)\,.
\end{equation}
Consequently, we have~\eqref{eq:prolongation:K} in the form
\[
	K_{ij,k} = P\indices{_{ijk}^{ab}}K_{ab}\,.
\]
Given $T_{ijk}$, we can compute $P\indices{_{ijk}^{mn}}$.
The following lemma shows that $P\indices{_{ijk}^{mn}}$ contains all the 
information about $\tau_{ij}$, i.e.\ for abundant systems the secondary 
structure tensor is redundant.

\begin{lemma}\label{la:tau}
In an abundant conformally superintegrable system, the tensor $\tau_{ij}$ is 
given by
\begin{equation}\label{eqn:tau}
 \tau_{ij} = \frac{2}{n}\,
 	\bigg(
		\Lambda\indices{^{a}_{ja,i}}
		-\Lambda\indices{_{ij}^a_{,a}}
		-\Lambda\indices{_i^{ab}}\,P\indices{_{ab}^c_{cj}}
		+\Lambda^{cab}\,P_{abijc}
    \bigg)
\end{equation}
where
\[
 \Lambda\indices{_k^{ab}} = 
 \frac{1}{3n}\,\young(ab)\,(T\indices{_k^{ab}}-t^ag\indices{^b_k})\,.
\]
Moreover,
\begin{equation}\label{eqn:strange.equation}
  \young(i,j)\young(mn)\,\left(
       \Lambda\indices{_{imn}_{,j}} + \Lambda\indices{_i^{ab}} 
       P\indices{_{abjmn}}
      \right)_\circ = 0\,.
\end{equation}
\end{lemma}

\noindent Note that $\tau_{ij}$ in~\eqref{eqn:tau} is symmetric and 
trace-free due to~\eqref{eqn:strange.equation}.
Because of Equation~\eqref{eqn:tau} the superintegrable 
potential is completely determined by the primary structure tensor 
$T\indices{_{ij}^k}$, and this observation can be interpreted as follows: Any 
conformally 
superintegrable system corresponds to a properly superintegrable system 
\cite{Capel_phdthesis}, for which the Wilczynski 
equation~\eqref{eqn:Wilczynski.conformal} holds with 
$\tau_{ij}=0$. Applying a conformal transformation, due to 
\eqref{eqn:transformation.tau} the tensor~$\tau_{ij}$ can only contain 
information from the properly superintegrable system (and the conformal 
factor).
Indeed, this is the information appearing on the right hand side of 
Equation~\eqref{eqn:tau}.

\begin{proof}[Proof of Lemma~\ref{la:tau}]
 Equation~\eqref{eq:prolongation:omega} yields the antisymmetric part of 
 $\omega_{i,j}$,
 \begin{equation}\label{eqn:d.omega}
 \young(i,j)\omega_{i,j}
 = \young(i,j)C\indices{^m_j}\tau_{im}\,.
 \end{equation}
 On the other hand we obtain from~\eqref{eqn:omega}, after one differentiation,
 \[
  \omega_{i,j} = \left(
       \Lambda\indices{_{i}^{mn}_{,j}} + \Lambda\indices{_i^{ab}} 
       P\indices{_{abj}^{mn}}
      \right)\,C_{mn}
 \]
 Resubstituting into~\eqref{eqn:d.omega},
 \[
  \young(i,j)\,\left(
       \Lambda\indices{_{i}^{mn}_{,j}}
       + \Lambda\indices{_i^{ab}} P\indices{_{abj}^{mn}}
       + \young(i,j)\tau\indices{_i^m}g\indices{^n_j} \right) C_{mn}
       = 0\,.
 \]
 Next, using the fact that there are $\frac{n(n+1)}{2}-1=r_\text{max}$ 
 linearly independent, trace-free and symmetric $C_{ij}$, we conclude that 
 the symmetrisation of the coefficients of $C_{mn}$ must vanish 
 independently,
 \begin{equation}\label{eqn:pre.tau.eqn}
  \young(i,j)\young(mn)\,\left(
       \Lambda\indices{_{imn}_{,j}} + \Lambda\indices{_i^{ab}} 
       P\indices{_{abjmn}}
       +\tau_{im}g_{nj}
      \right)=0\,.
 \end{equation}
 Contracting in $(n,j)$ yields~\eqref{eqn:tau}.
 Contracting~\eqref{eqn:pre.tau.eqn} in $(n,m)$ shows that~\eqref{eqn:tau} 
 is the only independent trace of~\eqref{eqn:pre.tau.eqn}.
 The trace-free part of~\eqref{eqn:pre.tau.eqn} 
 is~\eqref{eqn:strange.equation}, and this completes the proof.
\end{proof}

Lemma~\ref{la:tau} ensures that $P\indices{_{ijk}^{mn}}$ contains enough
information to recompute the structure tensors $T_{ijk}$ and $\tau_{ij}$.
\begin{corollary}
	In an abundant system, the structure tensors can be recomputed from 
	$P\indices{_{ijk}^{mn}}$ defined in~\eqref{eqn:P}.
	We have
	\begin{align*}
		t_{k} &= -\frac{3}{n}\,P\indices{_{abi}^{ab}}
		\\
		S_{ijk} &= \frac{3}{n}\,
		\left(
		P\indices{_{ijak}^{a}}
		+\frac{n-1}{n}\,g_{ij} P\indices{_{abk}^{ab}}
		+\young(ij)\frac{n-2}{n}g_{ik} P\indices{_{abj}^{ab}}
		\right)\,,
	\end{align*}
	which yield $T_{ijk}$, and~\eqref{eqn:tau}, which yields $\tau_{ij}$.
\end{corollary}
\begin{proof}
	This follows from
	$P\indices{_{abi}^{ab}} = -\frac{n}{3}\,t_i$,
	and
	\[
	P\indices{_{ijak}^{a}}
	= \frac{n}{3} S_{ijk} +\frac{n-1}{3}\,g_{ij}t_k 
	+\young(ij)\frac{n-2}{3}g_{ik}t_j\,.
	\]
	Together with Lemma~\ref{la:tau} the claim follows.
\end{proof}

\subsection{Integrability conditions in an abundant system}

A trace-free conformal Killing tensor in an abundant conformally 
superintegrable system satisfies the prolongation 
system~\eqref{eq:prolongation:K}.
Due to the condition of abundantness, its integrability condition will only 
depend on $g$, $T$ and $\nabla T$.
We have already seen that non-degeneracy is the condition for the generic 
integrability of $V$.
Along a similar line, abundantness is then the condition for the generic 
integrability of $K_{ij}$.

\begin{proposition}\label{prop:SIC:K}
	For the trace-free conformal Killing tensor fields in an abundant 
	(conformally) superintegrable system, the integrability condition 
	of~\eqref{eq:prolongation:K} reads
	\begin{equation}\label{eq:SIC:K}
		\young(k,l)
		\Bigl(
		P\indices{_{ijk}^{mn}_{,l}}
		+P\indices{_{ijk}^{pq}}P\indices{_{pql}^{mn}}
		\Bigr)
		=
		\frac12\young(ij)\young(mn)
		R\indices{^m_{ikl}}g^n_j\,.
	\end{equation}
\end{proposition}

\noindent Note that the integrability conditions~\eqref{eq:SIC:K} are not 
conformally invariant. This is entirely analogous to~\eqref{eq:SIC:V}, which 
are not invariant either. However, our further analysis is going to show that 
we can distill invariant conditions out of~\eqref{eq:SIC:K} and, as we shall 
see, these already imply~\eqref{eq:SIC:V}.

\begin{proof}
	Writing
	\begin{equation}
		\label{eq:prolongation:K:bis}
		C_{ij,k} = P\indices{_{ijk}^{mn}}C_{mn}
	\end{equation}
	and taking the covariant derivative yields
	\[
		C_{ij,kl}=
		P\indices{_{ijk}^{mn}_{,l}}C_{mn}
		+P\indices{_{ijk}^{mn}}C_{mn,l}\,.
	\]
	After antisymmetrisation over~$(k,l)$ we can eliminate all derivatives
	of~$C$ by using the Ricci identity
	\[
		\young(k,l)C_{ij,kl}=\young(ij)R\indices{^m_{ikl}}C_{mj}
	\]
	on the left hand side, and substituting~\eqref{eq:prolongation:K:bis} for
	$C_{mn,l}$ on the right hand side.
	We obtain
	\begin{equation}\label{eqn:K.integrability}
		\young(ij)
			R\indices{^m_{ikl}}g\indices{^n_j}C_{mn}
		=
		\young(k,l)
			\Bigl(
				P\indices{_{ijk}^{mn}_{,l}}
				+
				P\indices{_{ijk}^{pq}}
				P\indices{_{pql}^{mn}}
			\Bigr)
			C_{mn}\,.
	\end{equation}
	An abundant conformally superintegrable system has 
	$\frac{n(n+1)}{2}-1=r_\text{max}$ linearly 
	independent trace-free conformal Killing tensors $C$. Since this is 
	exactly the number of independent components of the trace-free symmetric 
	tensor $C_{mn}$, we can replace $C_{mn}$ by a symmetrisation in $m$ and 
	$n$	as the expression in parentheses in~\eqref{eqn:K.integrability} is
	already trace-free in $m$ and $n$. We have thus obtained the 
	claim.
\end{proof}

\begin{lemma}\label{la:curvature}
  For an abundant system, the curvature tensor $R\indices{^l_{ijk}}$ 
  satisfies
  \begin{equation}\label{eqn:curvature}
       R\indices{^l_{ijk}}
       = \frac{2}{n+2}\,\young(j,k)\,\left(
           P\indices{_{iaj}^{la}_{,k}} + P\indices{_{iaj}^{pq}} 
           P\indices{_{pqk}^{la}}
         \right)\,.
  \end{equation}
\end{lemma}
\begin{proof}
  Contracting~\eqref{eq:SIC:K} in $n$ and $j$ immediately yields the result.
\end{proof}

\noindent Lemma~\ref{la:curvature} allows us to express the curvature in 
terms of the superintegrable structure tensor. Alternatively we can also view 
it as a curvature obstruction to the structure tensor. In any case, it 
enables us to (almost) eliminate the curvature from the integrability 
conditions.
\begin{lemma}\label{la:structural.equation}
An abundant conformally superintegrable system satisfies the curvature 
independent equation
 \begin{multline}\label{eqn:structural.equation}
 	\young(k,l)
	 	\Bigl(
		 	P\indices{_{ijk}^{mn}_{,l}}
		 	+P\indices{_{ijk}^{pq}}P\indices{_{pql}^{mn}}
	 	\Bigr)
 	=
	\young(ij)\young(mn)\young(k,l)
		\left(
           P\indices{_{iak}^{ma}_{,l}} + P\indices{_{iak}^{pq}} 
           P\indices{_{pql}^{ma}}
        \right)\,\frac{g^n_j}{n+1}\,.
 \end{multline}
\end{lemma}
The proof is straightforward.

%
By a tedious computation, the following is confirmed:
\begin{corollary}
	For an abundant system, \eqref{eqn:curvature} 
	and~\eqref{eqn:structural.equation} imply~\eqref{eq:SIC:V}.
\end{corollary}

\noindent Therefore, for abundant systems, the integrability conditions for 
the potential $V$, its trace-free conformal Killing tensors $C_{ij}$ and 
their respective scalar parts $W$ are equivalent to~\eqref{eqn:tau}, 
\eqref{eqn:strange.equation}, \eqref{eqn:curvature} 
and~\eqref{eqn:structural.equation}.

\subsection{Non-linear prolongation equations for the structure tensor}

We have found the prolongation~\eqref{eq:prolongation:V} for the potential 
and the prolongation~\eqref{eq:prolongation:K} for the trace-free conformal 
Killing tensors in a conformally superintegrable system. Now we show that 
these imply a third, non-linear prolongation for the structure tensor 
$T_{ijk}$, which expresses covariant derivatives of $T_{ijk}$ 
polynomially in terms of $T_{ijk}$ and the Ricci tensor $R_{ij}$.
\begin{proposition} 
\label{prop:nonlinear.prolongation}
 For an abundant conformally superintegrable potential in dimension 
 $n\geqslant3$, 
 the primary structure tensor $T_{ijk}$, decomposed according 
 to~\eqref{eq:T(S,t)} as
 \[
  T_{ijk}
  = S_{ijk} + \young(ij)\left( \bar t_ig_{jk} - \frac 1ng_{ij}\bar t_k 
  \right)\,,
 \]
 satisfies the following non-linear prolongation:
 \begin{subequations}\label{eqn:nonlinear.prolongation}
 \begin{align}
  \nabla_i\bar{t}_j &=  \left(
		\tfrac{3}{(n-2)}\,R_{ij}
		+\tfrac{1}{3(n-2)}\,{S_{i}}^{ab}S_{jab}
		+\tfrac{1}{3}\,\bar{t}_i\bar{t}_j
		\right)_\circ
		+\frac{1}{n} g_{ij} \nabla^{a}\bar{t}_a
		\label{eqn:tij.bar} \\
  \nabla^a\bar{t}_a &= \tfrac{3}{2(n-1)}\,R
		+\tfrac{3n+2}{6(n-1)(n+2)}S^{abc}S_{abc}
		-\tfrac{(n-2)}{6}\,\bar{t}^a\bar{t}_a
		\label{eqn:taa.bar} \\
  \nabla_lS_{ijk} &= \tfrac{1}{18}\,{\young(ijk)}_\circ
		\biggl(
		  \,S\indices{_{i l}^{a}} S_{j k a}
		  +3\,S_{i j l} \bar{t}_{k}
		  +S_{i j k} \bar{t}_{l}
		    +\left(
			\tfrac{4}{n-2}\,S\indices{_j^{ab}} S_{kab}
			-3\,S_{j k a} \bar{t}^{a}
		     \right)\,g_{i l}
		\biggr)
		\label{eqn:Sijk.l}
 \end{align}
 \end{subequations}
\end{proposition}

\begin{remark}
 Remarkably, the system~\eqref{eqn:nonlinear.prolongation} is the same as the 
 one found for \emph{properly} superintegrable systems in 
 \cite{Kress&Schoebel&Vollmer}, with the curvature term in 
 \eqref{eqn:tij.bar} replaced using \eqref{eqn:tau}. Here we leave the 
 curvature term, in order to not re-introduce $\tau_{ij}$ 
 into~\eqref{eqn:nonlinear.prolongation}.
\end{remark}

\begin{proof}[Proof of Proposition~\ref{prop:nonlinear.prolongation}]
 The proof is analogous to the corresponding result in 
 \cite{Kress&Schoebel&Vollmer}, but we provide more details here.
 Using \eqref{eqn:structural.equation}, define
 \begin{multline*}
	E\indices{_{ijk}^{mn}_l}
	= (n+1)\,\young(k,l)
	\Bigl(
	P\indices{_{ijk}^{mn}_{,l}}
	+P\indices{_{ijk}^{pq}}P\indices{_{pql}^{mn}}
	\Bigr)
	\\
	-\young(ij)\young(mn)\young(k,l)\,\left(
           P\indices{_{iak}^{ma}_{,l}} + P\indices{_{iak}^{pq}} 
           P\indices{_{pql}^{ma}}
        \right)g^n_j\,,
 \end{multline*}
 where we observe that $E\indices{_{ijk}^{mn}_l}$ is symmetric in $(i,j)$ and 
 in $(m,n)$, antisymmetric in $(k,l)$, and pure trace.
 Its trace components are given by
 \begin{align*}
  E^{(1)}_{kmnl} &= E\indices{^a_{akmnl}}\,, &
  E^{(2)}_{ijkl} &= E\indices{^a_{ja}^{mn}_l}\,, &
  E^{(3)}_{ijnl} &= E\indices{_{ija}^a_{nl}}\,.
 \end{align*}
 These are not independent, satisfying the relation
 \[
	2(n-2)\lb E^{(2)}_{ijkl} - E^{(3)}_{kjil} \rb
	= (n^2-2n-2)\,E^{(1)}_{lkji}\,.
 \]
 This equation allows us to express the trace-free part of 
 $\nabla_{l}S_{ijk}$ in terms of $T_{ijk}$ and the metric, yielding the 
 trace-free part of~\eqref{eqn:Sijk.l}.
 The trace part $\nabla^{a}S_{ija}$ is obtained from 
 ${E^{(1)}}\indices{_{ija}^a}$, which implies
 \[
  \nabla^aS_{ija} = \tfrac{2n}{3(n-2)}\,{S_{i}}^{ab}S_{jab}
		-\tfrac{n}{3}\,S_{ija}\bar{t}^a
		-\tfrac{2}{3(n-2)}\,g_{ij}S^{abc}S_{abc}\,.
 \]
 This confirms \eqref{eqn:Sijk.l}.
 The two other equations, \eqref{eqn:tij.bar} and~\eqref{eqn:taa.bar}, are 
 obtained from~\eqref{eqn:curvature} by contraction.
 Contracting~\eqref{eqn:curvature} in $i$ and $k$, we find \eqref{eqn:tij.bar}.
 Contracting~\eqref{eqn:curvature} further, in $(i,k)$ as well as $(j,l)$, we 
 find~\eqref{eqn:taa.bar}.
\end{proof}

\noindent Remarkably, the system~\eqref{eqn:nonlinear.prolongation} is 
already conformally invariant, which we show in detail 
Section~\ref{sec:invariant.nonlinear.prolongation}.
Indeed, in Equations~\eqref{eqn:tij.bar} and~\eqref{eqn:taa.bar} the terms 
involving $\bar{t}$ absorb the transformation behaviour of $R_{ij}$ 
under~\eqref{eqn:trafo.Ricci.curvature}. 

\subsection{The integrability conditions for abundant systems}
In addition to the integrability conditions for Killing tensors in abundant 
systems, we have two more equations, namely~\eqref{eqn:tau} 
and~\eqref{eqn:strange.equation}. Note that only~\eqref{eqn:tau} involves the 
secondary structure tensor $\tau_{ij}$, as it allows us to 
express~$\tau_{ij}$ in terms of the structure tensor $T_{ijk}$.

\begin{lemma}\label{la:Ricci.equation}~

\noindent (i) The non-linear prolongation~\eqref{eqn:nonlinear.prolongation} 
implies Equation~\eqref{eqn:strange.equation}.

\noindent (ii)
 For abundant systems in dimension $n\geqslant3$, the equations of the 
 non-linear prolongation~\eqref{eqn:nonlinear.prolongation}, together with
 \begin{equation}\label{eqn:Weyl.condition.abundant}
  W_{ijkl} = {\young(ik,jl)}^*_\circ S\indices{_{ik}^a}S_{jla} = 0\,,
 \end{equation}
 are equivalent to the integrability condition~\eqref{prop:SIC:K}.

\noindent (iii)
 With~\eqref{eqn:nonlinear.prolongation}, Equation~\eqref{eqn:tau} becomes
 \begin{equation}\label{eqn:Ricci.condition.strange.eqn}
  3\,(n-2)\,\left( \frac12\,\tau_{ij} + \mathring{\mathsf{P}}_{ij} \right)
  = \left(
          (n-2)(S_{ija}\bar{t}^a+\bar{t}_i\bar{t}_j)
	  - S\indices{_i^{ab}}S_{jab}
     \right)_\circ
 \end{equation}
\end{lemma}

\begin{proof}
 For the first part simply resubstitute~\eqref{eqn:nonlinear.prolongation} 
 into~\eqref{eqn:strange.equation}.
 The proof of part (ii), namely of 
 Equation~\eqref{eqn:Weyl.condition.abundant}, is analogous to that of 
 Theorem 5.9 in \cite{Kress&Schoebel&Vollmer}.
 Finally, for part (iii), resubstitute~\eqref{eqn:nonlinear.prolongation} 
 into Equation~\eqref{eqn:tau}. We 
 obtain~\eqref{eqn:Ricci.condition.strange.eqn}.
\end{proof}

\noindent As an immediate consequence of \eqref{eqn:nonlinear.prolongation} 
and \eqref{eqn:Weyl.condition.abundant}, we obtain the following obstruction 
on the geometry underlying an abundant system.

\begin{corollary}\label{cor:abundant.is.conformally.flat}
 Abundant conformally superintegrable systems can only exist on conformally 
 flat manifolds.
\end{corollary}
\begin{proof}
It follows immediately from \eqref{eqn:Weyl.condition.abundant} that a 
conformally superintegrable system can only exist on a Weyl flat manifold. 
Therefore, for dimension $n\geqslant4$, they can exist only on conformally 
flat 
manifolds. In dimension~2, any metric is conformally flat.
We are therefore left with the case $n=3$.
Using standard scale, i.e.\ $\bar t=0$, Equation~\eqref{eqn:Sijk.l} yields 
that
\[
  S\indices{_i^{ab}}S_{jab} - \frac3{20}\,S^{abc}S_{abc}\,g_{ij}
\]
is a Codazzi tensor. The claim then follows from the Weyl-Schouten Theorem.
\end{proof}

\section{Equivalence classes of abundant superintegrable systems}
\label{sec:ccsis}

So far, we have considered conformally superintegrable systems whose 
underlying geometry is a \emph{\mbox{(pseudo-)}Riemannian manifold}. We now 
turn towards conformal equivalence classes, i.e.\ towards c-superintegrable 
systems on conformal manifolds. For such systems, $S_{ijk}$ is the 
conformally invariant structure tensor.
According to~\eqref{eqn:Ricci.condition.strange.eqn} or~\eqref{eqn:tau}, the 
secondary structure tensor $\tau_{ij}$ is determined by $T_{ijk}$ and the 
Ricci curvature.
Table~\ref{tab:analogy} contrasts the setting of properly and conformally 
superintegrable systems as opposed to c-superintegrable systems.

\begin{table}[ht]
	\centering
	\textbf{Comparison of abundant second order systems:\\
		Proper vs.\ conformal vs.\ c-superintegrability}
	\begin{tabular}{l|c|c|c}
		\toprule
		Type
		& proper
		& conformal
		& c-superintegrability \\
		& superintegrability
		& superintegrability
		& \\
		\midrule
		Geometry
		& pseudo-Riem.
		& pseudo-Riem.
		& conformal\\
		& metric & metric & metric \\
		& $g_{ij}$
		& $g_{ij}$
		& $\mathbf{g}_{ij}$ \\
		\midrule
		Constants of motion
		& (proper) integrals
		& \multicolumn{2}{c}{conformal integrals} \\
		\midrule
		Primary
		& \multicolumn{2}{c|}{
			$T_{ijk}=S_{ijk}+{\young(ij)}_\circ \bar t_i g_{jk}$}
		& $S_{ijk}=\mathring T_{ijk}$
		\\
		structure tensor
		& \multicolumn{2}{c|}{~}
		&
		\\
		\midrule
		Secondary
		& $\tau_{ij}=0$
		& given by~\eqref{eqn:Ricci.condition.strange.eqn}
		& none (not 
		\\
		structure tensor
		&&& conformally invariant)
		\\
		\bottomrule
	\end{tabular}
	\bigskip
	\caption{%
		Synopsis of the main objects in proper, conformal and 
		c-superintegrability.
	}
	\label{tab:analogy}
\end{table}

\subsection{Obstructions to the integrability 
of the non-linear prolongation~\eqref{eqn:nonlinear.prolongation}}
Consider the non-linear prolongation~\eqref{eqn:nonlinear.prolongation} of 
PDEs for $\bar t$ and $S_{ijk}$.
We now investigate the integrability conditions for this system.
The prolongation equations~\eqref{eqn:nonlinear.prolongation} are non-linear 
in the components of $S_{ijk}$ and $\bar t_i$. Therefore the Ricci conditions,
\begin{subequations}\label{eqn:riemann.conditions}
\begin{align}
	\young(l,m)\,\nabla_m\nabla_l S_{ijk}
	&= \young(ijk) R_{ialm} S\indices{^a_{jk}}
	\label{eqn:riemann.condition.S}
	\\
	\young(jk)\,\nabla_k\nabla_j \bar t_i
	&= R_{iajk} \bar t^a\,,
	\label{eqn:riemann.condition.t}
\end{align}
\end{subequations}
are necessary but need not be sufficient for the integrability of 
\eqref{eqn:nonlinear.prolongation}.
Sufficiency is guaranteed if not only the Ricci condition but also all of its 
differential consequences are satisfied in a given point~$x_0$ 
\cite{Goldschmidt1967}.
We find that the integrability conditions, of which a priori there can be 
infinitely many, reduce to a single algebraic equation of the 
form~\eqref{eqn:master}.

\begin{theorem}~
	
	\noindent(i) If Equation~\eqref{eqn:Weyl.condition.abundant} holds, then 
	the integrability conditions~\eqref{eqn:riemann.conditions} of 
	\eqref{eqn:nonlinear.prolongation} are satisfied.
	
	\noindent(ii) Let~$M$ be a conformally flat, pseudo-Riemannian manifold 
	of dimension $n\geqslant3$ and let $x_0\in M$ be a point on this manifold.
	Then any solution $\Psi_{ijk}=S_{ijk}(x_0)$ of~\eqref{eqn:master} 
	together with the arbitrary initial values $\bar t(x_0)$ and $\nabla\bar 
	t(x_0)$ can be extended, in a neighborhood of $x_0$, to solutions 
	$S_{ijk}(x)$ and $\bar t(x)$ of the non-linear 
	prolongation~\eqref{eqn:nonlinear.prolongation}.
\end{theorem}
\begin{proof}
(i) Since the integrability condition cannot depend on $\bar t$, we may 
w.l.o.g.\ perform a conformal transformation such that the transformed system 
is in standard scale, $\bar t=0$.
As a result, \eqref{eqn:tij.bar} and~\eqref{eqn:taa.bar} determine the Ricci 
curvature tensor in this scale. We shall comment on this after finishing the 
proof.

First, let us investigate the integrability condition for \eqref{eqn:Sijk.l}, 
which by virtue of the aforementioned transformation has turned 
into
\begin{equation}\label{eqn:Sijk.l.t=0}
	\nabla_lS_{ijk} 
	= \tfrac{1}{18}\,{\young(ijk)}_\circ
	\biggl(
		\,S\indices{_{i l}^{a}} S_{j k a}
		+ \tfrac{4}{n-2}\,S_{j}^{a b} S_{k a b}\,g_{i l}
	\biggr).
\end{equation}
Using~\eqref{eqn:tij.bar}, \eqref{eqn:taa.bar} and 
Weyl-flatness~\eqref{eqn:Weyl.condition.abundant}, we replace the 
Riemann curvature in~\eqref{eqn:riemann.condition.S} by a quadratic 
expression in $S$.
Due to~\eqref{eqn:Weyl.condition.abundant} in combination with the 
non-linear prolongation~\eqref{eqn:nonlinear.prolongation} 
and~\eqref{eqn:Ricci.condition.strange.eqn}, 
Equation~\eqref{eqn:riemann.condition.S} 
is equivalent to the conformally invariant condition
\[
 \young(j,k)\,\left(  S^{abc}S_{ija}S_{kbc} \right)_\circ = 0\,,
\]
which is confirmed to be an algebraic consequence 
of~\eqref{eqn:Weyl.condition.abundant} by way of  
contracting~\eqref{eqn:Weyl.condition.abundant} with $S$.
We have therefore established~\eqref{eqn:Weyl.condition.abundant} as the only 
first order integrability condition of~\eqref{eqn:Sijk.l}. By an analogous 
computation we then confirm that all Ricci identities
of~\eqref{eqn:nonlinear.prolongation} are satisfied 
if~\eqref{eqn:Weyl.condition.abundant} holds.
As explained earlier, however, the first order integrability 
conditions~\eqref{eqn:riemann.conditions} need 
not be sufficient for the integrability of~\eqref{eqn:nonlinear.prolongation}.
We now proceed to show their sufficiency.

(ii) 
In order to find sufficient, pointwise integrability criteria, higher order 
integrability conditions have to be taken into account. Concretely, all 
differential consequences of~\eqref{eqn:Weyl.condition.abundant} need to be 
satisfied in a fixed point $x_0$ in order to allow us to extend $S_{ijk}$ and 
$\bar t$ such that the extensions satisfy~\eqref{eqn:nonlinear.prolongation} 
in a neighborhood of $x_0$.
Taking a covariant derivative of \eqref{eqn:Weyl.condition.abundant} and 
replacing derivatives of $S_{ijk}$ by \eqref{eqn:Sijk.l.t=0}, we find an 
algebraic condition on $S_{ijk}$.
Using~\eqref{eqn:Weyl.condition.abundant} again, it can be verified that this 
second order condition is an algebraic consequence of the first order one, 
i.e.\ of~\eqref{eqn:Weyl.condition.abundant}. This concludes the proof.
\end{proof}

\noindent If instead of a pseudo-Riemannian manifold we consider a conformal 
manifold, we obtain the following statement.
\begin{corollary}\label{cor:extend.S}
	Consider a flat conformal manifold $(M,c)$ of 
	dimension~$n\geqslant3$. 
	Then a solution of~\eqref{eqn:master} can be extended to a totally 
	symmetric and trace-free tensor field $S_{ijk}(x)$ in a neighborhood of a 
	point $x_0\in M$ such that there exists $g\in c$ 
	and~\eqref{eqn:Sijk.l.t=0} holds where $\nabla$ is the 
	Levi-Civita connection of $g$.
\end{corollary}
\begin{proof}
	Note that~\eqref{eqn:master} is an invariant condition.
	The statement follows from statement~(ii) in the theorem, making use of 
	the fact 
	that there exists a flat conformal scale choice, which removes the 
	curvature from the non-linear prolongation system 
	\eqref{eqn:nonlinear.prolongation}.
	If \eqref{eqn:master} holds in $x_0\in M$ for this case, the system is 
	integrable in a neighborhood of $x_0$. Note that the specific scale 
	choice is technical and irrelevant for the statement.
\end{proof}

\begin{corollary}\label{cor:no.Weyl.3D}
	In dimension $n=3$, \eqref{eqn:nonlinear.prolongation} can be integrated 
	for any $S_{ijk}(x_0)$.
\end{corollary}
\begin{proof}
	Note that \eqref{eqn:Weyl.condition.abundant} has Weyl symmetry and thus 
	vanishes trivially in dimension $n=3$.
\end{proof}

\begin{remark}
	The corollary is consistent with the existing literature.
	For dimension~3, \cite{Capel_phdthesis} finds that conformally 
	superintegrable systems lead to the irreducible representations 
	${\Yboxdim{7pt}\yng(3)}_\circ$ and $\Yboxdim{7pt}\yng(1)$ of the rotation 
	group.
	Specifically, these representations are of  dimension seven and three. 
	These correspond to our tensors $S_{ijk}$ and $\bar t_i$ (or $t_i$), 
	respectively.
	Moreover, this reference shows that every element of this $7$-dimensional 
	representation of the rotation group corresponds to at least one 
	superintegrable system on a $3$-dimensional conformally flat geometry, in 
	line with the statement of Corollary~\ref{cor:no.Weyl.3D}.
\end{remark}

\subsection{The conformal factors between standard scale, proper scale and 
flat scale.}\label{sec:shift}
In the standard scale $\bar t=0$, Equations~\eqref{eqn:tij.bar} 
and~\eqref{eqn:taa.bar} become algebraic conditions on the curvature tensor. 
In contrast, if a flat scale exists, exactly the curvature terms disappear. 
In an arbitrary scale choice we have the following formula for the curvature, 
in terms of the conformally invariant tensor $S_{ijk}$ and the conformal 
scale function $\cs$.

\begin{proposition}
	The Ricci tensor satisfies
	\begin{align}\label{eqn:Ricci.tensor}
		\mathrm{R}_{ij}
		&= -\frac19\,\left(
		S\indices{_i^{ab}}S_{jab}
		+\frac{n}{3(n+2)}\,g_{ij}S^{abc}S_{abc}
		\right)
		+ (n-2)\,\left(
		\frac{\mathring\nabla^2_{ij}\cs}{\cs}
		-\frac{4(n-1)}{n(n-2)^2}\,
		\frac{\Delta\cs^{\frac{n-2}2}}{\cs^{\frac{n-2}2}}\,g_{ij}
		\right)
	\end{align}
	where $\cs=\exp(-\frac13\bar t)$ is the conformal scaling from 
	Definition~\ref{def:scaling}.
\end{proposition}
%
\begin{proof}
	Solve~\eqref{eqn:tij.bar} and~\eqref{eqn:taa.bar} for the Ricci tensor.
\end{proof}

\noindent Note that on the right hand side of~\eqref{eqn:Ricci.tensor}, the 
first term is invariant (up to a conformal factor), while the second term 
vanishes in standard scale.
In a flat scale, $R_{ijkl}=0$, and the left hand side vanishes.
The function $\bar t=\rho$ arising from the primary structure tensor  
measures `how far' the flat scale is from the standard scale. According 
to~\eqref{eq:flat.scale.metric}, the conformal scale factor between the two 
scales is
\begin{equation}
	\theta:=e^{-\frac13\rho}
\end{equation}
(note that now we transform from the scale defined by $h_{ij}$ back to the 
standard scale with metric~$\mathfrak g_{ij})$.
From~\eqref{eqn:Ricci.tensor} we see that the flat conformal scales $\theta$ 
are determined by $S_{ijk}$. The following lemma makes this explicit.
\begin{lemma}
	For any conformal class of abundant superintegrable systems there are 
	functions $\theta>0$ satisfying
	\begin{subequations}\label{eqn:shift}
		\begin{align}
			\label{eqn:shift.tracefree}
			\mathring\nabla^2_{ij}\theta
			&= -\frac{\left(
				S\indices{_i^{ab}}S_{jab}
				\right)_\circ}{9(n-2)}\,\theta
			\\ 
			\label{eqn:shift.trace}
			\Delta\theta^{1-\frac{n}{2}}
			&= \frac{(n-2)(3n+2)}{36(n-1)(n+2)}\,S^{abc}S_{abc}\,
			\theta^{1-\frac{n}{2}}\,.
		\end{align}
	\end{subequations}
	where $\mathring\nabla^2$ is the flat, trace-free Hessian and $\Delta$ 
	the flat Laplace operator, and where $S_{ijk}$ is the conformal structure 
	tensor in the flat scale.
\end{lemma}
\begin{proof}
	Let the function $\bar t=\rho$ be the trace of the structure tensor after 
	a conformal transformation to a flat scale. Then a conformal 
	transformation with $\Upsilon=-\frac{\nabla\rho}{\rho}$ takes us back to 
	the standard scale. Rewriting~\eqref{eqn:tij.bar} and~\eqref{eqn:taa.bar} 
	in terms of $\theta$, and decomposition of the result into its trace-free 
	and trace parts, confirms the claim.
\end{proof}
\noindent Note that, by Corollary~\ref{cor:abundant.is.conformally.flat}, 
solutions $\theta$ do exist. However, solutions are not unique. Indeed, a 
positive constant scalar multiple of $\theta$ is again a solution, and in 
general more solutions can exist. For instance, in Table~\ref{tab:3D} the 
systems III and V are both in flat scale, and they are conformally 
equivalent, see \cite{Capel_phdthesis}.
Two different solutions $\theta$ represent two different flat conformally 
superintegrable systems within the same conformal class. They share the same 
$\mathring T\indices{_{ij}^k}=S\indices{_{ij}^k}$ but the traces of $T_{ijk}$ 
will be different unless the $\bar t$ differ only by an additive constant.
In order to understand the space of solutions~$\theta$ better, let us 
study~\eqref{eqn:shift.tracefree} further, ignoring the additional 
constraint~\eqref{eqn:shift.trace} for a moment. We find:
\begin{lemma}
	Equation~\eqref{eqn:shift.tracefree} has the linear prolongation
	\begin{subequations}\label{eqn:prolongation.theta}
		\begin{align}
			\label{eqn:prolongation.theta.1}
			\nabla^2_{ij}\theta &=
			-\frac{(S\indices{_i^{ab}}S_{jab})_\circ}{9(n-2)}\,\theta
			+\frac1n\,g_{ij}\,\Delta\theta
			\\
			\label{eqn:prolongation.theta.2}
			(\Delta\theta)_{,k} &=
			\frac{n}{9(n-2)}\,S^{abc}S_{kab}\,\theta_{,c}
			+\frac{3n+2}{27(n-1)(n-2)}\,
			S\indices{_k^{ab}}S\indices{_a^{cd}}S_{bcd}\,\theta
		\end{align}
	\end{subequations}
	where $\nabla^2$ is the flat Hessian and $\Delta$ 
	the flat Laplace operator, and where $S_{ijk}$ is the conformal structure 
	tensor in the flat scale.
	The integrability conditions for~\eqref{eqn:prolongation.theta} are 
	equivalent to~\eqref{eqn:Weyl.condition.abundant}.
\end{lemma}
\begin{proof}
	Equations~\eqref{eqn:prolongation.theta} are obtained in formal analogy 
	to~\eqref{eq:prolongation:V}, where formally $T\equiv0$.
	Its integrability conditions are satisfied due 
	to~\eqref{eqn:Weyl.condition.abundant},
	as~\eqref{eqn:prolongation.theta.1} is a special case 
	of~\eqref{eqn:tij.bar}.
\end{proof}

\noindent Equations~\eqref{eqn:prolongation.theta} are a linear prolongation 
system for $\theta$ and its coefficients are determined by $S_{ijk}$. 
Therefore the solutions $\theta$ lie in an (at most) $(n+2)$-dimensional 
linear space $\mathcal B$, determined by the values of 
$\Delta\theta,\nabla\theta$ and $\theta$ in a fixed point.
The additional constraint~\eqref{eqn:shift.trace} defines a quadric in 
$\mathcal B$.

\subsection{Classifying the conformal classes of conformally superintegrable 
systems}
In the previous section we have found algebraic integrability conditions 
whose form is the same for any conformally superintegrable system within a 
class (in the next section we reformulate them as equivariant conditions).
As initial data we need to specify $\Psi_{ijk}=S_{ijk}(x_0)$. For a 
conformally flat geometry we may choose a flat metric, which facilitates 
determining a solution for $S_{ijk}$.
In order to reconstruct an abundant conformally superintegrable system from 
the initial data $\Psi_{ijk}=S_{ijk}(x_0)$, we recall that an abundant 
superintegrable system requires a conformally flat metric $g=\phi^2\,h$ where 
$h$ is the flat metric.
Then we can use the following procedure.
\begin{enumerate}
	\item
	Let $\Psi_{ijk}=S_{ijk}(x_0)$ be the initial data given in a point $x_0$. 
	Assume that $\Psi_{ijk}$ solves the algebraic 
	condition~\eqref{eqn:master}. The $\Psi_{ijk}$ do not depend on $\phi$, 
	and if $\Psi_{ijk}$ are solutions then so are $k\Psi_{ijk}$ for $k\ne0$.
	\item
	We extend the initial data $\Psi_{ijk}$ to a solution in a neighborhood 
	of $x_0$ such that the non-linear 
	prolongation~\eqref{eqn:nonlinear.prolongation} holds.
	This is possible by virtue of Corollary~\ref{cor:extend.S}. For a 
	concrete computation we should choose some conformal scale, and the flat 
	scale is a reasonable choice. We then need to specify the initial data 
	$\nabla\rho(x_0)$ and $\rho(x_0)$ in addition to $\Psi_{ijk}$.
	\item
	This yields $T_{ijk}$ up to a conformal transformation. 
	Integrating the Wilczynski equation~\eqref{eqn:Wilczynski.conformal} for 
	$V$ in the specific scale given by $\bar t=\rho$, and computing 
	\smash{$\mathbf{v}=e^{\frac{2\rho}{3}}V$}, we find the conformally 
	invariant potential as an $(n+2)$-parameter family of densities of 
	weight~$-2$. This is the space~$\mathcal V^\text{max}$ The 
	space~$\mathcal C^\text{max}$ of conformal Killing tensors is similarly 
	obtained by integration of~\eqref{eq:prolongation:K}.
\end{enumerate}

\noindent Since all integrability conditions are satisfied generically, we 
find at least $\tfrac12\,n(n+1)$ many linearly independent conformal 
integrals.
We address their functional independence below in 
Lemma~\ref{la:functional.dependence}. 
\bigskip

\noindent The procedure just outlined allows one to reconstruct an abundant 
c-superintegrable system from the given initial data and the knowledge of the 
underlying conformal metric up to the choice of the potential from 
$\mathcal{V}^\text{max}$.
We recall Assumption~\ref{assumption:killing.space}, but remark that with 
Lemma~\ref{la:functional.dependence} below, we are able to restrict the 
space~$\mathcal C^\text{max}$ in order to obtain $2n-1$ \emph{functionally} 
independent conformal integrals.
Let us reinterpret the aforesaid in the light of classifying 
c-superintegrable systems.
In \cite[Theorem 6.4]{Kress&Schoebel&Vollmer} it is shown that the 
classification space for irreducible non-degenerate superintegrable
systems on a (pseudo-)Riemannian manifold $M$ with analytic metric is a 
quasi-projective subvariety $\mathcal U\subset G_{2n-1}(\mathcal K(M))$ in 
the Grassmannian of $(2n-1)$-dimensional 
subspaces in the space $\mathcal K(M)$ of Killing tensors on $M$.
%
%
Since any c-superintegrable system admits at least one 
system in proper scale, it follows that the classification space of 
irreducible non-degenerate c-superintegrable systems with analytic metric is 
the quotient
$\tilde{\mathcal{U}}=\mathcal U/\mathrm{Conf}(M)$.
For non-degenerate irreducible conformally superintegrable systems on 
analytic metrics, the classification space is a fibre bundle over 
$\tilde{\mathcal{U}}$.
\bigskip

\noindent The following lemma was proven for properly superintegrable systems 
in reference \cite{Kress&Schoebel&Vollmer}. We adapt it for conformal systems.
\begin{lemma}\label{la:functional.dependence}
	Let $C^{(\alpha)}$ be $2n-2$ linearly independent, trace-free conformal 
	Killing tensors satisfying the integrability conditions
	\eqref{eqn:K.integrability} for \eqref{eq:prolongation:K}, and
	\eqref{eq:SIC:V} for~\eqref{eq:prolongation:V}.  Then, in the linear 
	space $\mathcal{V}^\text{max}$ of solutions $V$ to 
	Equation~\eqref{eq:prolongation:V}, those $V$ that give rise to
	functionally dependent integrals are confined to an affine subspace of 
	$\mathcal{V}$ with non-empty complement.
\end{lemma}

\begin{proof}
	Suppose the integrals \eqref{eq:quadratic} were functionally
	dependent. Then there is a function $\varphi:\R^{2n-2}\to\mathbb 
	R$, non-zero in an open subset of its domain, such that
	\begin{equation}\label{eq:linear.independence.functional}
	\varphi(F^{(1)},\ldots,F^{(2n-2)})=0
	\end{equation}
	This implies the infinitesimal condition
	\begin{equation}\label{eq:linear.independence.infinitesimal}
		\sum_{\alpha=1}^{2n-2}
		\lambda_{(\alpha)} dF^{(\alpha)}
		=0,
	\end{equation}
	where
	\begin{align*}
		\lambda_{(\alpha)} 
		&=\frac{\partial\varphi}{\partial F^{(\alpha)}}
		(F^{(\alpha)})&
		dF^{(\alpha)}
		&
		=\frac{\partial F^{(\alpha)}}{\partial x^k}dx^k
		+\frac{\partial F^{(\alpha)}}{\partial p^k}dp^k\,.
	\end{align*}
	By a direct computation we find
	\begin{align*}
		\frac{\partial F^{(\alpha)}}{\partial x^k}
		&= C^{(\alpha)}_{ij,k}p^ip^j+V^{(\alpha)}_{,k} &
		\frac{\partial F^{(\alpha)}}{\partial p^k}
		&= 2C^{(\alpha)}_{jk}p^j.
	\end{align*}
	Separating the components of \eqref{eq:linear.independence.infinitesimal} 
	and substituting~\eqref{eqn:dV.KdV}, we conclude
	\begin{subequations}\label{eq:independence}
		\begin{align}
			\label{eq:independence:dx}
			\sum_\alpha\lambda_{(\alpha)}\left(
				C^{(\alpha)}_{ij,k}p^ip^j
				+C^{(\alpha)}_{jk}V^{,j}
				+\omega^{(\alpha)}_kV
			\right)&=0\\
			\label{eq:independence:dp}
			\sum_\alpha\lambda_{(\alpha)} C^{(\alpha)}_{jk}p^j&=0.
		\end{align}
	\end{subequations}	
	Invoking \eqref{eq:prolongation:K}, we obtain
	\begin{align*}
		C^{(\alpha)}_{ij,k}p^ip^j
		&=\frac23
		\left(
		T\indices{^a_{ij}}C^{(\alpha)}_{ka}
		- T\indices{^a_{kj}}C^{(\alpha)}_{ia}
		\right)
		p^ip^j
		-\frac{2}{3n}\,\left(
			T\indices{_k^{ab}}C_{ab}^{(\alpha)}
			-t^aC_{ak}^{(\alpha)}
		\right)\,p^cp_c\,.
	\end{align*}
	Multiplying with $\lambda_{(\alpha)}$ and summing over $\alpha$, we find, 
	using~\eqref{eq:independence:dp} and the decomposition~\eqref{eq:T(S,t)},
	\begin{align*}
		\sum_\alpha
		\lambda_{(\alpha)}
		C^{(\alpha)}_{ij,k}p^ip^j
		&=\frac23
		\sum_\alpha \lambda_{(\alpha)}
		\left(
			S\indices{^a_{ij}}C^{(\alpha)}_{ka}
			-\frac1n\,t^ag_{ij}C^{(\alpha)}_{ka}
			-\frac1n\,S\indices{_k^{ab}}C_{ab}^{(\alpha)}g_{ij}
		\right)\,p^ip^j\,.
	\end{align*}
	Substituting this back into \eqref{eq:independence:dx}, 
	invoking~\eqref{eqn:omega.from.divC}, and using again the 
	decomposition~\eqref{eq:T(S,t)}, we conclude
	\begin{equation}\label{eqn:kernel.Xi}
		C_{ab}\eta^{kab} = 0\,,
	\end{equation}
	where we use the abbreviations
	\[
		C_{ab}:=\sum_\alpha \lambda_{(\alpha)} C^{(\alpha)}_{ab}
	\]
	and
	\[
	  \eta^{kab}
	  =g^{ka}
	  \left(
		  S\indices{^b_{ij}}p^ip^j
		  -\frac{1}{n}t^bp^cp_c
		  +\frac32\,V^{,b}
		  -\frac{1}{n}\,\frac{(n+1)(n-2)}{(n-1)(n+2)}t^b\,V
	  \right)
	  +\frac{1}{n}\,S^{kab}\,(V-p^cp_c)\,.
	\]
	Note that $C(x_0)\not=0$.  Indeed, otherwise the $C^{(\alpha)}_{ab}(x_0)$ would
	be linearly dependent,
	\[
		\sum_{\alpha}k_{(\alpha)}C^{(\alpha)}_{ab}(x_0)=0,
		\qquad
		k_{(\alpha)}=\lambda_{(\alpha)}(x_0).
	\]
	Because of~\eqref{eq:prolongation:K} the derivatives of
	$C^{(\alpha)}_{ab}$ are linearly dependent at $x_0$, with the same
	constants $k_{(\alpha)}$.  Iterated application
	of~\eqref{eq:prolongation:K} to higher derivatives shows that the same is
	true for all higher derivatives. It readily follows that
	$\sum_\alpha k_{(\alpha)}C^{(\alpha)}=0$ everywhere, which contradicts the
	linear independence of the conformal Killing tensors $C^{(\alpha)}$.

	Now, for $x_0\in M$, consider the mapping $\Xi:T^{\otimes3}_{x_0}M\to 
	T_{x_0}M$, given by	contracting with $C_{ab}$,
	\[
		\Xi(\eta^{kab})
		= C_{ab}\eta^{kab}\,.
	\]
	By virtue of Equation~\eqref{eqn:kernel.Xi}, we conclude that for any 
	potential~$V\in\mathcal{V}^\text{max}$
	\[
		\eta^{kab}(x_0)\,\,\in\,\ker \Xi\,.
	\]
	Using the linearity of the kernel, we conclude further that
	\[
		\left[
			\frac32\,g^{kb}V^{,a}
			+\frac1n\,\left(
					S^{kab}
					-\frac{(n+1)(n-2)}{(n-1)(n+2)}g^{kb}t^a
				\right)\,V
		\right]_{x_0}
		\,\,\in\,\ker \Xi\,.
	\]
	Choosing $V(x_0)=0$, we obtain that
	\[
		C\indices{_k^a}V_{,a}(x_0)=0
	\]
	for any choice of $V_{,a}(x_0)$, contradicting that $C(x_0)\not=0$.
\end{proof}

\begin{theorem}
	Abundant conformally superintegrable Hamiltonians with their 
	$(n+2)$-parameter family of potentials, and identified under 
	conformal transformations, are classified by~\eqref{eqn:master}.
\end{theorem}
\begin{proof}
	An abundant conformally superintegrable Hamiltonian with its 
	$(n+2)$-dimensional space $\mathcal V^\text{max}$ of all compatible 
	potentials, can be recovered from $S_{ijk}$ up to a conformal 
	transformation of superintegrable systems and every abundant system 
	satisfies~\eqref{eqn:master}.
\end{proof}

\subsection{Invariant formulation of the non-linear prolongation equations}
\label{sec:invariant.nonlinear.prolongation}
In this section we express the non-linear prolongation 
equations~\eqref{eqn:nonlinear.prolongation} in a conformally invariant way.
\begin{proposition}~
	
	\noindent(i)
	Equation~\eqref{eqn:tij.bar} is equivalent to
	\begin{equation}\label{eqn:tf.hessian.D}
		\mathring{\mathbb{H}}_{ij}\cs = -\frac{1}{9(n-2)}\,\left( 
		S\indices{_i^{ab}}S_{jab} \right)_\circ\,\cs\,,
	\end{equation}
	where $\cs=e^{-\frac13\,\bar t}$ and where $\mathring{\mathbb{H}}$ is 
	the conformally invariant trace-free Hessian, defined by
	\[
		\mathring{\mathbb{H}}_{ij}
		= \left(\nabla^2_{ij}-\schouten_{ij}\right)_\circ.
	\]
	
	\noindent(ii)
	Equation~\eqref{eqn:taa.bar} is equivalent to
	\begin{equation}\label{eqn:laplace.D}
		\mathbb{L}\cs^{1-\frac{n}{2}}
		= -\frac29\,\frac{3n+2}{n+2}\,S^{abc}S_{abc}\,\cs^{1-\frac{n}{2}}\,,
	\end{equation}
	where $\mathbb{L}$ denotes the conformal Laplacian,
	\[
	\mathbb{L} = -4\,\frac{n-1}{n-2}\,\Delta + R\,.
	\]
	
	\noindent(iii)
	Equation~\eqref{eqn:Sijk.l} is equivalent to
	\begin{equation}\label{eqn:nabla.S}
		\nabla^{\bar{t}}_lS_{ijk}
		= \frac13\,{\young(ijk)}_\circ
		\left(
		S\indices{_{il}^a}S_{jka}
		-\frac{4}{n-2}\,g_{kl}\,S\indices{_i^{ab}}S_{jab}\,,
		\right)
	\end{equation}
	where $\nabla^{\bar t}$ is the conformally equivariant Weyl connection
	\cite{Weyl1918} defined by
	\begin{equation}\label{eqn:weyl.connection}
		\nabla^{\bar t}_i\alpha_j
		= \nabla_i\alpha_j
		- \frac{m+1}{3}\,\bar t_i\alpha_j
		- \frac13\,\bar t_j\alpha_i
		+ \frac13\,\bar t^a\alpha_ag_{ij}\,,
	\end{equation}
	for $\alpha_j$ of conformal weight $m$, i.e.\ 
	$\alpha_j\mapsto\Omega^m\alpha_j$ under conformal 
	transformations. Here we have $m=-\frac23$.
\end{proposition}
\begin{proof}
%
%
Parts (i) and (ii) are straightforward.
For part (iii), apply~\eqref{eqn:weyl.connection} to~$S_{ijk}$,
\begin{equation*}
	\nabla^{\bar{t}}_lS_{ijk}
	= \nabla_lS_{ijk}
	- \frac{1}{18}\,{\young(ijk)}_\circ
	\biggl(
	3\,S_{i j l} \bar{t}_{k}
	+S_{i j k} \bar{t}_{l}
	-3\,g_{k l}\,S_{i j a}\bar{t}^{a}
	\biggr)\,.
\end{equation*}
A direct computation indeed confirms that $\nabla^{\bar t}$ is 
invariant under conformal changes up to multiplication by a factor: The 
replacement rules are $g\to\Omega^2g$ and $\bar t\to\bar t-3\ln|\Omega|$, as 
well as, respectively, $S_{ijk}\to\Omega^2 S_{ijk}$ or 
$\alpha_i\to\Omega^m\alpha_i$ with $m=-\frac23$.
\end{proof}

\noindent Two remarks are in place with regard to the above proposition.
First, note that the conformal weights required for the trace-free conformal 
Hessian and the conformal Laplacian are different, leading to different 
powers of the conformal scale function.
The second remark concerns the conformal invariance of the operators.
Note that, under conformal transformations with rescale function $\Omega$, we 
have
\[
\cs^{1-\frac{n}{2}}\to\Omega^{1-\tfrac{n}{2}}\,\cs^{1-\frac{n}{2}}\,,
\]
and the conformal invariance of $\mathbb{L}$ means
\[
\mathbb{L}\circ\Omega^{1-\frac{n}{2}} = 
\Omega^{-1-\frac{n}{2}}\circ\mathbb{L}\,,
\]
which is consistent as $S^{abc}S_{abc}\to\Omega^{-2}S^{abc}S_{abc}$.
Note that in the standard scale, i.e.\ for $\bar{t}=0$ resp.\ $\cs=1$, 
Equation~\eqref{eqn:laplace.D} is an expression for the scalar curvature in 
terms of $S_{ijk}$, and~\eqref{eqn:tf.hessian.D} for the Schouten tensor.

\begin{remark}
	To determine whether the non-linear 
	system~\eqref{eqn:nonlinear.prolongation} 
	is integrable, it is sufficient to know the invariants in 
	\eqref{eqn:tf.hessian.D}, \eqref{eqn:laplace.D} and \eqref{eqn:nabla.S} 
	as well as those in \eqref{eqn:prolongation.theta}, which are constructed 
	algebraically from $S$.
	These invariants are
	\begin{align*}
		A_{ijkl}
		&= {\young(ijk,l)}_\circ (S\indices{_{il}^a}S_{jka})
			+ {\young(ijkl)}_\circ (S\indices{_{il}^a}S_{jka})\,,
		\\
		B_k
		&= S\indices{_k^{ab}}S\indices{_a^{cd}}S_{bcd}\,,
		\\
		\Sigma_{ij}
		&= \young(ij) S\indices{_{i}^{ab}}S_{jab}\,.
	\end{align*}
	The last of these invariants has a nice geometric 
	interpretation.
	First, due to \eqref{eqn:tij.bar}, we have
	\[
		\mathring{\mathfrak{P}}_{ij}
		= \frac{1}{9(n-2)}\,\mathring\Sigma_{ij}\,,
	\]
	where the Schouten curvature in standard scale is denoted by 
	$\mathfrak{P}_{ij}$.
	The trace of the Schouten curvature satisfies
	\[
		\mathfrak{P}\indices{^a_a}
		= \frac{3n+2}{18(n-1)(n+2)}\,\Sigma\indices{^a_a}\,.
	\]
	The second invariant also has a geometric meaning. 
	It is easy to show that
	\begin{equation}\label{eqn:standard.scale.scalar.curvature.differential}
		\mathfrak{P}\indices{^a_{a,k}}
		= -\frac{(3n+2)\,B_k}{27(n-2)(n-1)}\,,
	\end{equation}
	and therefore $B_k=0$ characterises the case when the standard scale 
	system has constant scalar curvature.
	We also observe that $B_k$ and $\Sigma\indices{^a_a}$ are not 
	(differentially) independent.
\end{remark}

\subsection{Properly superintegrable systems on constant curvature manifolds}
\label{sec:proper.constant.curvature}
In reference \cite{Kress&Schoebel&Vollmer} abundant properly 
superintegrable systems are studied. These systems satisfy $\tau_{ij}=0$ due 
to Lemma~\ref{la:proper.then.tau0} 
and thus \eqref{eqn:Ricci.condition.strange.eqn} becomes
\begin{equation}\label{eqn:Ricci.condition.ingrid}
 \left[
 	(n-2)(S_{ija}\bar{t}^a+\bar{t}_i\bar{t}_j) - S\indices{_i^{ab}}S_{jab}
 \right]_\circ
 = 3\,\mathring R_{ij}\,.
\end{equation}

\noindent In the case of a constant curvature metric, the right hand side of 
this 
equation vanishes. In \cite{Kress&Schoebel&Vollmer} the following equation is 
then proven.
\begin{lemma}
	On a manifold of constant sectional curvature, 
	\eqref{eqn:Ricci.condition.ingrid} together with the non-linear 
	prolongation~\eqref{eqn:nonlinear.prolongation} implies
	\begin{equation}\label{eqn:perfect.square.equation}
		S^{abc}S_{abc} - (n-1)(n+2)\,\bar t^a \bar t_a = 9 R\,.
	\end{equation}
\end{lemma}

\noindent Using the condition found in the lemma, we obtain the following 
conditions for St\"ackel equivalent properly superintegrable systems.

\begin{theorem}\label{thm:eigenvalue.equation}
 If~\eqref{eqn:perfect.square.equation} holds, then Equation~\eqref{eqn:taa.bar} becomes
 \begin{equation}\label{eqn:eigenvalue.equation}
   \Delta \cs^{n+2} = -2\,\frac{n+1}{n-1}\,R\,\cs^{n+2}
 \end{equation}
 where $\cs=e^{-\frac13\bar{t}}$ as in Definition~\ref{def:scaling}.
\end{theorem}
\noindent Note that Equation~\eqref{eqn:eigenvalue.equation} 
is~\eqref{eqn:Laplace.scale}.
\begin{proof}
 If~\eqref{eqn:perfect.square.equation} holds, then $S^{abc}S_{abc}$ can be eliminated from Equation~\eqref{eqn:taa.bar}, yielding
 \begin{equation*}
   \Delta\bar{t} = \frac{6(n+1)}{(n-1)(n+2)}\,R
   +\frac{n+2}{3}\,\bar{t}^{,a}\bar{t}_{,a}\,.
 \end{equation*}
 In terms of $\cs=e^{-\frac13\,\bar t}$, this rewrites as 
 \eqref{eqn:eigenvalue.equation}.
\end{proof}

\noindent For constant curvature spaces, Equation~\eqref{eqn:taa.bar} thus 
becomes a Laplace eigenvalue problem, and a power of the scale function~$\cs$ 
is an eigenfunction of $\Delta$.
For a flat manifold, \eqref{eqn:eigenvalue.equation} merely implies that 
$\cs^{n+2}$ is harmonic.
On the round sphere $\mathbb{S}^n\subset\R^{n+1}$, we have spherical 
harmonics with the quantum number $\mu=n+1$ satisfying
\begin{equation}\label{eqn:quantum.number}
	\mu\,(\mu+n-1) := 2\,\frac{n+1}{n-1}\,Rr^2 = 2\,n(n+1)\,,
\end{equation}
where the second equality follows from $R=\frac{n(n-1)}{r^2}$ with $r>0$ 
denoting the radius of the sphere.

A close connection between the Helmholtz-Laplace equations and conformal 
superintegrability has been found 
in~\cite{KMS2016,Kalnins&Kress&Miller&Post11}.
Such links also appear in the present paper, although in 
different context: 
Earlier we have seen that on conformally flat spaces we find a scalar 
function $\theta^{1-\frac{n}{2}}$ satisfying the generalised Helmholtz 
equation \eqref{eqn:shift.trace}.
Now we have found~\eqref{eqn:laplace.D}, which is a conformally invariant 
generalised Helmholtz equation.
In particular, in the case of proper superintegrability, the $(n+2)$-nd power 
of the conformal scale function satisfies the generalised Helmholtz equation 
\eqref{eqn:eigenvalue.equation}.
It is a proper Helmholtz equation in the case of constant scalar curvature.

We now use Equation~\eqref{eqn:eigenvalue.equation} to study conformally 
equivalent properly superintegrable systems further.

\begin{proposition}\label{prop:flattenable}
 Assume we are provided with an abundant second order properly 
 superintegrable system on the sphere with the round metric $g$, which is 
 conformally equivalent to a properly superintegrable system on flat space 
 with the flat metric $h=\Omega^{-2}g$.
 Then the conformal factor $\cs$ on the sphere has to satisfy
 \begin{equation}\label{eqn:flattenable}
   \Omega\,\left(\Delta\Omega-g(d\ln(\cs^{n+2}),d\Omega)\right) + \,g(d\Omega,d\Omega) = 0\,.
 \end{equation}
\end{proposition}
\noindent Note that $\Omega$ is the conformal factor mediating between the 
standard and the spherical scale, while~$\cs$ mediates between the spherical 
and the flat scale.
\begin{proof}
 Due to~\eqref{eqn:eigenvalue.equation}, a properly superintegrable system on 
 flat space must satisfy the condition 
 $\Delta_\text{flat}(\Omega^{-(n+2)}\cs^{n+2})=0$.
 A direct computation using~\eqref{eqn:eigenvalue.equation} then shows
 \begin{equation*}
 	\Delta_\text{flat}(\Omega^{-(n+2)}\cs^{n+2})
 	= (3n+2)\,\cs^{n+2}\Omega^{-(n+6)}\,\bigg[
 	\Omega\Delta\Omega
 	+ \Omega^{,a}\Omega_{,a}
 	- \Omega\Omega^{,a}\,(\ln\cs^{n+2})_{,a}
 	\bigg]\,,
 \end{equation*}
 taking into account that
 \[
	 R = -2\frac{n-1}{\Omega^2}\,\left(
	 \Omega\Delta\Omega
	 -\frac{n}{2}\,\Omega^{,a}\Omega_{,a}
	 \right)\,.
 \]
 due to~\eqref{eqn:trafo.Ricci.curvature}.
 The desired condition now follows from~\eqref{eqn:eigenvalue.equation} 
 after a conformal transformation via~$\Omega$.
\end{proof}

The following example employs condition~\eqref{eqn:flattenable} to show that 
there is no conformal transformation that takes the generic system on the 
$n$-sphere to a properly superintegrable system on flat space. Note that 
there is always a conformal transformation taking it to a conformally 
superintegrable system on flat space.

The example thereby generalises a result shown in 
\cite{Kalnins&Kress&Miller-IV}, which addresses the specific case of 
dimension~$3$, see also \cite{Capel_phdthesis}.
Note that the proof presented here is a relatively simple exercise, while 
with traditional methods the claim, if at all, cannot be obtained for 
arbitrary dimension in a straightforward fashion.\footnote{%
	We remark that the generic system on the $n$-sphere can be transformed 
	into a proper superintegrable systems on flat space using B\^ocher 
	transformations or orbit degenerations~\cite{Capel_phdthesis}.
	Opposed to conformal transformations, however, these are not equivalence 
	relations on conformally superintegrable systems.}

\begin{example}[Generic system on the $n$-sphere]\label{ex:generic.n.sphere}
	Consider the generic system on the $n$-sphere, with $n\geqslant3$. It has 
	already been introduced for dimension~3 in 
	Example~\ref{ex:generic.3.sphere}. In arbitrary dimension we have the 
	metric
	\[
		g = \sum_{m=1}^{n}
				\left(\prod_{k=2}^{m}\sin^2(\phi_{k-1})\right)
				d\phi_m^2
	\]
	with angular coordinates $\phi_1,\dots,\phi_n$.
	The superintegrable potential defining the generic system is
	\[
		V	= a_0
			+ \sum_{m=1}^{n} \left(
					\frac{a_m}{	\cos^2(\phi_m)
								\prod_{k=2}^{m} \sin^2(\phi_{k-1}) }
			\right)
			+ \frac{a_{n+1}}{\prod_{k=1}^n \sin^2(\phi_k)}
	\]
	For this system, $\cs^{n+2}$ satisfies the Laplace 
	eigen-equation with quantum number $n+1$,
	\[
		\Delta\cs^{n+2} = -2n(n+1)\cs^{n+2}\,.
	\]
	Solutions of this equation span a vector space of dimension 
	$(n+2)^2$ whose basis is given by hyperspherical harmonics \cite{Fryant} 
	or one of the bases in \cite{LR_2003,LR_2004}.
	Concretely, for the generic system, we have
	\[
		\cs^{n+2} = \prod_{k=1}^n \cos(\phi_k)\sin^{n-k+1}(\phi_k)\,,
	\]
	which does \emph{not} satisfy \eqref{eqn:flattenable}.
\end{example}

Since the generic system on the $n$-sphere does not satisfy 
\eqref{eqn:flattenable}, we have proven the following.
\begin{theorem}\label{thm:generic.system}
	The generic system on the $n$-sphere is not conformally equivalent to a 
	properly superintegrable system on flat space.
\end{theorem}

The next example illustrates further how the presented framework can be 
invoked for proving statements in arbitrary dimension with ease and in a 
rigorous manner.
\begin{example}
	A non-degenerate properly superintegrable system on the $n$-sphere cannot 
	be conformally equivalent to the harmonic oscillator.
	Indeed, the harmonic oscillator has a vanishing structure tensor, 
	$T_{ijk}=0$. Therefore the system has to satisfy $S_{ijk}=0$, because 
	$S_{ijk}$ is conformally invariant. Moreover, being proper, the system on 
	the sphere satisfies~\eqref{eqn:tf.hessian.D},~\eqref{eqn:laplace.D} 
	and~\eqref{eqn:eigenvalue.equation}, i.e.
	\[
		\mathring\nabla^2_{ij}\cs = 0\,,\quad
		\Delta\cs^{1-\frac{n}{2}} = 0\,,\quad
		\Delta\cs^{n+2} = -2\,\frac{n+1}{n-1}\,R\cs^{n+2}\ne0\,,
	\]
	where $\mathring\nabla^2$ and $\Delta$ are the trace-free Hessian and the 
	Laplace-Beltrami operator on the sphere of constant scalar curvature 
	$R\ne0$.
	This system does not admit a solution.
\end{example}

\noindent We continue our study of pairs of conformally equivalent, 
\emph{properly} superintegrable systems on manifolds of constant curvature.
\begin{definition}
	We say that a c-superintegrable system is \emph{basic} if it 
	contains a member system that is an abundant properly superintegrable 
	system on a manifold of constant curvature.
\end{definition}

\noindent In reference \cite{Kress&Schoebel&Vollmer}, it is proven that the 
structure tensor of an abundant second order properly superintegrable system 
on a constant curvature manifold of dimension $n\geqslant3$ satisfies
\begin{equation}\label{eqn:B}
	T_{ijk}
	= \frac16\,{\young(ijk)}_\circ B_{,ijk}
	+ {\young(ij)}_\circ\,\frac{1}{n+2}\,g_{ik}\,\left(
	(\Delta B)_{,j} +\frac{2(n+1)}{n(n-1)}\,RB_{,j}
	\right)\,.
\end{equation}
where $B$ is a scalar function, called its \emph{structure function}.
The following theorem allows us to extend this definition of the 
structure function $B$ to any basic c-superintegrable system.

\begin{theorem}\label{thm:B.transformation}
 Consider two manifolds of constant curvature and with properly 
 superintegrable systems that are conformally equivalent.
 Denote their metrics by $g$ and $\tilde g=\Omega^2g$. 
 Assume the superintegrable systems have structure functions $B$ and $\tilde 
 B$, respectively, c.f.~\eqref{eqn:B}.
 Then
 \begin{equation}\label{eqn:B.hat.to.B}
  \tilde B = \Omega^2 B\quad\text{modulo gauge transformations}\,.
 \end{equation}
\end{theorem}
\begin{proof}
 On a manifold with Hamiltonian $H=g^{ij}p_ip_j+V$, we infer from 
 \cite{Kress&Schoebel&Vollmer} the following formula for $B_{ijk}$ in an 
 abundant constant-curvature system,
 \begin{equation*}
  B_{ijk} = T_{ijk} +\frac{n+2}{n}\,g_{ij} \bar t_{,k} + \frac{1}{2(n-2)}\,\young(ijk)\,g_{ij}C_{,k}
 \end{equation*}
 where
 \begin{equation}\label{eq:C}
  C = \frac{n-2}{n+2}\,\Delta B + \frac{2(n+1)}{n(n-1)}\,RB - (n-2)\bar t
 \end{equation}
 and
 \begin{equation*}
	 B_{ijk} = \frac16\young(ijk)\left(
	 							B_{,ij}+\frac{4R}{n(n-1)}g_{ij}B
	 						\right)_{,k}
 \end{equation*}
 up to an irrelevant constant.
 By virtue of \eqref{eqn:transformation.rules}, we know the transformation behavior of $S_{ijk}$ and $\bar t_i=\bar t_{,i}$,
 \[
  S_{ijk} \mapsto \tilde{S}_{ijk}=\Omega^2S_{ijk}\qquad\text{and}\qquad
  \bar t_{,i} \mapsto \bar t_{,i} -3\Omega^{-1}\Omega_{,i}\,,
 \]
 and therefore that of $B_{ijk}$,
 \begin{equation}\label{eqn:til.B}
  \tilde B_{ijk} = \Omega^2 \left( B_{ijk} + \text{ trace terms} \right)\,.
 \end{equation}
 Secondly, we also know by construction that $S_{ijk}={\young(ijk)}_\circ B_{,ijk}$, where on the right hand side we recall that comma denotes the covariant derivative. An analogous equation holds for $\tilde S_{ijk}$ with a function $\tilde B$.
 Now, let us denote by $\nabla$ and $\tilde\nabla$ the Levi-Civita connections of $g$ and $\tilde g$, respectively.
 Then, for the third derivatives,
 \begin{equation}\label{eq:trafo.B.function}
  {\young(ijk)}_\circ \tilde\nabla^3_{ijk}\tilde B = \Omega^2\,{\young(ijk)}_\circ \nabla^3_{ijk}B
 \end{equation}
 because of the invariance of $S_{ijk}$.
 A straightforward computation verifies that $\tilde B=\Omega^2B$ satisfies 
 \eqref{eq:trafo.B.function}.
 We have therefore confirmed that $\bar{\tilde{t}}=\bar t-3\ln|\Omega|$ and 
 $\tilde B=\Omega^2B$ yield the correct structure tensors $\tilde S_{ijk}$ 
 and~$\bar{\tilde{t}}_{,i}$ for the conformally transformed manifold with 
 Hamiltonian $\tilde H=\Omega^{-2}H$.
 Since the structure functions are unique up to gauge transformations, this 
 concludes the proof.
\end{proof}


\noindent The theorem allows us to extend the definition of the structure 
function $B$ as follows: In \cite{Kress&Schoebel&Vollmer} the structure 
function $B$ is defined for abundant properly superintegrable systems on 
constant curvature spaces. We are now able to define a corresponding object 
for any c-superintegrable system arising from such systems.
\begin{corollary}\label{cor:b.density}
 Abundant second order properly superintegrable systems on constant curvature 
 spaces in dimension $n\geqslant3$ are St\"ackel equivalent if and only if 
 their densities $\mathbf{b}\in\mathcal{E}[-2]$ given by
 \begin{equation*}
 	\mathbf{b} = B\,\det(g)^{\frac1n}\,.
 \end{equation*}
 coincide up to a gauge transformation.
\end{corollary}


\begin{example}\label{ex:III.V}
 It is well understood that on 3-dimensional flat space the systems III and V are equivalent. There is no St\"ackel equivalent system of this class on the 3-sphere.
 In common coordinates $(x,z,\bar z)$, we may write
 \begin{align*}
 	g_\text{III} &= \frac{dx^2+dzd\bar z}{z^2}
 	&
 	V_\text{III} &= \omega\,z^2(4x^2+z\bar z) + a_1\,xz^2 + a_2\,z^2 + \frac{a_3\bar z}{z} +a_4
 	\\
 	g_\text{V} &= dx^2+dzd\bar z
 	&
 	V_\text{V} &= \omega\,(4x^2+z\bar z) + a_1x + \frac{a_4}{z^2} + \frac{a_3\bar z}{z^3} + a_2
 \end{align*}
 for which we find
 \[
   B_\text{III} = -\frac32\,\frac{\ln(z)\bar z}{z}\,,\qquad
   B_\text{V} = -\frac32\,\ln(z)z\bar z\,.
 \]
 The density $\mathbf{b}$ shared by these structure functions is
 \[
  \mathbf{b} = -3\ln(z)z\bar z\,.
 \]
 There is no properly superintegrable system conformally equivalent to those 
 obtained from $g_\text{III}$ and $V_\text{III}$ (or equivalently 
 $g_\text{V}$ and $V_\text{V}$). In order to confirm this, note that $Vg$ is 
 invariant, since $V$ transforms with $\Omega^{-2}$ and $g$ transforms with 
 $\Omega^2$. Indeed, the metric
 \[
   \hat g=V_\text{III}g_\text{III}=V_\text{V}g_\text{V}
 \]
 cannot have constant scalar curvature $1$ for any constants $a_1,\dots,a_4$ 
 and $\omega$.
\end{example}

\section{Application to dimension three}\label{sec:3D}

\setlength{\rotFPtop}{0pt plus 1fil}
\begin{sidewaystable}[ht]
	\centering
	\textbf{Non-degenerate second order superintegrable systems in 
	dimension~$n=3$, for Euclidean geometry $g=\sum_i dx_i^2$}\\[.5cm]
	\begin{tabular}{p{4.2cm}|p{7.1cm}|p{6.4cm}|p{2.8cm}}
		\toprule
		\textbf{Example}
		& \textbf{potential mod constant}
		& \textbf{$\mathbf B$ mod gauge terms}
		& $\mathbf{-\frac{5}{3}\bar{t}}$ \textbf{mod const.} \\
		\midrule
		\multicolumn{3}{l}{\emph{Regular systems} (linked with elliptic 
		separation coordinates)} \\
		\midrule
		``Generic system'' I [2111]/\newline Smorodinsky-Winternitz I
		& $\sum_{i=1}^n \left( \frac{a_i}{x_i^2}+\omega\,x_i^2 \right)$
		& $-\frac{3}{2}\,(x^2\ln(x) + y^2\ln(y) + z^2\ln(z))$
		& $\sum_k \ln(x_k)$ \\
		\midrule
		System II [221]
		& $\omega(x^2+y^2+z^2) + a_1\,\frac{x-iy}{(x+iy)^3}
		+ a_2\,\frac{1}{(x+iy)^2} + \frac{a_3}{z^2}$
		& $(x^2+y^2)(\frac12\,\ln(x^2 + y^2)- 
		i\,\arctan(\frac{x}{y}))$\newline
		$ + z^2\ln(z)$
		& $\ln(z) + \ln(x^2 + y^2)$\newline $-2i\arctan(\frac{x}{y})$ \\
		\midrule
		System III [23]
		& $\omega(x^2 + y^2 + z^2)$\newline
		$+ \frac{a_1}{(x+iy)^2}
		+ \frac{a_2z}{(x+iy)^3} + a_3\,\frac{x^2 + y^2 - 3z^2}{(x +iy)^4}$
		& $(x^2 +y^2 +z^2)(\frac12\,\ln(x^2 + y^2)-i\arctan(\frac{x}{y}))$
		& $\frac32\,\ln(x^2 + y^2)$\newline $- 3i\,\arctan(\frac{x}{y})$ \\
		System V [32]
		& $\omega(4x^2 + y^2 + z^2)$\newline
		$+ a_1x + \frac{a_2}{(y+iz)^2}
		+ a_3\,\frac{y-iz}{(y+iz)^3}$
		& $(y^2+z^2)(\frac12\,\ln(y^2 + z^2)-i\,\arctan(\tfrac{y}{z}))$
		& $\ln(y^2 + z^2)$\newline $- 6i\arctan(\frac{y}{z})$ \\
		\midrule
		System IV [311]/\newline Smorodinsky-Winternitz II
		& $\omega(4x^2 + y^2 + z^2) + a_1x + \frac{a_2}{y^2} + 
		\frac{a_3}{z^2}$
		& $y^2\ln(y) + z^2\ln(z)$
		& $\ln(y) + \ln(z)$ \\
		\midrule
		System VI [41]
		& $\omega(z^2 - 2(x-iy)^3 + 4x^2 + 4y^2)$\newline
		$+ a_1(2x + 2iy - 3(x-iy)^2) + a_2(x-iy) + \frac{a_3}{z^2}$
		& $z^2\ln(z) + \frac16\,(x-iy)^3$
		& $\ln(z)$ \\
		\midrule
		System VII [5]
		& $\omega\,(x+iy) + a_1\,(\frac34(x+iy)^2 + \frac{z}4)$\newline\quad
		$+ a_2\,((x+iy)^3 + \frac{1}{16}\,(x-iy) + \frac34\,(x+iy)z)$
		\newline\quad
		$+ \frac{a_3}{16}\,(5(x+iy)^4 + x^2 + y^2 + z^2 + 6(x+iy)^2z)$
		& $-\frac13\,((x+iy)^2+6z)\,(x+iy)^2$
		& $0$ \\
		\midrule
		\multicolumn{4}{l}{\emph{Exceptional systems} (linked with degenerate 
		separation coordinates)} \\
		\midrule
		Isotropic oscillator O
		& $\omega^2\,\sum_ix_i^2 +\sum_i \alpha_ix_i$
		& 0
		& 0 \\
		\midrule
		System OO
		& $\frac{\omega}{2}\,\left(x^2 + y^2 + \frac{z^2}{4}\right)
		+ a_1x + a_2y + \frac{a_3}{z^2}$
		& $z^2\ln(z)$
		& $\ln(z)$ \\
		\midrule
		System A
		& $\omega\,((x-iy)^3 + 6(x^2+y^2+z^2))$\newline
		$+ a_1\,((x-iy)^2 + 2(x+iy)) + a_2\,(x-iy) + a_3z$
		& $-\frac{1}{18}\,(x-iy)^3$
		& $0$ \\
		\bottomrule
	\end{tabular}
	\caption{The properly superintegrable systems in dimension 3 on flat 
	space. The Systems III and V are conformally equivalent under St\"ackel 
	transform, see Example~\ref{ex:III.V} for details.}
	\label{tab:3D}
\end{sidewaystable}

We have already discussed a few examples in the previous section, along with their behavior under conformal transformations. In the present section, we apply our framework to the 3-dimensional case.
Non-degenerate second order conformally superintegrable systems in 
dimension~3 are classified in~\cite{Kalnins&Kress&Miller-IV,Capel_phdthesis}. 
Also, it is known that all these systems are abundant 
\cite{Kalnins&Kress&Miller-III}.
In \cite{Kalnins&Kress&Miller-IV} it has been established that any 
non-degenerate second order conformally superintegrable system is St\"ackel 
equivalent to a non-degenerate second order abundant and proper system on a 
constant curvature geometry.
We shall therefore restrict to the study of abundant systems for constant 
curvature metrics.
All non-degenerate $3$-dimensional systems are equivalent to these.
Recall that in dimension~3, the non-linear 
condition~\eqref{eqn:Weyl.condition.abundant} is void. Consequently, no 
further restriction exists on the tensor $S_{ijk}$.
Hence any trace-free symmetric initial conditions $\Psi_{ijk}=S_{ijk}(x_0)$ 
in a point $x_0\in M$ can be integrated to a structure tensor of an abundant 
second order conformally superintegrable system, c.f.\ 
Corollary~\ref{cor:no.Weyl.3D}.  
Therefore the set of conformal equivalence classes of such systems is 
parametrised by the seven dimensional space of trace-free symmetric 3-tensors 
$\Psi_{ijk}$ or, equivalently, harmonic ternary cubics 
$\Psi(\mathbf p)=\Psi_{ijk}p^ip^jp^k$.
This parametrisation is equivariant with respect to the stabiliser subgroup 
of the point $x_0$ in the conformal group, which is isomorphic to 
$\mathrm{SO}(3)$.

We comment that this agrees with the
references~\cite{Kalnins&Kress&Miller&Post11,Capel&Kress,Capel_phdthesis}.  In
\cite{Kalnins&Kress&Miller&Post11} a $10$-parameter classification
space is mentioned, corresponding to a $10$-dimensional representation of
$\mathrm{SO}(3)$ in \cite{Capel_phdthesis,Capel&Kress}. This $10$-dimensional
representation decomposes into two irreducible components of dimension 7 and
3, corresponding to $S_{ijk}$ and $\bar t_{,i}$ in our framework.  Note that
the $3$-dimensional component is restricted in the references, which 
corresponds to imposing proper superintegrability here.
The 7-dimensional component is realised as the space of binary sextics and it 
is shown that no restrictions exist on this component.
The relation to our framework is given by a known correspondence between
harmonic ternary cubics and binary sextics as follows.

The adjoint action of $\mathrm{SL}(2,\mathbb C)$ on its Lie algebra
$\mathrm{sl}(2,\mathbb C)\cong\mathbb C^3$ preserves the Killing form.  This
defines a group morphism $\mathrm{SL}(2,\mathbb C)\to\mathrm{SO}(3,\mathbb C)$
with kernel $\{\pm1\}$ and hence an isomorphism $\mathrm{SL}(2,\mathbb
C)/\mathbb Z_2\to\mathrm{SO}(3,\mathbb C)$.  The standard action of
$\mathrm{SL}(2,\mathbb C)$ on $\mathbb C^2$ induces an $\mathrm{SL}(2,\mathbb
C)$-action on $S^2\mathbb C^2$ which descends to an $\mathrm{SO}(3,\mathbb
C)$-action, because the elements $\pm1$ act trivially.  The latter induces an
$\mathrm{SO}(3,\mathbb C)$-action on $S^3S^2\mathbb C^2$ which descends to
$S^6\mathbb C^2$ under total symmetrisation $S^3S^2\mathbb C^2\to S^6\mathbb
C^2$.  Together with the isomorphism $S^2\mathbb C^2\cong\mathbb C^3$, we
obtain a morphism $S^3\mathbb C^3\to S^6\mathbb C^2$ of $\mathrm{SO}(3,\mathbb
C)$-representations, giving an $\mathrm{SO}(3,\mathbb C)$-equivariant morphism
from the 10-dimensional space of ternary cubics to the 7-dimensional space of
binary sextics.  Its restriction to the 7-dimensional space of harmonic
ternary cubics is non-trivial and hence an isomorphism by Schur's lemma.
Explicitly, it is given by defining a sextic $s$ from the cubic $\Psi(\mathbf 
p)$ via
\begin{equation}\label{eqn:binary.sextic}
	s(z,w)=\Psi(z^2-w^2,2zw,z^2+w^2).
\end{equation}
Note that the stabiliser subgroup contains only rotations.
The action of translations is not linear and more involved.

\subsection{Special case: Simultaneous standard scale and proper scale}
Let us confine to constant curvature geometries, since in dimension~3 any 
non-degenerate system is conformally equivalent to one on constant curvature 
\cite[Theorem 4]{Kalnins&Kress&Miller-IV}.
We begin with the very particular situation where the standard scale choice 
is 
a properly superintegrable system, i.e.\ we have both $\bar t=0$ and 
$\tau_{ij}=0$.
In this case, due to~\eqref{eqn:Ricci.condition.strange.eqn} 
and~\eqref{eqn:tij.bar}, the tracefree Ricci tensor satisfies
\begin{equation}\label{eqn:Ric.standard.proper.scale}
  \mathring R_{ij} = (S\indices{_i^{ab}}S_{jab})_\circ = 0\,,
\end{equation}
A comparison with \cite{Kalnins&Kress&Miller-III} shows that there are only 
three normal forms on $3$-dimensional Euclidean space that satisfy this 
criterion, c.f.\ also Table~\ref{tab:3D}.
These are the systems labeled VII, A and O in \cite{Kalnins&Kress&Miller-III}.
%

\subsection{General picture}
Table \ref{tab:3D} lists the established normal forms for 3-dimensional non-degenerate systems on flat space, see \cite{Kalnins&Kress&Miller-IV,Capel_phdthesis}.
The functions $B$ and $\bar t$ are obtained as established in 
\cite{Kress&Schoebel&Vollmer}, and using 
Theorem~\ref{thm:B.transformation}, as well as Equation \eqref{eqn:trafo.t}, 
we may compute the corresponding functions for any 3-dimensional 
non-degenerate system conformally equivalent to one of the systems in 
Table~\ref{tab:3D}.

If in Table \ref{tab:3D}, we take the quotient under conformal equivalence 
for each example, then the systems III and V are identified and we obtain a 
list of nine abundant c-superintegrable systems. Due to 
\cite{Capel_phdthesis} these are all abundant c-superintegrable systems in 
dimension~3, up to one exception.
Indeed, from 
\cite{Kalnins&Kress&Miller-III,Kalnins&Kress&Miller-IV,Capel_phdthesis} it 
follows that there is one equivalence class of non-degenerate 3-dimensional 
superintegrable system that does not admit a representative \emph{properly} 
superintegrable system on flat 3-space.
This system is the generic system on the 3-sphere from 
Example~\ref{ex:generic.3.sphere}, see also Theorem~\ref{thm:generic.system}.
Its conformally equivariant structure tensor $S_{ijk}$ is generated by the 
structure function (up to gauge freedom)
\begin{equation*}
	B = -\frac32\,\sum_k s_k^2\,\ln(s_k)\,.
\end{equation*}

%

\bibliographystyle{amsalpha}
\bibliography{conformal_superintegrability}
\medskip

\end{document}